\documentclass[11pt]{preprint}
\usepackage[full]{textcomp}
\usepackage[osf]{newtxtext} 
\usepackage[cal=boondoxo]{mathalfa}
\usepackage{colortbl}
\usepackage{comment}

\usepackage{amssymb}
\usepackage{mathtools}

\usepackage{hyperref}
\usepackage{breakurl}
\usepackage{mhenvs}
\usepackage{mhequ}
\usepackage{mhsymb}
\usepackage{booktabs}
\usepackage{tikz}
\usepackage{tcolorbox}
\usepackage{mathrsfs}
\usepackage[utf8]{inputenc}
\usepackage{longtable}
\usepackage{wrapfig}

\usepackage{microtype}
\usepackage{wasysym}
\usepackage{centernot}
\usepackage{enumitem}
\usepackage{bm}
\usepackage{stackrel}
\usepackage{graphicx}

\makeatletter
\newcommand{\globalcolor}[1]{%
  \color{#1}\global\let\default@color\current@color
}
\makeatother

\usetikzlibrary{calc}
\usetikzlibrary{decorations}
\usetikzlibrary{positioning}
\usetikzlibrary{shapes}
\usetikzlibrary{external}

\newif\ifdark
\darkfalse

\ifdark

\definecolor{darkred}{rgb}{0.9,0.2,0.2}
\definecolor{darkblue}{rgb}{0.7,0.3,1}
\definecolor{darkgreen}{rgb}{0.1,0.9,0.1}
\definecolor{franck}{rgb}{0,0.8,1}
\definecolor{pagebackground}{rgb}{.15,.21,.18}
\definecolor{pageforeground}{rgb}{.84,.84,.85}
\pagecolor{pagebackground}
\AtBeginDocument{\globalcolor{pageforeground}}
\definecolor{symbols}{rgb}{0,0.7,1}
\colorlet{connection}{red!80!black}
\colorlet{boxcolor}{blue!50}

\else

\definecolor{darkred}{rgb}{0.7,0.1,0.1}
\definecolor{darkblue}{rgb}{0.4,0.1,0.8}
\definecolor{darkgreen}{rgb}{0.1,0.7,0.1}
\definecolor{franck}{rgb}{0,0,1}
\definecolor{pagebackground}{rgb}{1,1,1}
\definecolor{pageforeground}{rgb}{0,0,0}
\colorlet{symbols}{blue!90!black}
\colorlet{connection}{red!30!black}
\colorlet{boxcolor}{blue!50!black}

\fi

\def\slash{\leavevmode\unskip\kern0.18em/\penalty\exhyphenpenalty\kern0.18em}
\def\dash{\leavevmode\unskip\kern0.18em--\penalty\exhyphenpenalty\kern0.18em}

\DeclareMathAlphabet{\mathbbm}{U}{bbm}{m}{n}

\DeclareFontFamily{U}{BOONDOX-calo}{\skewchar\font=45 }
\DeclareFontShape{U}{BOONDOX-calo}{m}{n}{
  <-> s*[1.05] BOONDOX-r-calo}{}
\DeclareFontShape{U}{BOONDOX-calo}{b}{n}{
  <-> s*[1.05] BOONDOX-b-calo}{}
\DeclareMathAlphabet{\mcb}{U}{BOONDOX-calo}{m}{n}
\SetMathAlphabet{\mcb}{bold}{U}{BOONDOX-calo}{b}{n}

\setlist{noitemsep,topsep=4pt,leftmargin=1.5em}

\DeclareMathAlphabet{\mathbbm}{U}{bbm}{m}{n}

\DeclareMathAlphabet{\mcb}{U}{BOONDOX-calo}{m}{n}
\SetMathAlphabet{\mcb}{bold}{U}{BOONDOX-calo}{b}{n}
\DeclareFontFamily{U}{mathx}{\hyphenchar\font45}
\DeclareFontShape{U}{mathx}{m}{n}{
      <5> <6> <7> <8> <9> <10>
      <10.95> <12> <14.4> <17.28> <20.74> <24.88>
      mathx10
      }{}
\DeclareSymbolFont{mathx}{U}{mathx}{m}{n}
\DeclareMathSymbol{\bigtimes}{1}{mathx}{"91}

\providecommand{\figures}{false}
{ \ifthenelse{\equal{\figures}{false}} {#1}{\[ {\rm Figure \ missing !} \]} }{}
\def\id{\mathrm{id}}

\def\CC{\mathcal{C}}

\tikzstyle{tinydots}=[dash pattern=on \pgflinewidth off \pgflinewidth]
\tikzstyle{superdense}=[dash pattern=on 4pt off 1pt]

\newcommand{\mbn}{\mathbf{n}}

\def\eps{\varepsilon}

\def\${|\!|\!|}

\newenvironment{DIFnomarkup}{}{} 

\theorembodyfont{\rmfamily}
\newtheorem{example}[lemma]{Example}

\newfont{\indic}{bbmss12}

\def\PPi{\boldsymbol{\Pi}}

\def\Nabla_#1{\nabla_{\!#1}}

\makeatletter
\pgfdeclareshape{crosscircle}
{
  \inheritsavedanchors[from=circle] 
  \inheritanchorborder[from=circle]
  \inheritanchor[from=circle]{north}
  \inheritanchor[from=circle]{north west}
  \inheritanchor[from=circle]{north east}
  \inheritanchor[from=circle]{center}
  \inheritanchor[from=circle]{west}
  \inheritanchor[from=circle]{east}
  \inheritanchor[from=circle]{mid}
  \inheritanchor[from=circle]{mid west}
  \inheritanchor[from=circle]{mid east}
  \inheritanchor[from=circle]{base}
  \inheritanchor[from=circle]{base west}
  \inheritanchor[from=circle]{base east}
  \inheritanchor[from=circle]{south}
  \inheritanchor[from=circle]{south west}
  \inheritanchor[from=circle]{south east}
  \inheritbackgroundpath[from=circle]
  \foregroundpath{
    \centerpoint%
    \pgf@xc=\pgf@x%
    \pgf@yc=\pgf@y%
    \pgfutil@tempdima=\radius%
    \pgfmathsetlength{\pgf@xb}{\pgfkeysvalueof{/pgf/outer xsep}}%
    \pgfmathsetlength{\pgf@yb}{\pgfkeysvalueof{/pgf/outer ysep}}%
    \ifdim\pgf@xb<\pgf@yb%
      \advance\pgfutil@tempdima by-\pgf@yb%
    \else%
      \advance\pgfutil@tempdima by-\pgf@xb%
    \fi%
    \pgfpathmoveto{\pgfpointadd{\pgfqpoint{\pgf@xc}{\pgf@yc}}{\pgfqpoint{-0.707107\pgfutil@tempdima}{-0.707107\pgfutil@tempdima}}}
    \pgfpathlineto{\pgfpointadd{\pgfqpoint{\pgf@xc}{\pgf@yc}}{\pgfqpoint{0.707107\pgfutil@tempdima}{0.707107\pgfutil@tempdima}}}
    \pgfpathmoveto{\pgfpointadd{\pgfqpoint{\pgf@xc}{\pgf@yc}}{\pgfqpoint{-0.707107\pgfutil@tempdima}{0.707107\pgfutil@tempdima}}}
    \pgfpathlineto{\pgfpointadd{\pgfqpoint{\pgf@xc}{\pgf@yc}}{\pgfqpoint{0.707107\pgfutil@tempdima}{-0.707107\pgfutil@tempdima}}}
  }
}
\makeatother

\def\symbol#1{\textcolor{symbols}{#1}}

\def\decorate#1#2{
        \ifnum#2>0
    		\foreach \count in {1,...,#2}{
	       	let
				\p1 = (sourcenode.center),
                \p2 = (sourcenode.east),
				\n1 = {\x2-\x1},
				\n2 = {1mm},
				\n3 = {(1.3+0.6*(\count-1))*\n1},
				\n4 = {0.7*\n1}
			in 
        		node[rectangle,fill=symbols,rotate=30,inner sep=0pt,minimum width=0.2*\n2,minimum height=\n2] at ($(sourcenode.center) + (\n3,\n4)$) {}
				}
		\fi
        \ifnum#1>0
    		\foreach \count in {1,...,#1}{
	       	let
				\p1 = (sourcenode.center),
                \p2 = (sourcenode.east),
				\n1 = {\x2-\x1},
				\n2 = {1mm},
				\n3 = {(1.3+0.6*(\count-1))*\n1},
				\n4 = {0.7*\n1}
			in 
        		node[rectangle,fill=symbols,rotate=-30,inner sep=0pt,minimum width=0.2*\n2,minimum height=\n2] at ($(sourcenode.center) + (-\n3,\n4)$) {}
				}
		\fi
}

\tikzset{
    dectriangle/.style 2 args={
        triangle,
        alias=sourcenode,
        append after command={\decorate{#1}{#2}}
    },
    dectriangle/.default={0}{0},
}

\tikzset{
	cross/.style={path picture={ 
  		\draw[symbols]
			(path picture bounding box.south east) -- (path picture bounding box.north west) (path picture bounding box.south west) -- (path picture bounding box.north east);
		}},
root/.style={circle,fill=green!50!black,inner sep=0pt, minimum size=1.2mm},
        dot/.style={circle,fill=pageforeground,inner sep=0pt, minimum size=1mm},
        blank/.style={circle,fill=white,inner sep=0pt, minimum size=1mm},
        dotred/.style={circle,fill=pageforeground!50!pagebackground,inner sep=0pt, minimum size=2mm},
        var/.style={circle,fill=pageforeground!10!pagebackground,draw=pageforeground,inner sep=0pt, minimum size=3mm},
        sqvar/.style={rectangle,fill=pageforeground!10!pagebackground,draw=pageforeground,inner sep=0pt, minimum size=3mm},
        kernel/.style={semithick,shorten >=2pt,shorten <=2pt},
        kernels/.style={snake=zigzag,shorten >=2pt,shorten <=2pt,segment amplitude=1pt,segment length=4pt,line before snake=2pt,line after snake=5pt,},
        rho/.style={densely dashed,semithick,shorten >=2pt,shorten <=2pt},
           testfcn/.style={dotted,semithick,shorten >=2pt,shorten <=2pt},
        renorm/.style={shape=circle,fill=pagebackground,inner sep=1pt},
        labl/.style={shape=rectangle,fill=pagebackground,inner sep=1pt},
        xic/.style={very thin,circle,draw=symbols,fill=symbols,inner sep=0pt,minimum size=1.2mm},
        g/.style={very thin,rectangle,draw=symbols,fill=symbols!10!pagebackground,inner sep=0pt,minimum width=2.5mm,minimum height=1.2mm},
        xi/.style={very thin,circle,draw=symbols,fill=symbols!10!pagebackground,inner sep=0pt,minimum size=1.2mm},
	xies/.style={very thin,rectangle,fill=green!50!black!25,draw=symbols,inner sep=0pt,minimum size=1.1mm},
	xiesf/.style={very thin,rectangle,fill=green!50!black,draw=symbols,inner sep=0pt,minimum size=1.1mm},
        xix/.style={very thin,crosscircle,fill=symbols!10!pagebackground,draw=symbols,inner sep=0pt,minimum size=1.2mm},
        X/.style={very thin,cross,rectangle,fill=pagebackground,draw=symbols,inner sep=0pt,minimum size=1.2mm},
	xib/.style={thin,circle,fill=symbols!10!pagebackground,draw=symbols,inner sep=0pt,minimum size=1.6mm},
	xie/.style={thin,circle,fill=green!50!black,draw=symbols,inner sep=0pt,minimum size=1.6mm},
	xid/.style={thin,circle,fill=symbols,draw=symbols,inner sep=0pt,minimum size=1.6mm},
	xibx/.style={thin,crosscircle,fill=symbols!10!pagebackground,draw=symbols,inner sep=0pt,minimum size=1.6mm},
	kernels2/.style={very thick,draw=connection,segment length=12pt},
	keps/.style={thin,draw=symbols,->},
	kepspr/.style={thick,draw=connection,->},
	krho/.style={thin,draw=symbols,superdense,->},
	krhopr/.style={thick,draw=connection,superdense,->},
	triangle/.style = { regular polygon, regular polygon sides=3},
	not/.style={thin,circle,draw=connection,fill=connection,inner sep=0pt,minimum size=0.5mm},
	diff/.style = {very thin,draw=symbols,triangle,fill=red!50!black,inner sep=0pt,minimum size=1.6mm},
	diff1/.style = {very thin,dectriangle={1}{0},fill=red!50!black,draw=symbols,inner sep=0pt,minimum size=1.6mm},
	diff2/.style = {very thin,dectriangle={1}{1},fill=red!50!black,draw=symbols,inner sep=0pt,minimum size=1.6mm},
		diffmini/.style = {very thin,rectangle,fill=black,draw=black,inner sep=0pt,minimum size=0.75mm},
	 kernelsmod/.style={very thick,draw=connection,segment length=12pt},
	 rec/.style = {very thin,rectangle,fill=black,draw=black,inner sep=0pt,minimum size=2mm},
	cerc/.style={very thin,circle,draw=black,fill=symbols,inner sep=0pt,minimum size=2mm},
	stars/.style={very thin,star,star points=6,star point ratio=0.5, draw=black,fill=red,inner sep=0pt,minimum size=0.7mm},
	>=stealth,
        }

\makeatletter
\def\DeclareSymbol#1#2#3{%
	\expandafter\gdef\csname MH@symb@#1\endcsname{\tikzsetnextfilename{symbol#1}%
	\tikz[baseline=#2,scale=0.15,draw=symbols,line join=round]{#3}}%
	\expandafter\gdef\csname MH@symb@#1s\endcsname{\scalebox{0.75}{\tikzsetnextfilename{symbol#1}%
	\tikz[baseline=#2,scale=0.15,draw=symbols,line join=round]{#3}}}%
	\expandafter\gdef\csname MH@symb@#1ss\endcsname{\scalebox{0.65}{\tikzsetnextfilename{symbol#1}%
	\tikz[baseline=#2,scale=0.15,draw=symbols,line join=round]{#3}}}%
	}
\def\<#1>{\ifthenelse{\boolean{mmode}}{\mathchoice{\csname MH@symb@#1\endcsname}{\csname MH@symb@#1\endcsname}{\csname MH@symb@#1s\endcsname}{\csname MH@symb@#1ss\endcsname}}{\csname MH@symb@#1\endcsname}}
\makeatother

\DeclareSymbol{Xi22}{0.5}{\draw (0,0) node[xi] {} -- (-1,1) node[not] {} -- (0,2) node[xi] {};} 
\DeclareSymbol{Xi2}{-2}{\draw (-1,-0.25) node[xi] {} -- (0,1) node[xi] {};} 
\DeclareSymbol{Xi2C}{-2}{\draw (-1,-0.25) node[xi] {} -- (0,1) node[xi] {}; \node at (-1,1) (a) {\scriptsize 1};}
\DeclareSymbol{Xi2alpha}{-2}{\draw (-1,-0.25) node[xi] {} -- (0,1) node[xi] {}; \node at (-1.2,1) (a) {\scriptsize $\alpha$};}
\DeclareSymbol{I(Xi)}{0}{\node[xi] at (0,1) (a) {}; \draw (a) -- (0,0);}
\DeclareSymbol{pXi2}{-2}{\draw (-1,-0.25) node[xi] {} -- (0,1) node[xi] {}; \node at (-0.95,0.95) {\tiny $1$};}
\DeclareSymbol{puXi2}{-2}{\draw (-1,-0.25) node[xi] {} -- (0,1) node[xi] {}; \node at (-1.25,0.95) {\tiny $m$};}
\DeclareSymbol{PlantedNoise}{0}{\draw[black] (0,1.5) node[xi] {} -- (0,-0.25) {};}
\DeclareSymbol{PlantedNoiseC}{0}{\draw[black] (0,1.5) node[xi] {} -- (0,-0.25) {}; \node at (-0.75,0.5) (a) {\scriptsize 1};}
\DeclareSymbol{PlantedNoiseCC}{0}{\draw[black] (0,1.5) node[xi] {} -- (0,-0.25) {}; \node at (-0.75,0.5) (a) {\scriptsize 2};}
\DeclareSymbol{PlantedNoiseCX}{0}{\draw[kernels2] (0,1.5) node[xi] {} -- (0,-0.25) {}; \node at (-0.75,0.5) (a) {\scriptsize 1};}
\DeclareSymbol{PlantedNoiseX}{0}{\draw[kernels2] (0,1.5) node[xi] {} -- (0,-0.25) {};}
\DeclareSymbol{PlantedNoisea}{0}{\draw (0,1.5) node[xi] {} -- (0,-0.25) {}; \node at (-0.75,0.5) (a) {\scriptsize $\alpha$};}
\DeclareSymbol{XiIt}{2}{\node[xi] at (-1.5,-0.1) (a) {};
\node[var] at (0,2.25) (b) {\tiny{$\tau$ }};
\node[blank] at (-1.5,1.25) {\tiny{$k$}};
\draw[symbols]  (a) -- (b);}
\DeclareSymbol{Xi1It}{2}{\node[xi] at (-1.5,-0.1) (a) {};
\node[var] at (0,2.25) (b) {\tiny{$\tau$ }};
\node[blank] at (-1.5,1.25) {\tiny{$1$}};
\draw[symbols]  (a) -- (b);}
\DeclareSymbol{Xi0It}{2}{\node[xi] at (-1.5,-0.1) (a) {};
\node[var] at (0,2.25) (b) {\tiny{$\tau$ }};
\node[blank] at (-1.5,1.25) {};
\draw[symbols]  (a) -- (b);}
\DeclareSymbol{It}{2}{\node at (0,-1) (a) {};
\node[var] at (0,2.25) (b) {\tiny{$\tau$ }};
\node[blank] at (-0.75,0.5) {\tiny{$k$}};
\draw[symbols]  (a) -- (b);}
\DeclareSymbol{It0}{2}{\node at (0,-1) (a) {};
\node[var] at (0,2.25) (b) {\tiny{$\tau$ }};
\draw[symbols]  (a) -- (b);}
\DeclareSymbol{1It}{2}{\node at (0,-1) (a) {};
\node[var] at (0,2.25) (b) {\tiny{$\tau$ }};
\node[blank] at (-0.75,0.5) {\tiny{$1$}};
\draw[symbols]  (a) -- (b);}
\DeclareSymbol{1I1t1}{2}{\node at (0,-1) (a) {};
\node[var] at (0,2.25) (b) {\tiny{$\tau_1$ }};
\node[blank] at (-0.75,0.5) {\tiny{$1$}};
\draw[kernels2]  (a) -- (b);}
\DeclareSymbol{1I1t2}{2}{\node at (0,-1) (a) {};
\node[var] at (0,2.25) (b) {\tiny{$\tau_2$ }};
\node[blank] at (-0.75,0.5) {\tiny{$1$}};
\draw[kernels2]  (a) -- (b);}
\DeclareSymbol{1I1t}{2}{\node at (0,-1) (a) {};
\node[var] at (0,2.25) (b) {\tiny{$\tau$ }};
\node[blank] at (-0.75,0.5) {\tiny{$1$}};
\draw[kernels2]  (a) -- (b);}
\DeclareSymbol{I1t}{2}{\node at (0,-1) (a) {};
\node[var] at (0,2.25) (b) {\tiny{$\tau$ }};
\node[blank] at (-0.75,0.5) {\tiny{$1$}};
\draw[kernels]  (a) -- (b);}
\DeclareSymbol{I1t1}{2}{\node at (0,-1) (a) {};
\node[var] at (0,2.25) (b) {\tiny{$\tau_1$}};
\node[blank] at (-0.75,0.5) {};
\draw[kernels2]  (a) -- (b);}
\DeclareSymbol{Ita}{2}{\node at (0,-1) (a) {};
\node[var] at (0,2.25) (b) {\tiny{$\tau$ }};
\node[blank] at (-0.75,0.5) {\tiny{$\alpha$}};
\draw[symbols]  (a) -- (b);}
\DeclareSymbol{I1t1I1t2}{0}{\node[var] at (-2,2) (a) {\tiny{$\tau_1$}}; 
\node[var] at (2,2) (b) {\tiny{$\tau_2$}}; 
\draw[kernels2] (0,0) -- (a); 
\draw[kernels2] (0,0) -- (b); 
\node at (-1.5,-0.05) {{\tiny $\ell$}}; 
\node at (1.5,-0.25) {{\tiny $m$}};}
\DeclareSymbol{0I1t10I1t2}{0}{\node[var] at (-2,2) (a) {\tiny{$\tau_1$}}; 
\node[var] at (2,2) (b) {\tiny{$\tau_2$}}; 
\draw[kernels2] (0,0) -- (a); 
\draw[kernels2] (0,0) -- (b);}
\DeclareSymbol{0I1t11I1t2}{0}{\node[var] at (-2,2) (a) {\tiny{$\tau_1$}}; 
\node[var] at (2,2) (b) {\tiny{$\tau_2$}}; 
\draw[kernels2] (0,0) -- (a); 
\draw[kernels2] (0,0) -- (b); 
\node at (-1.5,-0.05) {{\tiny $0$}}; 
\node at (1.5,-0.05) {{\tiny $1$}};}
\DeclareSymbol{I1XiItI1Xi}{0}{\node[xi] at (-2,2) (a) {}; 
\node[xi] at (2,2) (b) {}; 
\node[var] at (0,3) (c) {\tiny{$\tau$}};
\draw[kernels2] (0,0) -- (a); 
\draw[kernels2] (0,0) -- (b);
\draw[] (0,0) -- (c);
\node at (-1.5,-0.05) {{\tiny $\ell$}}; 
\node at (1.5,-0.25) {{\tiny $m$}};}
\DeclareSymbol{I1t1I1t2It3}{0}{\node[var] at (-3,2) (a) {\tiny{$\tau_1$}}; 
\node[var] at (3,2) (b) {\tiny{$\tau_3$}}; 
\node[var] at (0,4) (c) {\tiny{$\tau_2$}};
\draw[kernels2] (0,0) -- (a); 
\draw (0,0) -- (b);
\draw[kernels2] (0,0) -- (c);
\node at (-1.5,-0.05) {{\tiny $\ell$}};
\node at (-1,2) {\tiny{$m$}}; 
\node at (1.5,-0.25) {{\tiny $n$}};}
\DeclareSymbol{XiIt1It2It3}{0}{\node[var] at (-3,2) (a) {\tiny{$\tau_1$}}; 
\node[var] at (3,2) (b) {\tiny{$\tau_3$}}; 
\node[var] at (0,4) (c) {\tiny{$\tau_2$}};
\node[xi] at (0,0) (d) {};
\draw (d) -- (a); 
\draw (d) -- (b);
\draw (d) -- (c);
\node at (-1.5,-0.05) {{\tiny $\ell$}};
\node at (-1,2) {\tiny{$m$}}; 
\node at (1.5,-0.25) {{\tiny $n$}};}
\DeclareSymbol{It1It2It3}{0}{\node[var] at (-3,2) (a) {\tiny{$\tau_1$}}; 
\node[var] at (3,2) (b) {\tiny{$\tau_3$}}; 
\node[var] at (0,4) (c) {\tiny{$\tau_2$}};
\draw (0,0) -- (a); 
\draw (0,0) -- (b);
\draw (0,0) -- (c);
\node at (-1.5,-0.05) {{\tiny $\ell$}};
\node at (-1,2) {\tiny{$m$}}; 
\node at (1.5,-0.25) {{\tiny $n$}};}
\DeclareSymbol{It1It2I1t3}{0}{\node[var] at (-3,2) (a) {\tiny{$\tau_1$}}; 
\node[var] at (3,2) (b) {\tiny{$\tau_3$}}; 
\node[var] at (0,4) (c) {\tiny{$\tau_2$}};
\draw (0,0) -- (a); 
\draw[kernels2] (0,0) -- (b);
\draw (0,0) -- (c);
\node at (-1.5,-0.05) {{\tiny $\ell$}};
\node at (-1,2) {\tiny{$m$}}; 
\node at (1.5,-0.25) {{\tiny $n$}};}
\DeclareSymbol{It1I1t2It3}{0}{\node[var] at (-3,2) (a) {\tiny{$\tau_1$}}; 
\node[var] at (3,2) (b) {\tiny{$\tau_3$}}; 
\node[var] at (0,4) (c) {\tiny{$\tau_2$}};
\draw (0,0) -- (a); 
\draw (0,0) -- (b);
\draw[kernels2] (0,0) -- (c);
\node at (-1.5,-0.05) {{\tiny $\ell$}};
\node at (-1,2) {\tiny{$m$}}; 
\node at (1.5,-0.25) {{\tiny $n$}};}
\DeclareSymbol{It1It2It3It4}{0}{\node[var] at (-4,2) (a) {\tiny{$\tau_1$}}; 
\node[var] at (4,2) (b) {\tiny{$\tau_3$}}; 
\node[var] at (-1.5,4) (c) {\tiny{$\tau_2$}};
\node[var] at (1.5,4) (d) {\tiny{$\tau_2$}};
\draw (0,0) -- (a); 
\draw (0,0) -- (b);
\draw (0,0) -- (c);
\draw (0,0) -- (d);
\node at (-1.5,-0.05) {{\tiny $k$}};
\node at (-1.5,1.75) {\tiny{$\ell$}}; 
\node at (1.5,-0.25) {{\tiny $m$}};
\node at (1.5, 1.75) {{\tiny $n$}};}
\DeclareSymbol{It1I1t2It3It4}{0}{\node[var] at (-4,2) (a) {\tiny{$\tau_1$}}; 
\node[var] at (4,2) (b) {\tiny{$\tau_3$}}; 
\node[var] at (-1.5,4) (c) {\tiny{$\tau_2$}};
\node[var] at (1.5,4) (d) {\tiny{$\tau_2$}};
\draw (0,0) -- (a); 
\draw (0,0) -- (b);
\draw[kernels2] (0,0) -- (c);
\draw (0,0) -- (d);
\node at (-1.5,-0.05) {{\tiny $k$}};
\node at (-1.5,1.75) {\tiny{$\ell$}}; 
\node at (1.5,-0.25) {{\tiny $m$}};
\node at (1.5, 1.75) {{\tiny $n$}};}
\DeclareSymbol{It1I1t2I1t3}{0}{\node[var] at (-3,2) (a) {\tiny{$\tau_1$}}; 
\node[var] at (3,2) (b) {\tiny{$\tau_3$}}; 
\node[var] at (0,4) (c) {\tiny{$\tau_2$}};
\draw (0,0) -- (a); 
\draw[kernels2] (0,0) -- (b);
\draw[kernels2] (0,0) -- (c);
\node at (-1.5,-0.05) {{\tiny $\ell$}};
\node at (-1,2) {\tiny{$m$}}; 
\node at (1.5,-0.25) {{\tiny $n$}};}
\DeclareSymbol{It1It2I1t3I1t4}{0}{\node[var] at (-4,2) (a) {\tiny{$\tau_1$}}; 
\node[var] at (4,2) (b) {\tiny{$\tau_4$}}; 
\node[var] at (2,4) (c) {\tiny{$\tau_3$}};
\node[var] at (-2,4) (d) {\tiny{$\tau_2$}};
\draw (0,0) -- (a); 
\draw[kernels2] (0,0) -- (b);
\draw[kernels2] (0,0) -- (c);
\draw (0,0) -- (d);
\node at (-1.75,-0.1) {{\tiny $k$}};
\node at (-2.15,2) {\tiny{$\ell$}};
\node at (2.15,2) {\tiny{$m$}}; 
\node at (1.75,-0.3) {{\tiny $n$}};}
\DeclareSymbol{XiItIXi}{0}{\node[var] at (-2,2) (a) {\tiny{$\tau$}}; 
\node[xi] at (2,2) (b) {}; 
\draw (0,0) -- (a); 
\draw (0,0) -- (b); 
\node at (-1.5,-0.05) {{\tiny $\ell$}}; 
\node at (1.5,-0.25) {{\tiny $m$}};
\node[xi] at (0,0) (c) {};}
\DeclareSymbol{It1It2}{0}{\node[var] at (-2,2) (a) {\tiny{$\tau_1$}}; 
\node[var] at (2,2) (b) {\tiny{$\tau_2$}}; 
\draw (0,0) -- (a); 
\draw (0,0) -- (b); 
\node at (-1.5,-0.05) {{\tiny $\ell$}}; 
\node at (1.5,-0.25) {{\tiny $m$}};}
\DeclareSymbol{It1I1t2}{0}{\node[var] at (-2,2) (a) {\tiny{$\tau_1$}}; 
\node[var] at (2,2) (b) {\tiny{$\tau_2$}}; 
\draw (0,0) -- (a); 
\draw[kernels2] (0,0) -- (b); 
\node at (-1.5,-0.05) {{\tiny $\ell$}}; 
\node at (1.5,-0.25) {{\tiny $m$}};}
\DeclareSymbol{Xi2b}{-2}{\draw (-1,-0.25) node[xic] {} -- (0,1) node[xic] {};} 
\DeclareSymbol{Xi2g}{-2}{\draw (-1,-0.25) node[xies] {} -- (0,1) node[xi] {};} 
\DeclareSymbol{Xi2g2}{-2}{\draw (-1,-0.25) node[xi] {} -- (0,1) node[xies] {};} 
\DeclareSymbol{cXi2}{-2}{\draw (0,-0.25) node[xi] {} -- (-1,1) node[xic] {};}
\DeclareSymbol{Xi3}{0}{\draw (0,0) node[xi] {} -- (-1,1) node[xi] {} -- (0,2) node[xi] {};}
\DeclareSymbol{XiIIXi}{0}{\draw (0,0) node[xi] {} -- (-1,1); \draw[kernels2] (-1,1) node[not] {} -- (0,2) node[xi] {};}

\DeclareSymbol{Xi4}{2}{\draw (0,0) node[xi] {} -- (-1,1) node[xi] {} -- (0,2) node[xi] {} -- (-1,3) node[xi] {};}
\DeclareSymbol{Xi4_1}{2}{\draw (0,0) node[xic] {} -- (-1,1) node[xic] {} -- (0,2) node[xi] {} -- (-1,3) node[xi] {};}
\DeclareSymbol{Xi4_2}{2}{\draw (0,0) node[xic] {} -- (-1,1) node[xi] {} -- (0,2) node[xi] {} -- (-1,3) node[xic] {};}
\DeclareSymbol{Xi2X}{-2}{\draw (0,-0.25) node[xi] {} -- (-1,1) node[xix] {};}
\DeclareSymbol{XXi2}{-2}{\draw (0,-0.25) node[xix] {} -- (-1,1) node[xi] {};}
\DeclareSymbol{IIXi}{0}{\draw (0,-0.25) node[not] {} -- (-1,1) node[xi] {} -- (0,2) node[xi] {};}
\DeclareSymbol{IXi^2}{-1}{\draw (-1,1) node[xi] {} -- (0,0) node[not] {} -- (1,1) node[xi] {};}
\DeclareSymbol{IIXi^2}{-4}{\draw (0,-1.5) node[not] {} -- (0,0);
\draw[kernels2] (-1,1) node[xi] {} -- (0,0) node[not] {} -- (1,1) node[xi] {};}
\DeclareSymbol{XiX}{-2.8}{\node[xibx] {};}
\DeclareSymbol{tauX}{-2.8}{ \node[X] {};}
\DeclareSymbol{Xi}{-2.8}{\node[xib] {};}

\DeclareSymbol{IXiX}{-1}{\draw (0,-0.25) node[not] {} -- (-1,1) node[xix] {};}
\DeclareSymbol{IXi3}{2}{\draw (0,-0.25) node[not] {} -- (-1,1) node[xi] {} -- (0,2) node[xi] {} -- (-1,3) node[xi] {};}
\DeclareSymbol{IXi}{-2}{\draw (0,-0.25) node[not] {} -- (-1,1) node[xi] {};}
\DeclareSymbol{XiI}{-2}{\draw (0,-0.25) node[xi] {} -- (-1,1) node[not] {};}

\DeclareSymbol{Xi4b}{0}{\draw(0,1.5) node[xi] {} -- (0,0); \draw (-1,1) node[xi] {} -- (0,0) node[xi] {} -- (1,1) node[xi] {};}
\DeclareSymbol{Xi4b'}{0}{\draw(0,1.5) node[xi] {} -- (0,-0.2); \draw (-1,1) node[xi] {} -- (0,-0.2) node[not] {} -- (1,1) node[xi] {};}
\DeclareSymbol{Xi4c}{0}{\draw (0,1) -- (0.8,2.2) node[xi] {};\draw (0,-0.25) node[xi] {} -- (0,1) node[xi] {} -- (-0.8,2.2) node[xi] {};}
\DeclareSymbol{Xi4d}{-4.5}{\draw (0,-1.5) node[not] {} -- (0,0); \draw (-1,1) node[xi] {} -- (0,0) node[xi] {} -- (1,1) node[xi] {};}
\DeclareSymbol{Xi4e}{0}{\draw (0,2) node[xi] {} -- (-1,1) node[xi] {} -- (0,0) node[xi] {} -- (1,1) node[xi] {};}
\DeclareSymbol{Xi4e'}{0}{\draw (0,2) node[xi] {} -- (-1,1) node[xi] {} -- (0,-0.2) node[not] {} -- (1,1) node[xi] {};}

\DeclareSymbol{Xitwo}
{0}{\draw[kernels2] (0,0) node[not] {} -- (-1,1) node[not] {}
-- (-2,2) node[not]{} -- (-3,3) node[xi]  {};
\draw[kernels2] (0,0) -- (1,1) node[xi] {};
\draw[kernels2] (-1,1) -- (0,2) node[xi] {};
\draw[kernels2] (-2,2) -- (-1,3) node[xi] {};}

\DeclareSymbol{IXitwo}
{0}{\draw (-.7,1.2) node[xi] {} -- (0,-0.2) -- (.7,1.2) node[xi] {};}
\DeclareSymbol{I1Xitwo}
{0}{\draw[kernels2] (0,0) node[not] {} -- (-1,1) node[xi] {};
\draw[kernels2] (0,0) -- (1,1) node[xi] {};}
\DeclareSymbol{I1Xitwou}
{0}{\draw[kernels2] (0,0) node[not] {} -- (-1,1) node[xi] {};
\draw[kernels2] (0,0) -- (1,1) node[xi] {}; \node at (-0.85,-0.2) {{\tiny $1$}}; \node at (0.9,-0.2) {{\tiny $0$}};}

\DeclareSymbol{I1Xitwoub}
{0}{\draw[kernels2] (0,0) node[not] {} -- (-1,1) node[xi] {};
	\draw[kernels2] (0,0) -- (1,1) node[xi] {}; \node at (-0.85,-0.2) {{\tiny $0$}}; \node at (0.9,-0.2) {{\tiny $1$}};}

\DeclareSymbol{I1Xitwoab}
{0}{\draw[kernels2] (0,0) node[not] {} -- (-1,1) node[xi] {};
\draw[kernels2] (0,0) -- (1,1) node[xi] {}; \node at (-0.85,-0.1) {{\tiny $\alpha$}}; \node at (0.9,-0.2) {{\tiny $\beta$}};}
\DeclareSymbol{I1Xitwoup}
{0}{\draw[kernels2] (0,0) node[not] {} -- (-1,1) node[xi] {};
\draw[kernels2] (0,0) -- (1,1) node[xi] {}; \node at (-0.85,0) {{\tiny $k$}}; \node at (0.9,-0.1) {{\tiny $l$}};}
\DeclareSymbol{I1Xitwobis}
{0}{\draw[kernels2] (0,0) node[not] {} -- (-1,1) node[xies] {};
\draw[kernels2] (0,0) -- (1,1) node[xies] {};}

\DeclareSymbol{I1Xitwog}
{0}{\draw[kernels2] (0,0) node[not] {} -- (-1,1) node[xies] {};
\draw[kernels2] (0,0) -- (1,1) node[xi] {};}

\DeclareSymbol{cI1Xitwo}
{0}{\draw[kernels2] (0,0) node[not] {} -- (-1,1) node[xic] {};
\draw[kernels2] (0,0) -- (1,1) node[xi] {};}

\DeclareSymbol{I1IXi3}{0}{\draw (0,0) node[xi] {} -- (-1,1) ; 
\draw[kernels2] (-1,1) node[not] {} -- (0,2) node[xi] {};
\draw[kernels2] (-1,1) node[not] {} -- (-2,2) node[xi] {};}

\DeclareSymbol{I1Xi3c}{-1}{\draw[kernels2](0,1.5) node[xi] {} -- (0,0) node[not] {}; \draw (-1,1) node[xi] {} -- (0,0) ; \draw[kernels2] (0,0) -- (1,1) node[xi] {};}

\DeclareSymbol{I1Xi3cbis}{-1}{\draw[kernels2](0,1.5) node[xies] {} -- (0,0) node[not] {}; \draw (-1,1) node[xies] {} -- (0,0) ; \draw[kernels2] (0,0) -- (1,1) node[xies] {};}

\DeclareSymbol{I1IXi3b}{0}{\draw[kernels2] (0,0) node[not] {} -- (-1,1) ; \draw[kernels2] (0,0)   -- (1,1) node[xi] {} ;
\draw (-1,1) node[xi] {} -- (0,2) node[xi] {};
}

\DeclareSymbol{I1IXi3c}{0}{\draw[kernels2] (0,0) node[not] {} -- (-1,1) ; \draw[kernels2] (0,0)   -- (1,1) node[xi] {} ;
\draw[kernels2] (-1,1) node[not] {} -- (0,2) node[xi] {};
\draw[kernels2] (-1,1) node[not] {} -- (-2,2) node[xi] {};}

\DeclareSymbol{I1IXi3cbis}{0}{\draw[kernels2] (0,0) node[not] {} -- (-1,1) ; \draw[kernels2] (0,0)   -- (1,1) node[xies] {} ;
\draw[kernels2] (-1,1) node[not] {} -- (0,2) node[xies] {};
\draw[kernels2] (-1,1) node[not] {} -- (-2,2) node[xies] {};}

\DeclareSymbol{I1Xi}{0}{\draw[kernels2] (0,0) node[not] {} -- (-1,1)  node[xi] {} ;}

\DeclareSymbol{I1Xi4a}{2}{\draw[kernels2] (0,0) node[not] {} -- (-1,1) ; \draw[kernels2] (0,0) node[not] {} -- (1,1) node[xi] {} ;
\draw (-1,1) node[xi] {} -- (0,2) node[xi] {} -- (-1,3) node[xi] {};}

\DeclareSymbol{cI1Xi4a}{2}{\draw[kernels2] (0,0) node[not] {} -- (-1,1) ; \draw[kernels2] (0,0) node[not] {} -- (1,1) node[xic] {} ;
\draw (-1,1) node[xic] {} -- (0,2) node[xi] {} -- (-1,3) node[xi] {};}

\DeclareSymbol{I1Xi4b}{2}{\draw (0,0) node[xi] {} -- (-1,1) node[xi] {} -- (0,2) ; \draw[kernels2] (0,2) node[not] {} -- (-1,3) node[xi] {};\draw[kernels2] (0,2)  -- (1,3) node[xi] {};
}

\DeclareSymbol{cI1Xi4b}{2}{\draw (0,0) node[xic] {} -- (-1,1) node[xic] {} -- (0,2) ; \draw[kernels2] (0,2) node[not] {} -- (-1,3) node[xi] {};\draw[kernels2] (0,2)  -- (1,3) node[xi] {};
}

\DeclareSymbol{I1Xi4c}{2}{\draw (0,0) node[xi] {} -- (-1,1) node[not] {}; \draw[kernels2] (-1,1) -- (0,2) ; 
\draw[kernels2] (-1,1) -- (-2,2) node[xi] {} ;
\draw (0,2) node[xi] {} -- (-1,3) node[xi] {};}

\DeclareSymbol{cI1Xi4c}{2}{\draw (0,0) node[xic] {} -- (-1,1) node[not] {}; \draw[kernels2] (-1,1) -- (0,2) ; 
\draw[kernels2] (-1,1) -- (-2,2) node[xic] {} ;
\draw (0,2) node[xi] {} -- (-1,3) node[xi] {};}

\DeclareSymbol{I1Xi4ab}{2}{\draw[kernels2] (0,0) node[not] {} -- (-1,1) ; \draw[kernels2] (0,0) node[not] {} -- (1,1) node[xi] {};\draw (-1,1) node[xi] {} -- (0,2) ; \draw[kernels2] (0,2) node[not] {} -- (-1,3) node[xi] {};\draw[kernels2] (0,2)  -- (1,3) node[xi] {}; }

\DeclareSymbol{cI1Xi4ab}{2}{\draw[kernels2] (0,0) node[not] {} -- (-1,1) ; \draw[kernels2] (0,0) node[not] {} -- (1,1) node[xic] {};\draw (-1,1) node[xic] {} -- (0,2) ; \draw[kernels2] (0,2) node[not] {} -- (-1,3) node[xi] {};\draw[kernels2] (0,2)  -- (1,3) node[xi] {}; }

\DeclareSymbol{I1Xi4bc}{2}{\draw (0,0) node[xi] {} -- (-1,1) node[not] {}; \draw[kernels2] (-1,1) -- (0,2) ; 
\draw[kernels2] (-1,1) -- (-2,2) node[xi] {} ; \draw[kernels2] (0,2) node[not] {} -- (-1,3) node[xi] {};\draw[kernels2] (0,2)  -- (1,3) node[xi] {};
}

\DeclareSymbol{cI1Xi4bc}{2}{\draw (0,0) node[xic] {} -- (-1,1) node[not] {}; \draw[kernels2] (-1,1) -- (0,2) ; 
\draw[kernels2] (-1,1) -- (-2,2) node[xic] {} ; \draw[kernels2] (0,2) node[not] {} -- (-1,3) node[xi] {};\draw[kernels2] (0,2)  -- (1,3) node[xi] {};
}

\DeclareSymbol{I1Xi4abcc1}{2}{\draw[kernels2] (0,0) node[not] {} -- (-1,1) node[not] {}
-- (-2,2) node[not]{} -- (-3,3) node[xic]  {};
\draw[kernels2] (0,0) -- (1,1) node[xic] {};
\draw[kernels2] (-1,1) -- (0,2) node[xi] {};
\draw[kernels2] (-2,2) -- (-1,3) node[xi] {};
}

\DeclareSymbol{I1Xi4abcc1b}{2}{\draw[kernels2] (0,0) node[not] {} -- (-1,1) node[not] {}
-- (-2,2) node[not]{} -- (-3,3) node[xi]  {};
\draw[kernels2] (0,0) -- (1,1) node[xic] {};
\draw[kernels2] (-1,1) -- (0,2) node[xic] {};
\draw[kernels2] (-2,2) -- (-1,3) node[xi] {};
}

\DeclareSymbol{I1Xi4abcc2}{2}{\draw[kernels2] (0,0) node[not] {} -- (-1,1) node[not] {}
-- (-2,2) node[not]{} -- (-3,3) node[xic]  {};
\draw[kernels2] (0,0) -- (1,1) node[xi] {};
\draw[kernels2] (-1,1) -- (0,2) node[xi] {};
\draw[kernels2] (-2,2) -- (-1,3) node[xic] {};
}

\DeclareSymbol{I1Xi4ac}{2}{\draw[kernels2] (0,0) node[not] {} -- (-1,1) ; \draw[kernels2] (0,0) node[not] {} -- (1,1) node[xi] {}; 
\draw[kernels2] (-1,1) node[not] {} -- (0,2) ; 
\draw[kernels2] (-1,1) -- (-2,2) node[xi] {} ;
\draw (0,2) node[xi] {} -- (-1,3) node[xi] {};}

\DeclareSymbol{cI1Xi4ac}{2}{\draw[kernels2] (0,0) node[not] {} -- (-1,1) ; \draw[kernels2] (0,0) node[not] {} -- (1,1) node[xic] {}; 
\draw[kernels2] (-1,1) node[not] {} -- (0,2) ; 
\draw[kernels2] (-1,1) -- (-2,2) node[xic] {} ;
\draw (0,2) node[xi] {} -- (-1,3) node[xi] {};}

\DeclareSymbol{I1Xi4acc1}{2}{\draw[kernels2] (0,0) node[not] {} -- (-1,1) ; \draw[kernels2] (0,0) node[not] {} -- (1,1) node[xic] {}; 
\draw[kernels2] (-1,1) node[not] {} -- (0,2) ; 
\draw[kernels2] (-1,1) -- (-2,2) node[xi] {} ;
\draw (0,2) node[xic] {} -- (-1,3) node[xi] {};}

\DeclareSymbol{I1Xi4acc2}{2}{\draw[kernels2] (0,0) node[not] {} -- (-1,1) ; \draw[kernels2] (0,0) node[not] {} -- (1,1) node[xic] {}; 
\draw[kernels2] (-1,1) node[not] {} -- (0,2) ; 
\draw[kernels2] (-1,1) -- (-2,2) node[xi] {} ;
\draw (0,2) node[xi] {} -- (-1,3) node[xic] {};}

\DeclareSymbol{2I1Xi4}{2}{\draw[kernels2] (0,0) node[not] {} -- (-1,1) node[not] {};
\draw[kernels2] (0,0) -- (1,1) node[not] {};
\draw[kernels2] (-1,1) -- (-1.5,2.5) node[xi] {};
\draw[kernels2] (-1,1) -- (-0.5,2.5) node[xi] {};
\draw[kernels2] (1,1) -- (0.5,2.5) node[xi] {};
\draw[kernels2] (1,1) -- (1.5,2.5) node[xi] {};
}

\DeclareSymbol{2I1Xi4dis}{2}{\draw[kernels2] (0,0) node[not] {} -- (-1,1) node[not] {};
\draw[kernels2] (0,0) -- (1,1) node[not] {};
\draw[kernels2] (-1,1) -- (-1.5,2.5) node[xies] {};
\draw[kernels2] (-1,1) -- (-0.5,2.5) node[xies] {};
\draw[kernels2] (1,1) -- (0.5,2.5) node[xies] {};
\draw[kernels2] (1,1) -- (1.5,2.5) node[xies] {};
}

\DeclareSymbol{2I1Xi4c1}{2}{\draw[kernels2] (0,0) node[not] {} -- (-1,1) node[not] {};
\draw[kernels2] (0,0) -- (1,1) node[not] {};
\draw[kernels2] (-1,1) -- (-1.5,2.5) node[xic] {};
\draw[kernels2] (-1,1) -- (-0.5,2.5) node[xi] {};
\draw[kernels2] (1,1) -- (0.5,2.5) node[xic] {};
\draw[kernels2] (1,1) -- (1.5,2.5) node[xi] {};
}

\DeclareSymbol{2I1Xi4c2}{2}{\draw[kernels2] (0,0) node[not] {} -- (-1,1) node[not] {};
\draw[kernels2] (0,0) -- (1,1) node[not] {};
\draw[kernels2] (-1,1) -- (-1.5,2.5) node[xic] {};
\draw[kernels2] (-1,1) -- (-0.5,2.5) node[xic] {};
\draw[kernels2] (1,1) -- (0.5,2.5) node[xi] {};
\draw[kernels2] (1,1) -- (1.5,2.5) node[xi] {};
}

\DeclareSymbol{2I1Xi4b}{2}{\draw[kernels2] (0,0) node[not] {} -- (-1,1) ;
\draw[kernels2] (0,0) -- (1,1);
\draw (-1,1) node[xi] {} -- (-1,2.5) node[xi] {};
\draw (1,1)  node[xi] {} -- (1,2.5) node[xi] {};
}

\DeclareSymbol{2I1Xi4bb}{2}{\draw[kernels2] (0,0) node[not] {} -- (-1,1) ;
\draw[kernels2] (0,0) -- (1,1);
\draw (-1,1) node[xi] {} -- (-1,2.5) node[xiesf] {};
\draw (1,1)  node[xi] {} -- (1,2.5) node[xic] {};
}

\DeclareSymbol{2I1Xi4c}{2}{\draw[kernels2] (0,0) node[not] {} -- (-1,1);
\draw[kernels2] (0,0) -- (1,1) node[not] {};
\draw (-1,1)  node[xi] {} -- (-1,2.5) node[xi] {};
\draw[kernels2] (1,1) -- (0.4,2.5) node[xi] {};
\draw[kernels2] (1,1) -- (1.6,2.5) node[xi] {};
}

\DeclareSymbol{2I1Xi4cc1}{2}{\draw[kernels2] (0,0) node[not] {} -- (-1,1);
\draw[kernels2] (0,0) -- (1,1) node[not] {};
\draw (-1,1)  node[xic] {} -- (-1,2.5) node[xi] {};
\draw[kernels2] (1,1) -- (0.4,2.5) node[xic] {};
\draw[kernels2] (1,1) -- (1.6,2.5) node[xi] {};
}

\DeclareSymbol{2I1Xi4cc2}{2}{\draw[kernels2] (0,0) node[not] {} -- (-1,1);
\draw[kernels2] (0,0) -- (1,1) node[not] {};
\draw (-1,1)  node[xic] {} -- (-1,2.5) node[xic] {};
\draw[kernels2] (1,1) -- (0.4,2.5) node[xi] {};
\draw[kernels2] (1,1) -- (1.6,2.5) node[xi] {};
}

\DeclareSymbol{Xi4ba}{0}{\draw(-0.5,1.5) node[xi] {} -- (0,0); \draw (-1.5,1) node[xi] {} -- (0,0) node[not] {}; \draw[kernels2] (0,0) -- (1.5,1) node[xi] {};
\draw[kernels2] (0,0) -- (0.5,1.5) node[xi] {} ;}

\DeclareSymbol{Xi4badis}{0}{\draw(-0.5,1.5) node[xies] {} -- (0,0); \draw (-1.5,1) node[xies] {} -- (0,0) node[not] {}; \draw[kernels2] (0,0) -- (1.5,1) node[xies] {};
\draw[kernels2] (0,0) -- (0.5,1.5) node[xies] {} ;}

\DeclareSymbol{Xi4ba1}{0}{\draw(-0.5,1.5) node[xi] {} -- (0,0); \draw (-1.5,1) node[xi] {} -- (0,0) node[not] {}; \draw[kernels2] (0,0) -- (1.5,1) node[xic] {};
\draw[kernels2] (0,0) -- (0.5,1.5) node[xic] {} ;}

\DeclareSymbol{Xi4ba1b}{0}{\draw(-0.5,1.5) node[xic] {} -- (0,0); \draw (-1.5,1) node[xic] {} -- (0,0) node[not] {}; \draw[kernels2] (0,0) -- (1.5,1) node[xi] {};
\draw[kernels2] (0,0) -- (0.5,1.5) node[xi] {} ;}

\DeclareSymbol{Xi4ba1bdiff}{0}{\draw(-0.5,1.5) node[xic] {} -- (0,0); \draw (-1.5,1) node[xic] {} -- (0,0) node[not] {}; \draw (0,0) -- (1.5,1) node[xi] {};
\draw (0,0) -- (0.5,1.5) node[xi] {};
\draw(0,0) node[diff] {};}

\DeclareSymbol{Xi4ba1bb}{0}{\draw(-0.5,1.5) node[xic] {} -- (0,0); \draw (-1.5,1) node[xiesf] {} -- (0,0) node[not] {}; \draw[kernels2] (0,0) -- (1.5,1) node[xi] {};
\draw[kernels2] (0,0) -- (0.5,1.5) node[xi] {} ;}

\DeclareSymbol{Xi4ba2}{0}{\draw(-0.5,1.5) node[xi] {} -- (0,0); \draw (-1.5,1) node[xic] {} -- (0,0) node[not] {}; \draw[kernels2] (0,0) -- (1.5,1) node[xi] {};
\draw[kernels2] (0,0) -- (0.5,1.5) node[xic] {} ;}

\DeclareSymbol{Xi4ba2b}{0}{\draw(-0.5,1.5) node[xi] {} -- (0,0); \draw (-1.5,1) node[xic] {} -- (0,0) node[not] {}; \draw[kernels2] (0,0) -- (1.5,1) node[xi] {};
\draw[kernels2] (0,0) -- (0.5,1.5) node[xiesf] {} ;}

\DeclareSymbol{Xi4ca}{0}{\draw (0,1) -- (-1,2.2) node[xi] {};\draw (0,-0.25) node[xi] {} -- (0,1) ; \draw[kernels2] (0,1) node[not] {} -- (1,2.2) node[xi] {};
\draw[kernels2] (0,1) {} -- (0,2.7) node[xi] {};
}

\DeclareSymbol{Xi4cb}{0}{\draw (-1,1) -- (-2,2) node[xi] {};\draw[kernels2] (0,0)  -- (-1,1) node[xi] {} ; \draw[kernels2] (0,0) node[not] {} -- (1,1) node[xi] {} ; 
\draw (-1,1) node[xi] {} -- (0,2) node[xi] {};}

\DeclareSymbol{Xi4cbb}{0}{\draw (-1,1) -- (-2,2) node[xiesf] {};\draw[kernels2] (0,0)  -- (-1,1) node[xi] {} ; \draw[kernels2] (0,0) node[not] {} -- (1,1) node[xi] {} ; 
\draw (-1,1) node[xi] {} -- (0,2) node[xic] {};}

\DeclareSymbol{Xi4cbc1}{0}{\draw (-1,1) -- (-2,2) node[xic] {};\draw[kernels2] (0,0)  -- (-1,1) node[xic] {} ; \draw[kernels2] (0,0) node[not] {} -- (1,1) node[xi] {} ; 
\draw (-1,1) node[xic] {} -- (0,2) node[xi] {};}

\DeclareSymbol{Xi4cbc2}{0}{\draw (-1,1) -- (-2,2) node[xi] {};\draw[kernels2] (0,0)  -- (-1,1) node[xi] {} ; \draw[kernels2] (0,0) node[not] {} -- (1,1) node[xic] {} ; 
\draw (-1,1) node[xic] {} -- (0,2) node[xi] {};}

\DeclareSymbol{Xi4cab}{0}{\draw (-1,1) -- (-2,2) node[xi] {};\draw[kernels2] (0,0)  -- (-1,1); \draw[kernels2] (0,0) node[not] {} -- (1,1) node[xi] {} ; 
\draw[kernels2] (-1,1)  {} -- (0,2) node[xi] {};
\draw[kernels2] (-1,1) node[not] {} -- (-1,2.5) node[xi] {};
}

\DeclareSymbol{Xi4cabdis}{0}{\draw (-1,1) -- (-2,2) node[xies] {};\draw[kernels2] (0,0)  -- (-1,1); \draw[kernels2] (0,0) node[not] {} -- (1,1) node[xies] {} ; 
\draw[kernels2] (-1,1)  {} -- (0,2) node[xies] {};
\draw[kernels2] (-1,1) node[not] {} -- (-1,2.5) node[xies] {};
}

\DeclareSymbol{Xi4cabc1}{0}{\draw (-1,1) -- (-2,2) node[xi] {};\draw[kernels2] (0,0)  -- (-1,1); \draw[kernels2] (0,0) node[not] {} -- (1,1) node[xic] {} ; 
\draw[kernels2] (-1,1)  {} -- (0,2) node[xic] {};
\draw[kernels2] (-1,1) node[not] {} -- (-1,2.5) node[xi] {};
}

\DeclareSymbol{Xi4cabc2}{0}{\draw (-1,1) -- (-2,2) node[xic] {};\draw[kernels2] (0,0)  -- (-1,1); \draw[kernels2] (0,0) node[not] {} -- (1,1) node[xic] {} ; 
\draw[kernels2] (-1,1)  {} -- (0,2) node[xi] {};
\draw[kernels2] (-1,1) node[not] {} -- (-1,2.5) node[xi] {};
}

\DeclareSymbol{Xi4ea}{1.5}{\draw (-1,2.5) node[xi] {} -- (-1,1) node[xi] {} -- (0,0); 
 \draw[kernels2] (0,0)  -- (1,1) node[xi] {};
\draw[kernels2] (0,0) node[not] {} -- (0,1.5) node[xi] {}; }

\DeclareSymbol{Xi4eac1}{1.5}{\draw (-1,2.5) node[xic] {} -- (-1,1) node[xi] {} -- (0,0); 
 \draw[kernels2] (0,0)  -- (1,1) node[xic] {};
\draw[kernels2] (0,0) node[not] {} -- (0,1.5) node[xi] {}; }

\DeclareSymbol{Xi4eac1b}{1.5}{\draw (-1,2.5) node[xic] {} -- (-1,1) node[xi] {} -- (0,0); 
 \draw[kernels2] (0,0)  -- (1,1) node[xiesf] {};
\draw[kernels2] (0,0) node[not] {} -- (0,1.5) node[xi] {}; }

\DeclareSymbol{Xi4eac2}{1.5}{\draw (-1,2.5) node[xic] {} -- (-1,1) node[xic] {} -- (0,0); 
 \draw[kernels2] (0,0)  -- (1,1) node[xi] {};
\draw[kernels2] (0,0) node[not] {} -- (0,1.5) node[xi] {}; }

\DeclareSymbol{Xi4eact1}{1.5}{\draw (-1,2.5) node[xic] {} -- (-1,1) node[xi] {} -- (0,0); 
 \draw (0,0)  -- (1,1) node[xic] {};
\draw[rho] (0,0) node[not] {} -- (0,1.5) node[xi] {}; }

\DeclareSymbol{Xi4eact2}{1.5}{\draw[rho] (-1,2.5) node[xic] {} -- (-1,1) node[xi] {} -- (0,0); 
 \draw (0,0)  -- (1,1) node[xic] {};
\draw (0,0) node[not] {} -- (0,1.5) node[xi] {}; }

\DeclareSymbol{Xi4eabis}{1.5}{\draw (-1,2.5) node[xi] {} -- (-1,1) ; \draw[kernels2] (-1,1) node[xi] {} -- (0,0); 
 \draw (0,0)  -- (1,1) node[xi] {};
\draw[kernels2] (0,0) node[not] {} -- (0,1.5) node[xi] {}; }

\DeclareSymbol{Xi4eabisc1}{1.5}{\draw (-1,2.5) node[xic] {} -- (-1,1) ; \draw[kernels2] (-1,1) node[xi] {} -- (0,0); 
 \draw (0,0)  -- (1,1) node[xi] {};
\draw[kernels2] (0,0) node[not] {} -- (0,1.5) node[xic] {}; }

\DeclareSymbol{Xi4eabisc1b}{1.5}{\draw (-1,2.5) node[xic] {} -- (-1,1) ; \draw[kernels2] (-1,1) node[xi] {} -- (0,0); 
 \draw (0,0)  -- (1,1) node[xi] {};
\draw[kernels2] (0,0) node[not] {} -- (0,1.5) node[xiesf] {}; }

\DeclareSymbol{Xi4eabisc1bis}{1.5}{\draw (-1,2.5) node[xi] {} -- (-1,1) ; \draw[kernels2] (-1,1) node[xi] {} -- (0,0); 
 \draw (0,0)  -- (1,1) node[xi] {};
\draw[kernels2] (0,0) node[not] {} -- (0,1.5) node[xi] {};
\draw (-2,1) node[] {\tiny{$i$}};
\draw (-2,2.5) node[] {\tiny{$\ell$}};
\draw (2,1) node[] {\tiny{$k$}};
\draw (0,2.5) node[] {\tiny{$j$}};
 }

\DeclareSymbol{Xi4eabisc1tris}{1.5}{\draw (-1,2.5) node[xi] {} -- (-1,1) ; \draw[kernels2] (-1,1) node[xi] {} -- (0,0); 
 \draw (0,0)  -- (1,1) node[xi] {};
\draw[kernels2] (0,0) node[not] {} -- (0,1.5) node[xi] {};
\draw (-2,1) node[] {\tiny{i}};
\draw (-2,2.5) node[] {\tiny{j}};
\draw (2,1) node[] {\tiny{j}};
\draw (0,2.5) node[] {\tiny{i}};
 }

\DeclareSymbol{Xi4eabisc1quater}{1.5}{\draw (-1,2.5) node[xic] {} -- (-1,1) ; \draw[kernels2] (-1,1) node[xi] {} -- (0,0); 
 \draw (0,0)  -- (1,1) node[xic] {};
\draw[kernels2] (0,0) node[not] {} -- (0,1.5) node[xi] {};
 }

\DeclareSymbol{Xi4eabisc2}{1.5}{\draw (-1,2.5) node[xic] {} -- (-1,1) ; \draw[kernels2] (-1,1) node[xi] {} -- (0,0); 
 \draw (0,0)  -- (1,1) node[xic] {};
\draw[kernels2] (0,0) node[not] {} -- (0,1.5) node[xi] {}; }

\DeclareSymbol{Xi4eabisc2l}{1.5}{\draw (-1,2.5) node[xiesf] {} -- (-1,1) ; \draw[kernels2] (-1,1) node[xi] {} -- (0,0); 
 \draw (0,0)  -- (1,1) node[xic] {};
\draw[kernels2] (0,0) node[not] {} -- (0,1.5) node[xi] {}; }

\DeclareSymbol{Xi4eabisc2r}{1.5}{\draw (-1,2.5) node[xic] {} -- (-1,1) ; \draw[kernels2] (-1,1) node[xi] {} -- (0,0); 
 \draw (0,0)  -- (1,1) node[xiesf] {};
\draw[kernels2] (0,0) node[not] {} -- (0,1.5) node[xi] {}; }

\DeclareSymbol{Xi4eabisc3}{1.5}{\draw (-1,2.5) node[xic] {} -- (-1,1) ; \draw[kernels2] (-1,1) node[xic] {} -- (0,0); 
 \draw (0,0)  -- (1,1) node[xi] {};
\draw[kernels2] (0,0) node[not] {} -- (0,1.5) node[xi] {}; }

\DeclareSymbol{Xi4eb}{0}{
\draw[kernels2] (0,2) node[xi] {} -- (-1,1) ; \draw[kernels2] (-2,2)  node[xi] {} -- (-1,1) ; \draw (-1,1)  node[not] {} -- (0,0); 
 \draw (0,0) node[xi] {}  -- (1,1) node[xi] {};
}

\DeclareSymbol{Xi4eab}{1.5}{\draw[kernels2] (-1,2.5) node[xi] {} -- (-1,1) ; \draw[kernels2] (-2,2)  node[xi] {} -- (-1,1) ; \draw (-1,1)  node[not] {} -- (0,0); 
 \draw[kernels2] (0,0)  -- (1,1) node[xi] {};
\draw[kernels2] (0,0) node[not] {} -- (0,1.5) node[xi] {}; 
}

\DeclareSymbol{Xi4eabdis}{1.5}{\draw[kernels2] (-1,2.5) node[xies] {} -- (-1,1) ; \draw[kernels2] (-2,2)  node[xies] {} -- (-1,1) ; \draw (-1,1)  node[not] {} -- (0,0); 
 \draw[kernels2] (0,0)  -- (1,1) node[xies] {};
\draw[kernels2] (0,0) node[not] {} -- (0,1.5) node[xies] {}; 
}

\DeclareSymbol{Xi4eabc1}{1.5}{\draw[kernels2] (-1,2.5) node[xic] {} -- (-1,1) ; \draw[kernels2] (-2,2)  node[xi] {} -- (-1,1) ; \draw (-1,1)  node[not] {} -- (0,0); 
 \draw[kernels2] (0,0)  -- (1,1) node[xic] {};
\draw[kernels2] (0,0) node[not] {} -- (0,1.5) node[xi] {}; 
}

\DeclareSymbol{Xi4eabc2}{1.5}{\draw[kernels2] (-1,2.5) node[xi] {} -- (-1,1) ; \draw[kernels2] (-2,2)  node[xi] {} -- (-1,1) ; \draw (-1,1)  node[not] {} -- (0,0); 
 \draw[kernels2] (0,0)  -- (1,1) node[xic] {};
\draw[kernels2] (0,0) node[not] {} -- (0,1.5) node[xic] {}; 
}

\DeclareSymbol{Xi4eabbis}{1.5}{\draw[kernels2] (-1,2.5) node[xi] {} -- (-1,1) ; \draw[kernels2] (-2,2)  node[xi] {} -- (-1,1) ; \draw[kernels2] (-1,1)  node[not] {} -- (0,0); 
 \draw (0,0)  -- (1,1) node[xi] {};
\draw[kernels2] (0,0) node[not] {} -- (0,1.5) node[xi] {}; 
}

\DeclareSymbol{Xi4eabbisc1}{1.5}{\draw[kernels2] (-1,2.5) node[xic] {} -- (-1,1) ; \draw[kernels2] (-2,2)  node[xi] {} -- (-1,1) ; \draw[kernels2] (-1,1)  node[not] {} -- (0,0); 
 \draw (0,0)  -- (1,1) node[xic] {};
\draw[kernels2] (0,0) node[not] {} -- (0,1.5) node[xi] {}; 
}

\DeclareSymbol{Xi4eabbisc1perm}{1.5}{\draw[kernels2] (-1,2.5) node[xi] {} -- (-1,1) ; \draw[kernels2] (-2,2)  node[xic] {} -- (-1,1) ; \draw[kernels2] (-1,1)  node[not] {} -- (0,0); 
 \draw (0,0)  -- (1,1) node[xic] {};
\draw[kernels2] (0,0) node[not] {} -- (0,1.5) node[xi] {}; 
}

\DeclareSymbol{Xi4eabbisc2}{1.5}{\draw[kernels2] (-1,2.5) node[xi] {} -- (-1,1) ; \draw[kernels2] (-2,2)  node[xi] {} -- (-1,1) ; \draw[kernels2] (-1,1)  node[not] {} -- (0,0); 
 \draw (0,0)  -- (1,1) node[xic] {};
\draw[kernels2] (0,0) node[not] {} -- (0,1.5) node[xic] {}; 
}

\DeclareSymbol{Xi2cbis}{0}{\draw[kernels2] (0,1) -- (0.8,2.2) node[xi] {};\draw[kernels2] (0,-0.25) node[not] {} -- (0,1); \draw[kernels2] (0,1) node[not] {} -- (-0.8,2.2) node[xi] {};}

\DeclareSymbol{Xi2cbis1}{0}{\draw (0,1) -- (-0.8,2.2) node[xi] {};\draw[kernels2] (0,-0.25) node[not] {} -- (0,1) node[xi] {}; }

\DeclareSymbol{Xi2Xbis}{-2}{\draw[kernels2] (0,-0.25)  -- (-1,1) ; \draw (-1,1) node[xix] {};
\draw[kernels2] (0,-0.25) node[not] {} -- (1,1) node[xi] {};}

\DeclareSymbol{XXi2bis}{-2}{\draw[kernels2] (0,-0.25) -- (-1,1) node[xi] {};
\draw[kernels2] (0,-0.25) node[X] {} -- (1,1) node[xi] {};}

\DeclareSymbol{I1XiIXi}{0}{\draw[kernels2] (0,-0.25) -- (1,1) node[xi] {};
\draw (0,-0.25) node[not] {} -- (-1,1) node[xi] {};}

\DeclareSymbol{I1XiIXib}{0}{\draw  (0,-0.25) node[xi] {} -- (0,1) node[not] {};
\draw[kernels2] (0,1) -- (0,2.25) ; \draw (0,2.25) node[xi]{}; }

\DeclareSymbol{I1XiIXic}{0}{
\draw[kernels2] (0,0) -- (1,1) node[xi] {} ; 
\draw[kernels2] (0,0) node[not] {}  -- (-1,1) node[not] {} -- (0,2) node[xi] {};
}

\DeclareSymbol{thin}{1.4}{\draw[pagebackground] (-0.3,0) -- (0.3,0); \draw  (0,0) -- (0,2);}
\DeclareSymbol{thin2}{1.4}{\draw[pagebackground] (-0.3,0) -- (0.3,0); \draw[tinydots]  (0,0) -- (0,2);}

\DeclareSymbol{thick}{1.4}{\draw[pagebackground] (-0.3,0) -- (0.3,0); \draw[kernels2]  (0,0) -- (0,2);}
\DeclareSymbol{thick2}{1.4}{\draw[pagebackground] (-0.3,0) -- (0.3,0); \draw[kernels2,tinydots]  (0,0) -- (0,2);}

\DeclareSymbol{Xi4ind}{2}{\draw (0,0) node[xi,label={[label distance=-0.2em]right: \scriptsize  $ i $}]  { } -- (-1,1) node[xi,label={[label distance=-0.2em]left: \scriptsize  $ j $}] {} -- (0,2) node[xi,label={[label distance=-0.2em]right: \scriptsize  $ k $}] {} -- (-1,3) node[xi,label={[label distance=-0.2em]left: \scriptsize  $ \ell $}] {};}

\DeclareSymbol{Xi4c1}{2}{\draw (0,0) node[xic] {} -- (-1,1) node[xi] {} -- (0,2) node[xic] {} -- (-1,3) node[xi] {};} 
\DeclareSymbol{IXi2ex}{0}{\draw (0,-0.25) node[xie] {} -- (-1,1) node[xi] {} ; \draw (0,-0.25)-- (1,1) node[xi] {};}
\DeclareSymbol{IXi2ex1}{0}{\draw (0,-0.25) node[xie] {} -- (-1,1) node[xi] {} -- (0,2) node[xi] {};}

\DeclareSymbol{Xi4b1}{0}{\draw(0,1.5) node[xic] {} -- (0,0); \draw (-1,1) node[xic] {} -- (0,0) node[xi] {} -- (1,1) node[xi] {};}

\DeclareSymbol{Xi4ec1}{0}{\draw (0,2) node[xi] {} -- (-1,1) node[xic] {} -- (0,0) node[xic] {} -- (1,1) node[xi] {};}
\DeclareSymbol{Xi4ec2}{0}{\draw (0,2) node[xic] {} -- (-1,1) node[xi] {} -- (0,0) node[xic] {} -- (1,1) node[xi] {};}
\DeclareSymbol{Xi4ec3}{0}{\draw (0,2) node[xic] {} -- (-1,1) node[xic] {} -- (0,0) node[xi] {} -- (1,1) node[xi] {};}

\DeclareSymbol{I1Xi4ac1}{2}{\draw[kernels2] (0,0) node[not] {} -- (-1,1) ; \draw[kernels2] (0,0) node[not] {} -- (1,1) node[xic] {} ;
\draw (-1,1) node[xi] {} -- (0,2) node[xic] {} -- (-1,3) node[xi] {};}

\DeclareSymbol{I1Xi4ac2}{2}{\draw[kernels2] (0,0) node[not] {} -- (-1,1) ; \draw[kernels2] (0,0) node[not] {} -- (1,1) node[xic] {} ;
\draw (-1,1) node[xi] {} -- (0,2) node[xi] {} -- (-1,3) node[xic] {};}

\DeclareSymbol{I1Xi4bp}{2}{\draw (0,0) node[not] {} -- (-1,1) node[xi] {} -- (0,2) ; \draw[kernels2] (0,2) node[not] {} -- (-1,3) node[xi] {};\draw[kernels2] (0,2)  -- (1,3) node[xi] {};
}

\DeclareSymbol{I1Xi4bc1}{2}{\draw (0,0) node[xic] {} -- (-1,1) node[xi] {} -- (0,2) ; \draw[kernels2] (0,2) node[not] {} -- (-1,3) node[xi] {};\draw[kernels2] (0,2)  -- (1,3) node[xic] {};
}

\DeclareSymbol{I1Xi4bc2}{2}{\draw (0,0) node[xic] {} -- (-1,1) node[xi] {} -- (0,2) ; \draw[kernels2] (0,2) node[not] {} -- (-1,3) node[xic] {};\draw[kernels2] (0,2)  -- (1,3) node[xi] {};
}

\DeclareSymbol{I1Xi4cp}{2}{\draw (0,0) node[not] {} -- (-1,1) node[not] {}; \draw[kernels2] (-1,1) -- (0,2) ; 
\draw[kernels2] (-1,1) -- (-2,2) node[xi] {} ;
\draw (0,2) node[xi] {} -- (-1,3) node[xi] {};}

\DeclareSymbol{I1Xi4cc1}{2}{\draw (0,0) node[xic] {} -- (-1,1) node[not] {}; \draw[kernels2] (-1,1) -- (0,2) ; 
\draw[kernels2] (-1,1) -- (-2,2) node[xi] {} ;
\draw (0,2) node[xic] {} -- (-1,3) node[xi] {};}

\DeclareSymbol{I1Xi4cc2}{2}{\draw (0,0) node[xic] {} -- (-1,1) node[not] {}; \draw[kernels2] (-1,1) -- (0,2) ; 
\draw[kernels2] (-1,1) -- (-2,2) node[xi] {} ;
\draw (0,2) node[xi] {} -- (-1,3) node[xic] {};}

\DeclareSymbol{I1Xi4abc1}{2}{\draw[kernels2] (0,0) node[not] {} -- (-1,1) ; \draw[kernels2] (0,0) node[not] {} -- (1,1) node[xic] {};\draw (-1,1) node[xi] {} -- (0,2) ; \draw[kernels2] (0,2) node[not] {} -- (-1,3) node[xic] {};\draw[kernels2] (0,2)  -- (1,3) node[xi] {}; }

\DeclareSymbol{I1Xi4abc2}{2}{\draw[kernels2] (0,0) node[not] {} -- (-1,1) ; \draw[kernels2] (0,0) node[not] {} -- (1,1) node[xic] {};\draw (-1,1) node[xi] {} -- (0,2) ; \draw[kernels2] (0,2) node[not] {} -- (-1,3) node[xi] {};\draw[kernels2] (0,2)  -- (1,3) node[xic] {}; }

\DeclareSymbol{R1}{0}{\draw (-1,1) node[xi] {} -- (0,0) node[not] {};
\draw[kernels2] (0,1.5) node[xic] {} -- (0,0) -- (1,1) node[xic] {};}
\DeclareSymbol{R2}{0}{\draw (-1,1) node[xic] {} -- (0,0) node[not] {};
\draw[kernels2] (0,1.5)  {} -- (0,0) -- (1,1)  {};
\draw (0,1.5) node[xi] {};
\draw (1,1) node[xic] {};
}
\DeclareSymbol{R3}{1}{\draw[kernels2] (-1,1.5)  {} -- (0,0) node[not] {} -- (1,1.5);
\draw (-1,1.5) node[xi] {};
\draw[kernels2] (0,3) {} -- (1,1.5) -- (2,3)  {};
\draw  (0,3) node[xic] {} ;
\draw (2,3) node[xic] {};}
\DeclareSymbol{R4}{1}{\draw[kernels2] (-1,1.5) node[xic] {} -- (0,0) node[not] {} -- (1,1.5);
\draw[kernels2] (0,3) {} -- (1,1.5) -- (2,3) node[xic] {};
\draw (0,3) node[xi] {};}

\DeclareSymbol{I1Xi4bcp}{2}{\draw (0,0) node[not] {} -- (-1,1) node[not] {}; \draw[kernels2] (-1,1) -- (0,2) ; 
\draw[kernels2] (-1,1) -- (-2,2) node[xi] {} ; \draw[kernels2] (0,2) node[not] {} -- (-1,3) node[xi] {};\draw[kernels2] (0,2)  -- (1,3) node[xi] {};
}

\DeclareSymbol{I1Xi4bcc1}{2}{\draw (0,0) node[xic] {} -- (-1,1) node[not] {}; \draw[kernels2] (-1,1) -- (0,2) ; 
\draw[kernels2] (-1,1) -- (-2,2) node[xi] {} ; \draw[kernels2] (0,2) node[not] {} -- (-1,3) node[xi] {};\draw[kernels2] (0,2)  -- (1,3) node[xic] {};
}

\DeclareSymbol{I1Xi4bcc2}{2}{\draw (0,0) node[xic] {} -- (-1,1) node[not] {}; \draw[kernels2] (-1,1) -- (0,2) ; 
\draw[kernels2] (-1,1) -- (-2,2) node[xi] {} ; \draw[kernels2] (0,2) node[not] {} -- (-1,3) node[xic] {};\draw[kernels2] (0,2)  -- (1,3) node[xi] {};
} 

\DeclareSymbol{2I1Xi4bc1}{2}{\draw[kernels2] (0,0) node[not] {} -- (-1,1) ;
\draw[kernels2] (0,0) -- (1,1);
\draw (-1,1) node[xic] {} -- (-1,2.5) node[xi] {};
\draw (1,1)  node[xic] {} -- (1,2.5) node[xi] {};
}

\DeclareSymbol{2I1Xi4bc2}{2}{\draw[kernels2] (0,0) node[not] {} -- (-1,1) ;
\draw[kernels2] (0,0) -- (1,1);
\draw (-1,1) node[xi] {} -- (-1,2.5) node[xic] {};
\draw (1,1)  node[xic] {} -- (1,2.5) node[xi] {};
}

\DeclareSymbol{diff2I1Xi4bc2}{2}{\draw (0,0) node[diff] {} -- (-1,1) ;
\draw (0,0) -- (1,1);
\draw (-1,1) node[xi] {} -- (-1,2.5) node[xic] {};
\draw (1,1)  node[xic] {} -- (1,2.5) node[xi] {};
}

\DeclareSymbol{2I1Xi4bc3}{2}{\draw[kernels2] (0,0) node[not] {} -- (-1,1) ;
\draw[kernels2] (0,0) -- (1,1);
\draw (-1,1) node[xic] {} -- (-1,2.5) node[xic] {};
\draw (1,1)  node[xi] {} -- (1,2.5) node[xi] {};
}

\DeclareSymbol{Xi41}{0}{\draw (0,1) -- (0.8,2.2) node[xic] {};\draw (0,-0.25) node[xi] {} -- (0,1) node[xi] {} -- (-0.8,2.2) node[xic] {};} 

\DeclareSymbol{Xi42}{0}{\draw (0,1) -- (0.8,2.2) node[xi] {};\draw (0,-0.25) node[xic] {} -- (0,1) node[xi] {} -- (-0.8,2.2) node[xic] {};}

\DeclareSymbol{Xi4ca1}{0}{\draw (0,1) -- (-1,2.2) node[xic] {};\draw (0,-0.25) node[xi] {} -- (0,1) ; \draw[kernels2] (0,1) node[not] {} -- (1,2.2) node[xic] {};
\draw[kernels2] (0,1) {} -- (0,2.7) node[xi] {};
}

\DeclareSymbol{Xi4ca2}{0}{\draw (0,1) -- (-1,2.2) node[xi] {};\draw (0,-0.25) node[xi] {} -- (0,1) ; \draw[kernels2] (0,1) node[not] {} -- (1,2.2) node[xic] {};
\draw[kernels2] (0,1) {} -- (0,2.7) node[xic] {};
}

\DeclareSymbol{Xi4cap}{0}{\draw (0,1) -- (-1,2.2) node[xi] {};\draw (0,-0.25) node[not] {} -- (0,1) ; \draw[kernels2] (0,1) node[not] {} -- (1,2.2) node[xi] {};
\draw[kernels2] (0,1) {} -- (0,2.7) node[xi] {};
}

\DeclareSymbol{Xi3a}{0}{
 \draw (-1,1)  node[xi] {} -- (0,0); 
 \draw (0,0) node[xi] {}  -- (1,1) node[xi] {};
}

\DeclareSymbol{Xi4ebc1}{0}{
\draw[kernels2] (0,2) node[xi] {} -- (-1,1) ; \draw[kernels2] (-2,2)  node[xic] {} -- (-1,1) ; \draw (-1,1)  node[not] {} -- (0,0); 
 \draw (0,0) node[xic] {}  -- (1,1) node[xi] {};
}

\DeclareSymbol{Xi4ebc2}{0}{
\draw[kernels2] (0,2) node[xi] {} -- (-1,1) ; \draw[kernels2] (-2,2)  node[xi] {} -- (-1,1) ; \draw (-1,1)  node[not] {} -- (0,0); 
 \draw (0,0) node[xic] {}  -- (1,1) node[xic] {};
}

\DeclareSymbol{Xi2cbispex}{0}{\draw[kernels2] (0,1) -- (0.8,2.2) node[xi] {};\draw (0,-0.25) node[xie] {} -- (0,1); \draw[kernels2] (0,1) node[not] {} -- (-0.8,2.2) node[xi] {};}

\DeclareSymbol{Xi2cbis1p}{0}{\draw (0,1) -- (-0.8,2.2) node[xi] {};\draw (0,-0.25) node[not] {} -- (0,1) node[xi] {}; }

\DeclareSymbol{Xi2Xp}{-2}{\draw (0,-0.25) node[not] {} -- (-1,1) node[xix] {};} 

\DeclareSymbol{I1XiIXib}{0}{\draw  (0,-0.25) node[xi] {} -- (0,1) node[not] {};
\draw[kernels2] (0,1) -- (0,2.25) ; \draw (0,2.25) node[xi]{}; }

\DeclareSymbol{IXi2b}{0}{\draw  (0,-0.25) node[xi] {} -- (0,1) node[not] {};
\draw (0,1) -- (0,2.25) ; \draw (0,2.25) node[xi]{}; }

\DeclareSymbol{IXi2bex}{0}{\draw  (0,-0.25) node[xi] {} -- (0,1) node[xie] {};
\draw (0,1) -- (0,2.25) ; \draw (0,2.25) node[xi]{}; }

 \def\1{\mathbf{\symbol{1}}}

\def\eps{\varepsilon}

\DeclareSymbol{diff}{0}{
\draw (0,0.5) node[diff] {};
}

\DeclareSymbol{diff1}{0}{
\draw (0,0.5) node[diff1] {};
}

\DeclareSymbol{diff2}{0}{
\draw (0,0.5) node[diff2] {};
}

\DeclareSymbol{geo}{0}{
\draw (0,0) node[diff] {};
\draw (0.3,0) node[diff] {};
}

\DeclareSymbol{generic}{0}{
\draw (0,0.6) node[xi] {};
}
\DeclareSymbol{derivative}{scale=0.05,baseline=-2}{
\coordinate (root) at (0,-0.4);
			\coordinate (t1) at (-.8,1.3);
			\draw[kernels2] (t1) -- (root);
			\node[not] (rootnode) at (root) {};
}
\DeclareSymbol{derivative2}{scale=0.05,baseline=-2}{
\coordinate (root) at (0,-0.4);
			\coordinate (t1) at (-.8,1.3);
			\coordinate (t2) at (.8,1.3);
			\draw[kernels2] (t1) -- (root);
			\draw[kernels2] (t2) -- (root);
			\node[not] (rootnode) at (root) {};
}

\DeclareSymbol{g}{0}{
\draw (0,0.6) node[g] {};
}

\DeclareSymbol{Ito}{0}{
\draw (0,0.6) node[xies] {};
}

\DeclareSymbol{Itob}{0}{
\draw (0,0.6) node[xiesf] {};
}

\DeclareSymbol{greycirc}{0}{
\draw (0,0.3) node[xi] {};
}

\DeclareSymbol{not}{0}{
\draw (0,0.6) node[not] {};
\draw[tinydots] (0,0.6) circle (0.8);
}

\DeclareSymbol{genericb}{0}{
\draw (0,0.6) node[xic] {};
}

\DeclareSymbol{bluecirc}{0}{
\draw (0,0.3) node[xic] {};
}

\DeclareSymbol{genericxix}{0}{
\draw (0,0.6) node[xix] {};
}

\DeclareSymbol{genericX}{0}{
\draw (0,0.6) node[X] {};
}

\DeclareSymbol{diffIto}{1}{
\draw  (0,2.5) -- (0,0) ;
\draw (0,-0.1) node[diff] {};
\draw (0,2.5) node[xies] {};
}
\DeclareSymbol{Itodiff}{2}{
\draw(0,2.9) -- (0,-0.2);
\draw (0,2.9) node[diff] {};
\draw (0,-0.1) node[xies] {};
}

\DeclareSymbol{diffgeneric}{1}{
\draw  (0,2.5) -- (0,0) ;
\draw (0,-0.1) node[diff] {};
\draw (0,2.5) node[xi] {};
}

\DeclareSymbol{genericdiff}{2}{
\draw(0,2.9) -- (0,-0.2);
\draw (0,2.9) node[diff] {};
\draw (0,-0.1) node[xi] {};
}

\DeclareSymbol{diffdot}{2}{
\draw  (0,3) -- (0,-0.1) ;
\draw (0,3) node[not] {};
\draw (0,-0.1) node[diff] {};
}

\DeclareSymbol{diffdotmini}{0}{
\draw  (0,0) -- (0,1.2) ;
\draw (0,1.2) node[not] {};
\draw (0,0) node[diffmini] {};
}

\DeclareSymbol{dotdiff}{2}{
\draw[kernelsmod]  (0,3) -- (0,-0.1) ;
\draw (0,3) node[diff] {};
\draw (0,-0.1) node[not] {};
}

\DeclareSymbol{3}{-2}{\draw (0,-0.25) node[xi] {} -- (-1,1) node[xi] {};}
\DeclareSymbol{AAA}{-0.5}{\draw (0,0) node[xi] {} -- (-1,1) node[xi] {} -- (0,2) node[xi] {};
	\draw (0,0) node[xi] {} -- (-1,-1) node[xi] {};}
\DeclareSymbol{AAM}{0.5}{\draw (-1,1)  -- (0,2) node[xi] {};
	\draw (-1,1) node[xi] {} -- (0,0) node[xi] {} -- (1,1) node[xi] {};}
\DeclareSymbol{AMA}{-4}{\draw (-1,1) node[xi] {} -- (0,0) node[xi] {} -- (1,1) node[xi] {};
	\draw (0,0) node[xi] {} -- (0,-1.4) node[xi] {};}
\DeclareSymbol{AMM}{0.5}{\draw (0,0)  -- (0,1.4) node[xi] {};
	\draw (-1,1) node[xi] {} -- (0,0) node[xi] {} -- (1,1) node[xi] {};}

\DeclareSymbol{dotdiff1}{2}{
\draw[kernelsmod]  (0,3) -- (0,-0.1) ;
\draw (0,3) node[diff1] {};
\draw (0,-0.1) node[not] {};
}

\DeclareSymbol{dotdiff1mini}{0}{
\draw[kernelsmod]  (0,1.2) -- (0,0) ;
\draw (0,1.2) node[diffmini] {};
\draw (0,0) node[not] {};
}

\DeclareSymbol{dotdiff2}{2}{
\draw (0,3) -- (0,-0.1) ;
\draw (0,3) node[diff] {};
\draw (0,-0.1) node[not] {};
}

\DeclareSymbol{dotdiff2mini}{0}{
\draw (0,1.2) -- (0,0) ;
\draw (0,1.2) node[diffmini] {};
\draw (0,0) node[not] {};
}

\DeclareSymbol{dotdiffstraight}{0}{
\draw  (0,3) -- (0,-0.1) ;
\draw (0,3) node[diff] {};
\draw (0,-0.1) node[not] {};
}

\DeclareSymbol{arbre1}{0}{
\draw  (0,0) -- (1.5,1.5) ;
\draw (1.5,1.5) node[not] {};
\draw (0,0) node[not] {};
}

\DeclareSymbol{arbre2}{0}{
\draw  (0,0) -- (1.5,1.5) ;
\draw[kernelsmod] (0,0) -- (-1.5,1.5);
\draw (1.5,1.5) node[not] {};
\draw (0,0) node[not] {};
\draw (-1.5,1.5) node[xi] {};
}

\DeclareSymbol{arbre3}{0}{
\draw  (0,0) -- (1.5,1.5) ;
\draw[kernelsmod] (1.5,1.5) -- (0,3);
\draw (0,0) node[not] {};
\draw (1.5,1.5) node[not] {};
\draw (0,3) node[xi] {};
}

\DeclareSymbol{treeeval}{0}{
\draw (0,0) -- (1,1);
\draw (0,0) node[xi] {};
\draw (1.25,1.25) node[xi] {};
\draw (-0.6,0.6) node[]{\tiny{$i$}};
\draw (0.65,1.85) node[]{\tiny{$j$}};
}

\DeclareSymbol{testeval}{0}{
\draw (0,0) -- (1,1);
\draw (0,0) -- (-1,1);
\draw (0,0) node[xi] {};
\draw (1.25,1.25) node[xi] {};
\draw (-1.25,1.25) node[xi] {};
\draw (-0.6,-0.6) node[]{\tiny{$i$}};
\draw (0.65,1.85) node[]{\tiny{$j$}};
\draw (-1.95,1.85) node[]{\tiny{$k$}};
}

\DeclareSymbol{treeeval2}{0}{
\draw[kernelsmod] (-0.25,-1) -- (1,0.5) ;
\draw[kernelsmod] (1,0.5) -- (-0.25,2);
\draw (1,0.5) node[diff2] {};
\draw (-0.25,-1) node[not] {};
\draw (-0.25,2) node[xi] {};
\draw (-0.6,1.2) node[]{\tiny{1}};
}

\DeclareSymbol{arbreact}{1}{
\draw (0,0) node[not] {};
\draw[kernelsmod] (0,0) -- (1,1);
\draw[kernelsmod] (0,0) -- (-1,1);
\draw (-1,1) node[xic] {};
\draw  (0,2) -- (1,1) ;
\draw (0,2) node[xic] {};
\draw (1,1) node[xi] {};
}

\DeclareSymbol{arbreact1}{0}{
\draw (0,-1.5) -- (0,0);
\draw[kernelsmod] (0,0) -- (1,1);
\draw[kernelsmod] (0,0) -- (-1,1);
\draw  (0,2) -- (1,1) ;
\draw (0,-1.5) node[diff] {};
\draw (0,0) node[not] {};
\draw (-1,1) node[xic] {};
\draw (0,2) node[xic] {};
\draw (1,1) node[xi] {};
}

\DeclareSymbol{arbreact2}{0}{
\draw (0,-0.75) -- (-1,0.5); 
\draw (0,-0.75) -- (1,0.5);
\draw (0,1.5) -- (1,0.5);
\draw (0,1.5) node[xic] {};
\draw (1,0.5) node[xi] {};
\draw (-1,0.5) node[xic] {};
\draw (0,-0.75) node[diff] {};
}

\DeclareSymbol{arbreact3}{0}{
\draw[kernelsmod] (0,-0.75) -- (-1,0.5); 
\draw[kernelsmod] (0,-0.75) -- (1,0.5);
\draw (0,1.75) -- (1,0.5);
\draw (2,1.75) -- (1,0.5);
\draw (0,1.75) node[xic] {};
\draw (1,0.5) node[diff] {};
\draw (-1,0.5) node[xic] {};
\draw (2,1.75) node[xi] {};
\draw (0,-0.75) node[not] {};
}

\DeclareSymbol{pre_im_1}{0}{
\draw[kernels2] (0,-0.5) node[not] {} -- (-0.6,0.5) ;
\draw[kernels2] (0,-0.5) -- (0.6,0.5);
\draw (0,1.1)  -- (-0.55,2);
\draw (0,1.1)  -- (0.55,2);
\draw (0,0.7) node[g] {};
\draw (0,2.2) node[g] {};
}

\DeclareSymbol{disconnect}{0}{
\draw[kernels2] (0,-0.5) node[not] {} -- (-0.6,0.5) ;
\draw[kernels2] (0,-0.5) -- (0.6,0.5);
\draw (-0.55,1.1)  -- (-0.55,2.3);
\draw (0.55,2.3) -- (0.55,1.5) -- (1.2,1.5) -- (1.2,3.5) -- (0.55,3.5) -- (0.55,2.7);
\draw (0,0.7) node[g] {};
\draw (0,2.5) node[g] {};
}

\DeclareSymbol{pre_im_2}{2}{\draw[kernels2] (0,0) node[not] {} -- (-1,1) node[not] {};
\draw[kernels2] (0,0) -- (1,1) node[not] {};
\draw[kernels2] (-1,1) -- (-1.5,2.5);
\draw[kernels2] (-1,1) -- (-0.5,2.5);
\draw[kernels2] (1,1) -- (0.5,2.5);
\draw[kernels2] (1,1) -- (1.5,2.5);
\draw (-1,2.7) node[g] {};
\draw (1,2.7) node[g] {};
}

\DeclareSymbol{CX_rec}{0}{
\draw [black] (-0.3,1) to (-0.3,-0.3);
\draw [black] (0.3,1) to (0.3,-0.3);
\draw [black] (-0.3,1) to (-0.3,2.3);
\draw [black] (0.3,1) to (0.3,2.3);
\draw (0,1) node[rec] {};
}

\DeclareSymbol{CX_cerc}{0}{
\draw [black] (0,1) to (0,-0.3);
\draw (0,1) node[cerc] {};
}

\DeclareSymbol{proof0}{0}{
\draw (0,-3) node[] {\tiny$\tau_0$};
\draw (0,-2.3) -- (0,0.5);`
\draw (0,0.5) -- (-1.5,2.5);
\draw (0,0.5) -- (1.5,2.5);
\draw (-1,-2) node[] {\tiny$v$};
\draw (-1.5,3.1) node[] {\tiny$\tau_1$};
\draw (1.5,3.1) node[] {\tiny$\tau_2$};
\draw (0,0.5) node[diff] {};
}

\DeclareSymbol{proof0b}{0}{
\draw (0,0.5) -- (-1.5,2.5);
\draw (0,0.5) -- (1.5,2.5);
\draw (-1.5,3.1) node[] {\tiny$\tau_1$};
\draw (1.5,3.1) node[] {\tiny$\tau_2$};
\draw (0,0.5) node[diff] {};
}

\DeclareSymbol{proof}{0}{
\draw[kernelsmod] (-2,3) -- (0,0);
\draw[kernelsmod] (2,3) -- (0,0);
\draw (0,0) node[not] {};
}

\DeclareSymbol{prooftri}{0}{
\draw[kernelsmod] (-2,3) -- (0,0);
\draw[kernelsmod] (0,4) -- (0,0);
\draw (2,2.7)--(0,0);
\draw (0,0) node[not] {};
}

\DeclareSymbol{proofqua}{0}{
\draw[kernelsmod] (-3,3) -- (0,0);
\draw[kernelsmod] (-1,4) -- (0,0);
\draw (1,3.6)--(0,0);
\draw (3,2.7)--(0,0);
\draw (0,0) node[not] {};
}

\DeclareSymbol{proof1_1}{0}{
\draw (0,-2.7) node[] {\tiny$\tau_0$};
\draw (0,-2) -- (0,0.5);`
\draw[kernelsmod] (0,0.5) -- (-1.5,2.5); 
\draw[kernelsmod] (0,0.5) -- (1.5,2.5);
\draw (-1,-1.7) node[] {\tiny$v$};
\draw (-1.5,3.1) node[] {\tiny$\tau_1$};
\draw (1.5,3.1) node[] {\tiny$\tau_2$};
\draw (0,0.5) node[not] {};
}

\DeclareSymbol{proof1b_1}{0}{
\draw[kernelsmod] (0,0.5) -- (-1.5,2.5); 
\draw[kernelsmod] (0,0.5) -- (1.5,2.5);
\draw (-1.5,3.1) node[] {\tiny$\tau_1$};
\draw (1.5,3.1) node[] {\tiny$\tau_2$};
\draw (0,0.5) node[not] {};
}

\DeclareSymbol{proof1_2}{0}{
\draw (0,-2.7) node[] {\tiny$\tau_0$};
\draw[kernelsmod] (0,-2) -- (0,0.5);`
\draw[kernelsmod] (0,0.5) -- (-1.5,2.5); 
\draw[kernelsmod] (0,0.5) -- (1.5,2.5);
\draw (-1,-1.7) node[] {\tiny$v$};
\draw (-1.5,3.1) node[] {\tiny$\tau_1$};
\draw (1.5,3.1) node[] {\tiny$\tau_2$};
\draw (0,0.5) node[not] {};
}

\DeclareSymbol{proof2_1}{0}{
\draw (0,-2.7) node[] {\tiny$\tau_0$};
\draw (0,-2) -- (-1.8,-0.3);
\draw (-1.8,0.7) -- (0,2.5); 
\draw (1,-1.7) node[] {\tiny$v$};
\draw (-2.5,1.2) node[] {\tiny$r_1$};
\draw (0,3.1) node[] {\tiny$\tau_2$};
\draw (-2.5,0) node[] {\tiny$\tau_1$};
}

\DeclareSymbol{proof2b_1}{0}{
\draw (-1.8,0.7) -- (0,2.5); 
\draw (-2.5,1.2) node[] {\tiny$r_1$};
\draw (0,3.1) node[] {\tiny$\tau_2$};
\draw (-2.5,0) node[] {\tiny$\tau_1$};
}

\DeclareSymbol{proof2b_2}{0}{
\draw (-1.8,0.7) -- (0,2.5); 
\draw (-2.5,1.2) node[] {\tiny$r_2$};
\draw (0,3.1) node[] {\tiny$\tau_1$};
\draw (-2.5,0) node[] {\tiny$\tau_2$};
}

\DeclareSymbol{proof2_2}{0}{
\draw (0,-2.7) node[] {\tiny$\tau_0$};
\draw[kernelsmod] (0,-2) -- (-1.8,-0.3);
\draw (-1.8,0.7) -- (0,2.5); 
\draw (1,-1.7) node[] {\tiny$v$};
\draw (-2.5,1.2) node[] {\tiny$r_1$};
\draw (0,3.1) node[] {\tiny$\tau_2$};
\draw (-2.5,0) node[] {\tiny$\tau_1$};
}

\DeclareSymbol{proof3_1}{0}{
\draw (0,-2.7) node[] {\tiny$\tau_0$};
\draw (0,-2) -- (1.8,-0.3);
\draw (1.8,0.7) -- (0,2.5); 
\draw (-1,-1.7) node[] {\tiny$v$};
\draw (2.5,1.2) node[] {\tiny$r_2$};
\draw (0,3.1) node[] {\tiny$\tau_1$};
\draw (2.5,0) node[] {\tiny$\tau_2$};
}

\DeclareSymbol{proof3_2}{0}{
\draw (0,-2.7) node[] {\tiny$\tau_0$};
\draw[kernelsmod] (0,-2) -- (1.8,-0.3);
\draw (1.8,0.7) -- (0,2.5); 
\draw (-1,-1.7) node[] {\tiny$v$};
\draw (2.5,1.2) node[] {\tiny$r_2$};
\draw (0,3.1) node[] {\tiny$\tau_1$};
\draw (2.5,0) node[] {\tiny$\tau_2$};
}

\DeclareSymbol{proof4_1}{0}{
\draw (0,-2) node[] {\tiny$\tau_0$};
\draw (0,-1.3) -- (-1.5,1.8); 
\draw  (0,-1.3) -- (1.5,1.8);
\draw (-1,-1) node[] {\tiny$v$};
\draw (-1.5,2.4) node[] {\tiny$\tau_1$};
\draw (1.5,2.4) node[] {\tiny$\tau_2$};
}

\DeclareSymbol{proof4_2}{0}{
\draw (0,-2) node[] {\tiny$\tau_0$};
\draw[kernelsmod] (0,-1.3) -- (-1.5,1.8); 
\draw  (0,-1.3) -- (1.5,1.8);
\draw (-1,-1) node[] {\tiny$v$};
\draw (-1.5,2.4) node[] {\tiny$\tau_1$};
\draw (1.5,2.4) node[] {\tiny$\tau_2$};
}

\DeclareSymbol{proof4_3}{0}{
\draw (0,-2) node[] {\tiny$\tau_0$};
\draw (0,-1.3) -- (-1.5,1.8); 
\draw[kernelsmod]  (0,-1.3) -- (1.5,1.8);
\draw (-1,-1) node[] {\tiny$v$};
\draw (-1.5,2.4) node[] {\tiny$\tau_1$};
\draw (1.5,2.4) node[] {\tiny$\tau_2$};
}

\DeclareSymbol{prooftriple}{0}{
\draw (0,-2.7) node[] {\tiny$\tau_0$};
\draw (0,-2) -- (0,0.25);`
\draw (0,0.25) -- (-1.5,2.5); 
\draw (0,0.25) -- (1.5,2.5);
\draw (0,0.25) -- (0,2.5);
\draw (-1,-1.7) node[] {\tiny$v$};
\draw (-2,3.1) node[] {\tiny$\tau_1$};
\draw (0,3.1) node[] {\tiny$\tau_2$};
\draw (2,3.1) node[] {\tiny$\tau_3$};
\draw (0,0.25) node[diff] {};
}

\DeclareSymbol{prooftriple1}{0}{
\draw (0,-2.7) node[] {\tiny$\tau_0$};
\draw (0,-2) -- (0,0.25);
\draw[kernelsmod] (0,0.25) -- (-1.5,2.5); 
\draw (0,0.25) -- (1.5,2.5);
\draw[kernelsmod] (0,0.25) -- (0,2.5);
\draw (-1,-1.7) node[] {\tiny$v$};
\draw (-2,3.1) node[] {\tiny$\tau_1$};
\draw (0,3.1) node[] {\tiny$\tau_2$};
\draw (2,3.1) node[] {\tiny$\tau_3$};
\draw (0,0.25) node[not] {};
}

\DeclareSymbol{prooftripleperm1}{0}{
\draw (0,-2.7) node[] {\tiny$\tau_0$};
\draw (0,-2) -- (0,0.25);
\draw[kernelsmod] (0,0.25) -- (-3.5,2.5); 
\draw (0,0.25) -- (2.5,2.5);
\draw[kernelsmod] (0,0.25) -- (0,2.5);
\draw (-1,-1.7) node[] {\tiny$v$};
\draw (-3,3.1) node[] {\tiny$\tau_{\sigma_1}$};
\draw (0,3.1) node[] {\tiny$\tau_{\sigma_{\!2}}$};
\draw (3,3.1) node[] {\tiny$\tau_{\sigma_3}$};
\draw (0,0.25) node[not] {};
}

\DeclareSymbol{proofdouble}{0}{
\draw (0,0.5) -- (-1.5,2.5);
\draw (0,0.5) -- (1.5,2.5);
\draw (-1.5,3.1) node[] {\tiny$\tau_1$};
\draw (1.5,3.1) node[] {\tiny$\tau_2$};
\draw (0,0.5) node[diff] {};
}

\DeclareSymbol{proofquadruple}{0}{
\draw (0,0) -- (-2.5,2.5); 
\draw (0,0) -- (-1,2.5);
\draw (0,0) -- (1,2.5);
\draw (0,0) -- (2.5,2.5);
\draw (-3,3.1) node[] {\tiny$\tau_1$};
\draw (-1,3.1) node[] {\tiny$\tau_2$};
\draw (1,3.1) node[] {\tiny$\tau_3$};
\draw (3,3.1) node[] {\tiny$\tau_4$};
\draw (0,0) node[diff] {};
}

\DeclareSymbol{proofquadruple1}{0}{
\draw[kernelsmod] (0,0) -- (-2.5,2.5); 
\draw[kernelsmod] (0,0) -- (-1,2.5);
\draw (0,0) -- (1,2.5);
\draw (0,0) -- (2.5,2.5);
\draw (-3,3.1) node[] {\tiny$\tau_1$};
\draw (-1,3.1) node[] {\tiny$\tau_2$};
\draw (1,3.1) node[] {\tiny$\tau_3$};
\draw (3,3.1) node[] {\tiny$\tau_4$};
\draw (0,0) node[not] {};
}

\DeclareSymbol{proofquadrupleperm1}{0}{
\draw[kernelsmod] (0,-0.4) -- (-3.5,2.3); 
\draw[kernelsmod] (0,-0.4) -- (-1,2.3);
\draw (0,-0.4) -- (1,2.5);
\draw (0,-0.4) -- (3.5,2.5);
\draw (-4.5,3.1) node[] {\tiny$\tau_{\sigma_1}$};
\draw (-1.5,3.1) node[] {\tiny$\tau_{\sigma_2}$};
\draw (1.5,3.1) node[] {\tiny$\tau_{\sigma_3}$};
\draw (4.5,3.1) node[] {\tiny$\tau_{\sigma_4}$};
\draw (0,-0.4) node[not] {};
}

\DeclareMathAlphabet{\mathpzc}{OT1}{pzc}{m}{it}

\let\d\partial
\let\eps\varepsilon

\def\eqref#1{(\ref{#1})}

\def\Ito{{\text{\rm\tiny It\^o}}}

\def\geo{{\text{\rm\tiny geo}}}

\def\nice{{\text{\rm\tiny nice}}}

\makeatletter 
\newcommand*{\bigcdot}{}
\DeclareRobustCommand*{\bigcdot}{%
  \mathbin{\mathpalette\bigcdot@{}}%
}
\newcommand*{\bigcdot@scalefactor}{.5}
\newcommand*{\bigcdot@widthfactor}{1.15}
\newcommand*{\bigcdot@}[2]{%
  \sbox0{$#1\vcenter{}$}
  \sbox2{$#1\cdot\m@th$}%
  \hbox to \bigcdot@widthfactor\wd2{%
    \hfil
    \raise\ht0\hbox{%
      \scalebox{\bigcdot@scalefactor}{%
        \lower\ht0\hbox{$#1\bullet\m@th$}%
      }%
    }%
    \hfil
  }%
}
\makeatother

\def\BPHZ{\textnormal{\tiny \textsc{bphz}}}

\def\Moll{\mathrm{Moll}}

\tcbset
{colframe=boxcolor,colback=symbols!7!pagebackground,coltext=pageforeground,
fonttitle=\bfseries,nobeforeafter,center title,size=fbox,boxsep=1.5pt,
top=0mm,bottom=0mm,boxsep=0mm,tcbox raise base}

\def\two{{\<generic>\kern0.05em\<genericb>}}
\def\twoI{{\<Ito>\kern0.05em\<Itob>}}

\def\mail#1{\burlalt{#1}{mailto:#1}}

\begin{document}

\title{Symmetries for the gKPZ equation via multi-indices}
\author{Carlo Bellingeri$^1$, Yvain Bruned$^2$}
\institute{IRIMAS, 18 Rue des Frères Lumière, 68200 
Mulhouse, France \and Universite de Lorraine, CNRS, IECL, F-54000 Nancy, France
\\
Email: \begin{minipage}[t]{\linewidth}
\mail{carlo.bellingeri@uha.fr},\\
\mail{yvain.bruned@univ-lorraine.fr}
\end{minipage}}

\maketitle

\begin{abstract}
In this work, we study the two main symmetries for the one-dimensional generalised KPZ equation (gKPZ): the chain rule and the Itô Isometry. We consider the equation in the full subcritical regime and use multi-indices that avoid an over-parametrisation of the renormalised equation to compute the dimension of the two spaces associated with these two symmetries. Our proof is quite elementary and shows that multi-indices provide in this case a simplification in comparison to the results obtained via decorated trees. It also completes the program on the study of the chain rule initiated in \cite{BGHZ} and continued in \cite{BD24}.
\end{abstract}

\setcounter{tocdepth}{2}
\tableofcontents

\section{Introduction}

In this paper, we consider the symmetries of the following equation:
\begin{equs} \label{eq:gen_KPZ}
\partial_t u   = \partial_x^{2} u+ \Gamma(u) (\partial_x u)^{2} +g(u)\partial_x u + h(u)+  \sigma(u) \xi\,,
\end{equs}
which we call the generalised-KPZ equation. Here, $ (t,x) \in \mathbb{R}_+ \times  \mathbb{T}$,  where $ \mathbb{T} $ is the one dimensional torus and the non-linearities $\sigma,\Gamma, h, g\in \mathcal{C}^{\infty}(\mathbb{R}$ are smooth functions and the noise $\xi$ is a random distributional admissible noise whose trajectories belong a.s. to the negative H\"older space $\mathcal{C}^{-2+\delta}(\mathbb{R}_+ \times  \mathbb{T})$ for some $0<\delta< 1$ (i.e. in the subcritical regime), whose law is space symmetric, see Definition~\ref{defn_admiss_noise} and Definition~\ref{symmetric_noise}. Via a standard scaling argument, one expects the regularity of the solution to be that of the linear additive stochastic heat equation driven by $\xi$,
\[
\partial_t u   =\partial_x^{2} u+ \xi\,,
\]
which has the formal mild solution $u = \mathcal{P} * \xi$, where $\mathcal{P}$ is the heat kernel on $\mathbb{T}$ and $*$ denotes space-time convolution over $\mathbb{R}_+ \times \mathbb{T}$. The trajectories of this solution belong to $\mathcal{C}^{\delta}(\mathbb{R}_+ \times \mathbb{T})$, gaining $+2$ orders of regularity from the Schauder estimate as the space-time white noise is convolved with the heat kernel. Consequently, the trajectories of $\partial_x u$ lie in $\mathcal{C}^{\delta-1}(\mathbb{R}_+ \times \mathbb{T})$, which is a space-time distribution. As a result, the product $(\partial_x u)^2$ is ill-defined, making \eqref{eq:gen_KPZ} a singular stochastic partial differential equation (SPDE).

 This equation is solved in the subcritical regime via the theory of Regularity Structures. It was one of the main motivations for developing a self contained solution theory in a series of papers \cite{reg,BHZ,CH,BCCH} that cover a large class of singular SPDEs. Surveys on these results can be found in \cite{FrizHai,BH20}. By subcritical, we mean that the regularity of the noise is such that one can apply the theory of Regularity Structures. Indeed, the regularity of the noise governs the length of local approximations of the solution. This length increases when the regularity of the noise decreases. When it reaches a critical value, it becomes infinite. In the context of gKPZ, one has to have the Hölder space-time regularity of the noise $ \xi $ to be greater than $  -2$. For a more precise description, one can look at Definition~\ref{defn_admiss_noise}.

  The well-posedness of this equation 
can be formulated via the convergence of a properly given family of regularised and renormalised equations. Space-time noise $  \xi$ is replaced by a smoothened version $ \xi_{\eps} $ that converges to $ \xi $ when $\eps$ converges to zero. This is performed by convolution via a class of mollifiers defined precisely at the beginning of Section~\ref{section_2.2}.  Then, one writes a renormalised equation of the form:
\begin{equs}[eq:renorm nonlocal intro1]
	\d_t u_{\eps} & = \d_x^2 u_{\eps} + \Gamma(u_{\eps})\,(\d_x u_{\eps})^2
	+ g(u_{\eps})\,\d_x u_{\eps}
	+h(u_{\eps}) + \sigma(u_{\eps})\, \xi_{\eps}\; \\ & + \sum_{\tau \in \mathfrak{B}_{ \<derivative2>, \<generic>}^-} C_{\eps}(\tau)  \Upsilon_{\Gamma,\sigma}(\tau)(u_{\eps})\,.
\end{equs}
where $\mathfrak{B}_{ \<derivative2>, \<generic>}^-$ is a combinatorial set formed of decorated trees whose cardinality depends on the regularity of the noise. The coefficients (elementary differentials)  $ \Upsilon_{\Gamma,\sigma}(\tau)(u_{\eps}) $  depend on $ \Gamma, \sigma $ and the solution $ u_{\eps}$. The counter-terms described by $\mathfrak{B}_{ \<derivative2>, \<generic>}^-$ are needed for renormalising the various distributional products of the equation. The solution $u_\eps$ of the random PDEs \eqref{eq:renorm nonlocal intro1} converges as $\eps\to 0$ in probability, locally in time, to a nontrivial limit $u$.
In the end, from the applications of the results contained in \cite{reg,BHZ,CH,BCCH}, one gets a finite dimensional space of solutions with the choice of the parameters $ C_{\eps}(\tau) $. Indeed, each $ C_{\eps}(\tau) $ can be split into $ C_{\eps}(\tau) = \hat{C}_{\eps}(\tau) + C(\tau) $ where $ \hat{C}_{\eps}(\tau) $ is a potentially diverging constant needed for the convergence and $C(\tau) $ is a finite constant that corresponds to a degree of freedom. Then by the choice of the parameters $ C_{\eps}(\tau) $, we mean the choice of the finite part of these constants that is $C(\tau)$. 

 The goal of this paper is then to introduce proper symmetries among this family of solutions in order to reduce the dimensions of these vector spaces and reduce as far as possible the freedom  of choice to get a relevant notion of  solution. The two key symmetries studied for this equation are the following:
\begin{itemize}
	\item \textbf{Chain rule symmetry:} The equation \eqref{eq:gen_KPZ} transforms covariantly under composition with a diffeomorphism $\phi \colon \mathbb{R} \to \mathbb{R}$, in the sense that for any limit solution $u$ to the regularized equation \eqref{eq:renorm nonlocal intro1} then $\phi \circ u$ formally satisfies an equation of the same type, with coefficients modified according to the standard chain rule.
	
	\item \textbf{Itô isometry symmetry:} In case the law of $\xi$ is a stationary Gaussian distribution, the law of any limit solution $u$ to the regularized equation \eqref{eq:renorm nonlocal intro1} depends on the function $\sigma$ only through its square $\sigma^2$. In other words, the equation is invariant under the transformation $\sigma \mapsto -\sigma$, reflecting the symmetry property of Gaussian random variables.
\end{itemize}

This idea of studying symmetries was already applied to the geometric stochastic heat equation in \cite{BGHZ}, a system of SPDEs taking values in a compact Riemannian manifold. In local coordinates, after isometrically embedding the manifold into $\mathbb{R}^d$ and using the Einstein summation convention, this equation takes the form
\begin{equation}\label{high_dimension_equation}
\partial_t u^{\alpha} = \partial_{xx} u^{\alpha}+ \Gamma^{\alpha}_{\beta \gamma}(u)\, \partial_x u^{\beta} \partial_x u^{\gamma} + g^\alpha_{\beta}(u)\,\partial_x u^{\beta}+ h^{\alpha}(u)+\sigma^{\alpha}_j(u)\, \zeta^j\,,
\end{equation}
where $\alpha, \beta, \gamma \in \{1, \ldots , d\}$, $j \in \{1, \ldots, m\}$, and $\Gamma=\{\Gamma^{\alpha}_{\beta\gamma}\}$, $\sigma=\{\sigma^{\alpha}_j\}$, $g=\{g^{\alpha}_{\beta}\}$, and $h=\{h^{\alpha}\}$ are given smooth functions. Moreover, $\zeta = (\zeta^1, \dots, \zeta^m)$ denotes $m$ i.i.d. copies of a space-time white noise on $\mathbb{R}_+ \times \mathbb{T}$. The geometric motivation for equation \eqref{high_dimension_equation} is to construct a natural Langevin dynamics taking values in the space of loops in a compact Riemannian manifold.  A first attempt at this construction appears in \cite{Fun92} with $\zeta$ a noise white in time and coloured in space. The construction for space--time white noise is carried out in \cite{BGHZ} and announced in \cite{proc}. 

 Even in this general context, the theory of regularity structures can be applied to give a proper family of solutions with the same freedom in the choice of the underlying parameters. Nevertheless, the same fundamental symmetries, the chain rule and the Itô isometry symmetry, can be formulated in this higher-dimensional setting (the latter involving the quadratic form $\sigma^{\!T}\sigma$) and a detailed description of the degrees of freedom for this equation has been provided. Below we summarize the actual state of the art on the knowledge of these symmetries, relating it with our  paper.

\begin{center}
	\begin{tabular}{|c | c | c|c|c|}
\hline	& 	\multicolumn{2}{c|}{Space--time white noise} & \multicolumn{2}{c|}{Full subcritical regime} 
		\\
		\hline  & \text{Chain rule}  & \text{It\^o isometry} &  \text{Chain rule} & \text{It\^o isometry}
		\\
		\hline
		\text{Dimension one} & \text{This work} &  \text{This work} & \text{This work} & \text{This work} \\
		\hline
		\text{High dimension} & \cite{BGHZ} & \cite{BGHZ} & \cite{BD24} & \text{Open}
		\\
		\hline
	\end{tabular}
\end{center}
As can be seen immediately, the description of the symmetries in \cite{BGHZ} excludes the case $d = 1$, $m = 1$ (see \cite[Remark 1.9]{BGHZ}), which is precisely the case we treat here. Instead, their results hold in what we refer to as the high‑dimensional regime, i.e., when $d$ and $m$ are sufficiently large. In that regime, one can indeed apply a general injectivity result for a suitably extended evaluation map introduced in Definition \ref{def:ev_map}, see \cite[Theorem 5.25, 5,31]{BGHZ}, \cite[Remark 4.8]{BD24}. Let us also mention that in \cite{BGHZ}, by combining the two symmetries, one obtains uniqueness of solutions in a specific geometric setting. In general, however, one does not expect to recover a unique choice of solutions by combining these two symmetries alone. The work \cite{BGHZ} provides explicit combinatorial conditions for both symmetries. The proofs there for computing the vector space of solutions associated with these symmetries are carried out by hand in the space–time white noise setting and do not extend to the full subcritical regime.

In \cite{BD24}, using advanced algebraic tools such as operad theory and homological algebra, one can interpret the key map for the chain rule symmetry as a derivation in a homological complex. Its kernel—which characterizes this symmetry—can then be computed. This structure can be understood as a deformation of the so-called operadic twisting. The Itô isometry, however, remains an open problem in the high‑dimensional context. The results from \cite{BGHZ,BD24} can be applied in one dimension, but they lead to an over‑parametrization of the equation due to the loss of injectivity in the coefficients.

The study of the chain rule in one dimension was initiated in \cite{Bru}, where its dimension is derived formally for space–time white noise. We also mention \cite{wong}, which recovers the Itô product for the multiplicative stochastic heat equation, and \cite{Bel20}, where an Itô formula with KPZ‑type nonlinearities is derived for the additive stochastic heat equation.

It turns out that to preserve the crucial property of injectivity, one must replace decorated trees with a new combinatorial set called \emph{multi‑indices}. Multi‑indices were introduced in the context of singular SPDEs in \cite{OSSW} for studying quasilinear equations. Since then, works involving multi‑indices have appeared both on the analytical/probabilistic and algebraic sides \cite{LOT,LOTT,BK23,JZ,BL23,BD23,Li23,GT,GMZ24,BH24}. For comprehensive surveys on multi‑indices, we refer the interested reader to \cite{LO23,OST,BOT24}.

Multi‑indices have appeared in various forms in earlier literature. In \cite{DL}, the authors characterize the free Novikov algebra using multi‑indices. In numerical analysis, the authors of \cite{MV16} used composition maps to characterize affine‑equivariant methods; these maps are connected to multi‑indices (see also \cite{BEFH24}, where multi‑index $B$-series are introduced and this connection is made explicit).

For intermediate dimensions (between one and high), no combinatorial set is known that preserves injectivity while retaining the nice algebraic properties that multi‑indices and decorated trees offer for constructing solutions to singular SPDEs. Therefore, we omit this case in the summary of results that follows. We can now state our main theorem, which employs multi‑indices for the generalized KPZ equation.

\begin{theorem}
	\label{thm:main renormalisation_intro}
Let $u_0$ be an initial datum belonging to the H\"older space $\CC^r(\mathbb{T})$ for some $r>0$, and let $\xi$ be an admissible noise that is spatially symmetric (see Definition \ref{defn_admiss_noise} and Definition \ref{symmetric_noise}). For a given mollifier $\rho$ in the class $\Moll$ (see Section \ref{section_2.2}), there exists a finite combinatorial set $\mathfrak{B}_{\xi}$ (given in Definition~\ref{def_B_xi}), whose dimension as a vector space can be computed, and a finite collection of renormalisation constants $C_{\eps}(v)$ (depending also on the noise $\xi$, $\rho$, and $\eps>0$) such that the following holds. Consider the renormalised equation
	\begin{equs}[eq:renorm nonlocal intro2]
		\d_t u_{\eps} & = \d_x^2 u_{\eps} + \Gamma(u_{\eps})\,(\d_x u_{\eps})^2
		+ g(u_{\eps})\,\d_x u_{\eps}
		+ h(u_{\eps}) + \sigma(u_{\eps})\, \xi_{\eps}\; \\ & \qquad + \sum_{v \in  \mathfrak{B}_{\xi}} C_{\eps}(v)  \Upsilon_{\Gamma,\sigma}(v)(u_{\eps})\,,
	\end{equs}
	where $\xi_{\eps}$ denotes the regularisation of $\xi$ by $\rho$ with some parameter $\eps>0$, and the terms $ \Upsilon_{\Gamma,\sigma}(v)(u_{\eps}) $ are computed using $\sigma$ together with the covariant derivative $\nabla_X Y$ defined for $X,Y \in \mathcal{C}^{\infty}(\mathbb{R},\mathbb{R})$ by  
	\begin{equ}\label{eq:def_cov_der}
		\nabla_X Y (u) = X(u)\,\partial_u Y(u) + \Gamma(u) \,X(u) \,Y(u)\;.
	\end{equ}
	Then one has the following properties:
\begin{enumerate}
\item  The family $\{u_{\eps}\}_{\eps>0}$ converges in probability, locally in time, as $\eps\to 0$ to a nontrivial limit $u$ that is independent of the mollifier $\rho$.
\item The limiting process $u$ transforms according to the chain rule under composition with diffeomorphisms.
\item When $\xi$ is a stationary Gaussian noise, the law of the limit solution depends on function $\sigma$ only through $\sigma^2$.
\end{enumerate}	  
\end{theorem}
\begin{remark}
 Unlike in \cite{BD24}, where the dimension of the linear space generated by $\mathfrak B_{\xi}$ lacks a direct expression, in our case Theorem~\ref{dim_geo} allows us to relate it immediately to the dimension of the homogeneous elements of the graded vector space $\mathcal M_{\<generic>}$ defined in~\eqref{free_nov_1}. As shown in \cite{DL} and more generally in \cite{BD23}, this vector space is isomorphic to the free Novikov algebra on one generator, for which explicit generating series are known; see \cite[Thm.~7.8]{DL}.

This distinction between  decorated trees and multi-indices has a clear algebraic explanation. In one dimension, the covariant derivative in \eqref{eq:def_cov_der} consists of two terms with a clear algebraic interpretation: the term $X \triangleright Y = X(u),\partial_u Y(u)$, which yields a Novikov algebra, plus a perturbation $\Gamma(u) X(u), Y(u)$. This perturbation can also be described within the Novikov framework that is, inside a vector space of multi-indices as detailed in the definition of the covariant derivative in \eqref{covariant_derivative}.

In the context of equation \eqref{high_dimension_equation}, however, the situation is markedly different. Here, the covariant derivative is given by
\begin{equs}
(\nabla_X Y)^{\alpha}(u)
= X^{\beta}(u),\partial_{u^{\beta}} Y^{\alpha}(u)
+ \Gamma^{\alpha}_{\beta\gamma}(u),X^{\beta}(u),Y^{\gamma}(u),
\end{equs}
where $\alpha,\beta,\gamma \in {1,\dots,m}$. The first term produces the free pre-Lie algebra for sufficiently large $d$ and $m$ (see \cite{ChaLiv} for a characterization of this free pre-Lie structure), but the second term breaks this pre-Lie structure. This observation suggests that the multi-index formalism, which is natural for the one-dimensional case, could lead to deeper insights into the structure of Novikov algebras and their geometric realisations.
\end{remark}
\begin{remark}
As there is currently no analytic solution theory for the equation \eqref{eq:gen_KPZ} formulated in terms of multi-indices, the proof of Theorem~\ref{thm:main renormalisation_intro} must rely on most of the arguments developed in \cite{BGHZ}. In particular, one first passes through decorated trees in order to exploit the appropriate analytic properties of the solution map.
Recently, the authors of \cite{Broux25} proposed a complete solution theory for the $\Phi^4_3$ equation formulated using multi-indices. We expect that a similar solution theory for the gKPZ equation in a purely multi-index setting should also hold, and would allow for a full, tree-free proof of the results in this paper.
\end{remark}

\begin{remark}
Following the results in \cite{BGN24}, our results could be easily adapted to have a version of Theorem \ref{thm:main renormalisation_intro} for the quasilinear generalised KPZ equation
\begin{equs} \label{eq:gen_KPZ_quasi}
\partial_t u  =a(u)\partial_x^{2} u  + \Gamma(u) (\partial_x u)^{2} +g(u)\partial_x u + h(u)+  \sigma(u) \zeta\,,
\end{equs}
with $\zeta$ space-time white noise. One can use the solution theory for decorated trees provided therein and proceed similarly to the proof of Theorem~\ref{thm:main renormalisation_intro}. This leads to a  parametrisation of the renormalised equations with less combinatorial terms. Moreover, one could also adapt some of their arguments using the chain rule together with \cite[Assumption 2]{BGN24} to study  global well posedness of $u$ with a generic subcritical noise.
\end{remark}

\begin{remark}  We expect that the computation of the geometric counter-terms works for a more general class of subcritical equations, as defined in \cite[Sec. 2]{BCCH}. In our  context we could consider a scalar equation $u\colon \mathbb{R}\times \mathbb{T}^d\to \mathbb{R}$  given by
	\begin{equs}
\partial_t u - Lu =		\sum_{\beta_1,..., \beta_n \in \mathcal{O_-}}\Gamma_{\beta_1,...,\beta_n}(\partial^{\alpha} u \colon \alpha \in \mathcal{O_+}) \prod_{i=1}^n\partial^{\beta_i}u + \sigma(\partial^{\alpha} u \colon \alpha \in \mathcal{O_+})\xi\,,
		\end{equs}
		where $ \Gamma_{\beta_1,...,\beta_n} $ and $\sigma$ are smooth coefficients and $  \mathcal{O_+}$, $  \mathcal{O_-}$  are finite subsets of $ \mathbb{N}^{d+1} $ with $\partial^{\alpha}$ the corresponding derivatives in space. Here, the noise $\xi$ and the operator $L$ have to be chosen such that $\partial^{\alpha} u $ (resp. $\partial^{\beta} u $) are functions (resp. distribution) when $\alpha \in \mathcal{O}_+$ (resp. $\beta \in \mathcal{O}_-$ ). Under the general assumption of subcriticality for the products $ \prod_{i=1}^n\partial^{\beta_i} u $, see \cite[Def. 2.3]{BCCH}, there exists a general solution theory for $u$ and we can ask ourselves how to recover the geometric symmetries of the solution. 
		
In this general case, one has to work with more general multi-indices that encode the various products appearing in the equation. In particular, we expect most of the proofs of Section~\ref{Sec::3} to carry over, provided that the non-linearities admit a suitable extension to higher-order tensors. For the upper bound in Proposition~\ref{upper_bound_dimension}, one has to replace the free Novikov algebra with one generator by the free multi-Novikov algebra with one generator introduced in \cite{BD23}. This corresponds to working with several commuting derivatives, due to the fact that $\sigma$ and $\Gamma_{\beta_1,\ldots,\beta_n}$ depend on multiple variables.

Concerning the lower bound in Proposition~\ref{lower_bound_dimension}, the same reasoning should apply if the covariant derivative defined in \eqref{covariant_derivative} induces a subspace of geometric elements, which does not seem to be clear at present. In general, finding a generating set for the space of geometric elements appears to be a very challenging problem.  
\end{remark}

\subsection{Outline of the paper} 

Let us summarise the content of this paper.
In Section~\ref{Sec::2}, we introduce the multi-indices for solving the equation \eqref{eq:gen_KPZ}. We follow the formalism introduced in \cite{BL23}. It is formed of more general multi-indices than in \cite{OSSW,LOT}. The main new contribution of this section is the introduction of convenient notations that will be used throughout the rest of the paper to characterise the symmetries. The multi-indices are viewed as polynomials in the variables $ z_{ (\<generic>,k)}$ and $z_{( \<derivative2>,k)}$ for $k \in \mathbb{N}$. The first variable is associated with the product $ \sigma(u_{\eps}) \xi_{\eps}$ and the second one with $ \Gamma(u_{\eps}) (\partial_x u_{\eps})^2 $. They are also associated with some coefficients described below:
\begin{equs}
	z_{ (\<generic>,k)} \leftrightarrow \sigma^{(k)}(u_{\eps})\,, \quad  
	z_{( \<derivative2>,k)} \leftrightarrow 2 \Gamma^{(k)}(u_{\eps})\,.
\end{equs}
The coefficient $2$ in front of $\Gamma^{(k)}(u_{\eps})$ comes from the exponent of $ (\partial_x u_{\eps})^2 $.
We recall in \eqref{def_Psi} the definition of the map $ \Psi $ that allows us to associate a decorated tree with a multi-index. Then, we introduce the BPHZ model over general noise and reduced multi-indices. The model is given in \eqref{renormalised_recursion_equation}. After putting the correct assumptions on the noise $\xi$ in Definition~\ref{defn_admiss_noise}, we recall the BPHZ choice of the renormalisation constants in Theorem~\ref{BPHZ_theorem}. We consider symmetric noise given in Definition~\ref{symmetric_noise} and in Proposition~\ref{symmetric_ex}, we show that noises given by Definition~\ref{defn_admiss_noise} are actually symmetric. We then prove in Proposition~\ref{prop:reduction} that some BPHZ constants are equal to zero which allows us to work with a reduced set of multi-indices in Definition~\ref{reduced_multi}. We finish the section by defining the map $ \Upsilon_{\Gamma,\sigma} $ on the reduced multi-indices (see Definition~\ref{def_Uspilon}).

 In Section~\ref{Sec::3}, we start by defining precisely the symmetries considered in this work, which are the chain rule, the Itô isometry, and an additional symmetry (the "nice" symmetry) introduced in Definition~\ref{def_symmetries_nice}. 

We begin with the chain rule, which is considered via infinitesimal changes of coordinates. This requires the introduction of new variables $ z_{(\begin{tikzpicture}[scale=0.2,baseline=-2]
\coordinate (root) at (0,0);
\node[diff] (rootnode) at (root) {};
\end{tikzpicture},k)} $ associated with $h^{(k)}$. Here, $h$ is to be understood as the first-order expansion of a family of diffeomorphisms $(\psi_t)_{t \geq 0}$ satisfying $ \psi_t = \mathrm{id} + t h + o(t)$. We introduce then an important map $\hat{\phi}\geo$ given in Definition~\ref{def:injectivity_Upsilon} whose kernel characterises the Chain rule; see Theorem~\ref{geo_chain_rule}. The proof of this theorem crucially relies on Theorem~\ref{injectivity_Upsilon}, which shows the injectivity of $ \Upsilon^h_{\Gamma, \sigma} $. The map $ \Upsilon^h_{\Gamma, \sigma} $ is the extension of $ \Upsilon_{\Gamma, \sigma} $ to multi-indices containing variables of the form $z_{(\begin{tikzpicture}[scale=0.2,baseline=-2]
\coordinate (root) at (0,0);
\node[diff] (rootnode) at (root) {};
\end{tikzpicture},k)}$.

We then identify a natural subspace of geometric counterterms via Definition~\ref{def_B_xi}, given by iteration of covariant derivatives defined in \eqref{covariant_derivative}. The rest of the paper is then dedicated to showing that this subspace is the space of geometric counterterms and to computing its dimension, both given by Theorem~\ref{dim_geo}. In Proposition~\ref{lower_bound_dimension}, we establish a lower bound on the dimension. Example~\ref{example_1} illustrates this bound on multi-indices coming from space-time white noise. The upper bound is proved in Proposition~\ref{upper_bound_dimension} via the construction of a triangular system. This system is written explicitly in Example~\ref{example_2} with four variables $z_{(\<generic>,k)}$, again in the case where the noise is Gaussian and has the same regularity properties as space-time white noise. 

We then turn to the Itô isometry and show in Proposition~\ref{thm_Ito} that it cannot discriminate between solutions, as most counterterms have an even number of variables $z_{ (\<generic>,k)}$. Finally, we describe the properties of the  "nice"  symmetry in Proposition~\ref{thm_nice}.

Then, in the last section~\ref{sec:4}, we provide a proof of Theorem~\ref{thm:main renormalisation_intro}. This final section gathers the algebraic ingredients developed in the previous sections and combines them to establish the symmetries of the solution. 

Most of the ideas of the proofs are coming from \cite{BGHZ}, but due to the use of multi-indices, one can observe major simplifications. The most critical part is the upper bound and trying to prove Proposition~\ref{upper_bound_dimension} using decorated trees is likely intractable. This shows the need for more advanced algebra.

\subsection{Further consequences}

Multi-indices have been introduced to simplify the construction with decorated trees for singular SPDEs and propose a more natural expansion in dimension one. So far, most of the (new) proofs/constructions concerning multi-indices have their equivalent at the level of decorated trees and one cannot argue clearly that one is simpler than the other.
\begin{itemize}
	\setlength\itemsep{0.5em}
	\item The Hopf algebras at the core of Regularity Structures have been understood in \cite{reg,BHZ}. The equivalent construction could be found in \cite{LOT} but without the renormalisation via Hopf algebras which is better understood in \cite{Li23} in the context of Rough Paths. These Hopf algebras have been understood in \cite{BM23, BK23} as originating from the same type of post-Lie product. 
	\item The proof of the renormalised equation  in \cite{BCCH} has been largely simplified in \cite{BB21b} thanks to preparation maps introduced in \cite{BR18}. A different approach is advocated in \cite{LOTT}: The top-down approach which is to find first the correct ansatz for the equation that will produce renormalised multi-indices. In \cite{BCCH}, it is a bottom-up approach that has been chosen: one first renormalises divergent stochastic iterated integrals and then sees how this transformation acts on the equation. In \cite{BL23} where the renormalised equation is computed for a large class of equations via multi-indices, the renormalisation map is close in spirit to a preparation map.
	\item In \cite{LOTT}, the authors proposed a diagram-free approach to convergence based on a spectral gap inequality and the use of the Reconstruction Theorem. This multi-index–based proof was later extended and reformulated in the framework of decorated trees in \cite{BN24,HS,BH23}. So far, these proofs do not apply in a non-translation-invariant setting. The approach introduced in \cite{BB23}, which combines spectral gap estimates with diagrammatic techniques, provides an alternative way to address this problem.
\end{itemize}

In the present work, one can notice substantial simplifications due to multi-indices. In Definition~\ref{def_Uspilon}, the map $ \Upsilon_{\Gamma, \sigma} $ is given as a polynomial map. The proof of its injectivity in Theorem~\ref{injectivity_Upsilon} is immediate.  The key map  $\phi_\geo $ is defined without any algebra: in \cite{BGHZ}, it was given as a $T$-algebra infinitesimal morphism and in \cite{BD24} it respects the operadic structure. The most dramatic simplification is in Proposition~\ref{upper_bound_dimension} for the proof of the upper bound on the dimension of geometric elements.

In the end, we can provide an elementary self-contained proof of the open problem left in  \cite{BD24} which looks at first sight quite complicated to establish.
 Indeed, the chain rule for multi-indices was left open in the first arxiv version of \cite{BD24} and at that time the authors were thinking that it was a difficult problem more challenging than the decorated trees case. The elementary proof of the present paper solves this open problem.

 The homological method of \cite{BD24} relies on knowing the homology of the operadic twisting of the operad of pre-Lie algebras. For the operad of Novikov algebras, the homology of operadic twisting is not known and seems more difficult than the one for decorated trees. The main result of this work could give insight into this open problem.

As mentioned before, one does not have a complete solution theory for multi-indices. However, once it is completed, one can use the chain rule to get important results. Indeed, the approach for solving quasi-linear SPDEs in \cite{MH,Mate19,BGN24} depends on the chain rule for getting local counter-terms in the solution. Also, in \cite{BGN24}, the chain rule is the key to obtaining long-time existence results. Such results could be obtained in the context of multi-indices.

\subsection*{Acknowledgements}

{\small
	C.B. and Y.B. gratefully acknowledge funding support from the European Research Council (ERC) through the ERC Starting Grant Low Regularity Dynamics via Decorated Trees (LoRDeT), grant agreement No.\ 101075208. The authors thank V. Dotsenko for identifying an error in a previous version and are grateful to the referee for insightful comments. The authors thank the anonymous referees for their remarks that improve significantly the paper.  Views and opinions expressed are however those of the author(s) only and do not necessarily reflect those of the European Union or the European Research Council. Neither the European Union nor the granting authority can be held responsible for them.
}

\section{Tracking the solution with multi-indices}
\label{Sec::2}

To analyse rigorously the symmetries of \eqref{eq:gen_KPZ}, we consider its multi-index expansion  via the notion of generalised multi-index, as explained in \cite{BL23} and we put it in relation with the usual tree expansion of \eqref{eq:gen_KPZ}, given in \cite{BGHZ}, which is obtained by applying the usual results contained in \cite{reg,BHZ,CH,BCCH}.

\subsection{Multi-indices and trees associated to the gKPZ equation} 

The main idea is to rewrite \eqref{eq:gen_KPZ} as
\begin{align}\label{eq:gen_KPZ2}
\partial_t u  =\partial_x^{2} u +a_0(u, \partial_xu) + a_{\Xi}(u) \xi,
\end{align}
and then formally expand each non-linearity into a Taylor polynomial in the generalised variables, i.e.
\[
a_0(u, \partial_xu) = \sum_{k=(k_1,k_2)\in \mathbb{N}^2} \alpha_{0, k}u^{k_1}(\partial_x u)^{k_2}\,,
\qquad 
a_{\Xi}(u) = \sum_{l\in\mathbb{N}} \alpha_{\Xi,l}u^{l},
\]
for some coefficients $\alpha_{0, k}$, $\alpha_{\Xi,l}$, where $k\in\mathbb{N}^2$ and $l\in\mathbb{N}$.

Because these expansions involve constraints when expressed in the variables $z$ (in particular $k_2 \leq 2$), it is convenient to encode them in the combinatorial framework of generalised multi-indices. 
In general, for a countable set $I$, a \textbf{multi-index over $I$} is a map $m \colon I \to \mathbb{N}$ such that $m(i) = 0$ for all but finitely many $i \in I$. The set of multi-indices over $I$ is denoted by $M(I)$.

To describe \eqref{eq:gen_KPZ2} systematically we will follow the approach developed in \cite[Example 2.5]{BL23}. We start from the finite label set $\mathcal{L}= \{0,\Xi\}$. To each label $\mathfrak{l}\in\mathcal{L}$ we associate a countable subset $\mathcal{N}_{\mathfrak{l}} \subset M(\mathbb{N}^2)$ that encodes all possible non-linearities acting on $u$ and its derivatives. This is done by identifying an element $(k_1,k_2)\in\mathbb{N}^2$ with the corresponding space-time derivative $\partial_t^{k_1}\partial_x^{k_2}u$. Therefore for equation \eqref{eq:gen_KPZ2} we set
\begin{equs}\label{eq:explicit_sets}
 \mathcal{N}_{0} &= \bigl\{ k_1 e_{(0,0)} + k_2 e_{(0,1)} \colon k_1 \in \mathbb{N},\; k_2 \in \{0,1, 2\}\bigr\},\\
\mathcal{N}_{\Xi} &= \bigl\{ l e_{(0,0)} \colon l \in \mathbb{N}\bigr\},
\end{equs}
where $e_{\mathbf{n}} \colon \mathbb{N}^2\to \mathbb{N}$ denotes the indicator function of the singleton $\{\mathbf{n}\}$ (thus, for instance, the dependence on $(\partial_x u)^{k_2}$ is described by $k_2 e_{(0,1)}$).

Finally, we combine these non-linearities with the usual  polynomials  in space-time variables $t^{n_1}x^{n_2}$ for $\mathbf{n}=(n_1,n_2)$ (which, via a formal Taylor expansion, can represent all smooth functions $f\colon \mathbb{R}^2\to \mathbb{R}$), to obtain the set of decorations
\[
\mathcal{N} = \mathcal{N}_{0} \sqcup \mathcal{N}_{\Xi} \subset \mathcal{L} \times M(\mathbb{N}^2), \qquad 
\mathcal{R} = \mathcal{N} \sqcup \mathbb{N}^2,
\]
and define $\mathfrak{M} = M(\mathcal{R})$. Each element $\beta\in \mathfrak{M}$ will be in a one-to-one correspondence with  the canonical basis of the polynomial algebra
\[\mathcal{M}=\mathbb{R}[z_{k}\,, z_{\mathbf{n}}]_{k\in  \mathcal{N}, \,\mathbf{n}\in \mathbb{N}^2} \]
via the classical identification
\[\beta \to  z^{\beta}= \prod_{k\in  \mathcal{N}, \,\mathbf{n}\in \mathbb{N}^2} z_{k}^{\beta(k)} z_{\mathbf{n}}^{\beta(\mathbf{n})}\,,
\]
which we will use all along the paper. Usually the  variables associated to $\mathcal{N}_{\Xi}$ and  $\mathcal{N}_0$ have  the notation $z_{(\Xi, le_{(0,0)})}$, $z_{(0,k_1 e_{(0,0)} + k_2 e_{(0,1)})} $ but in what follows we drastically simplify this notation by using the following conventions:
\begin{equs}\label{eq:convention}
	z_{(\Xi, le_{(0,0)})}=z_{ (\<generic>,l)}\,, \quad z_{(0,ke_{(0,0)}+ 2 e_{(0,1)})}=z_{( \<derivative2>,k)}\,, \quad z_{(0,ke_{(0,0)}+  e_{(0,1)})}=z_{( \<derivative>,k)}\,.
\end{equs}

Each multi-index comes with several quantities associated to it, each one with a key role to describe the solution \eqref{eq:gen_KPZ2}: a 
 \textbf{noise  factor}  depending only on the values of $\beta$ over $\mathcal{N}_{\Xi}$
\[((\beta))=\sum_{k\in \mathcal{N}_{\Xi}}\beta(k)\,,\]
a \textbf{fertility factor} 
\[[\beta]=\sum_{k\in \mathcal{N}} (1- \ell(k))\beta(k)+\sum_{\mbn \in \mathbb{N}^2}\beta(\mbn)\,,\] 
where $\ell(k)=\sum_{u\in \mathcal{R}}k(u)$ is the length of an element in $M(\mathbb{N}^2)$. Finally, given the regularity values $\alpha_{\Xi}= -2+ \delta  $ with $0<\delta<1$ and $\alpha_{0}=0$, we consider the \textbf{homogeneity factor}
\[|\beta|=\sum_{(l,k)\in \mathcal{N}}\left(\alpha_l +\sum_{\mathbf{n} \in \mathbb{N}^2} (2- \Vert\mathbf{n}\Vert )k(\mathbf{n})\beta(k)\right)+ \sum_{\mathbf{n}}(\Vert\mathbf{n}\Vert-2)\beta(\mbn)\,,\] 
where we put the parabolic degree $\Vert \mathbf{n}\Vert= 2n_1+ n_2$. Following \cite[Def. 2.22]{BL23},  the following  set of multi-indices
\[T= \{\beta\in \mathfrak{M}\colon [\beta]=1, \quad ((\beta)) >0 \}\]
and the associated vector space $\mathcal{T}=\langle T\rangle$ contain the abstract symbols to describe the equation \eqref{eq:gen_KPZ2}. A key role for renormalisation is played by $T_-$, the elements of $T$ satisfying $|\beta|<0$. Indeed we can easily apply \cite[Lem. 2.23]{BL23} directly to this context and deduce that for any $0<\delta<1$ the set $T_-$ is finite (with a cardinality depending on $\delta$). In order to understand what type of multi-indices are present we list in Table~\ref{tab:kpz} the negative multi-indices when $\delta=\frac{1}{2}-\kappa$ for $\kappa>0$ sufficiently  small. This choice of the parameter $\delta$ is directly given when $\xi$ is the space-time white noise, see also \cite[Table 1]{BL23} (the notation $\gamma^-$ stands for expressions of the form $\gamma- m \kappa$ with $m\in \mathbb{N}$). Almost all elements of $T_-$ can be expressed using the conventions \eqref{eq:convention} with the exception of $z_{(0,1)}$, which represents the polynomial $x$. We denote the linear span of  $T_-$ by $\mathcal{T}_-$.

\begin{table}[ht]
\begin{tabular}{c | c | c}
			$|\beta|$ & $\beta$ & $\#\{\beta\}$ \\
			\hline
			&&\\
			$(-3/2)^-$ & $ z_{ (\<generic>,0)}$ & 1 \\ & & \\
			 \hline
			&&\\
			$(-1)^-$ & $ z_{ (\<generic>,0)}z_{ (\<generic>,1)}$, $\quad z_{ (\<generic>,0)}^2 z_{(\<derivative2>,0)}$ & 2 \\ & & \\
			\hline 
			 & & \\
			& $z_{ (\<generic>,0)}^2z_{ (\<generic>,2)}$, $ \quad z_{ (\<generic>,0)}z_{ (\<generic>,1)}^2$, $\quad z_{ (\<generic>,0)}z_{( \<derivative>,0)}$, $ \quad z_{ (\<generic>,1)}z_{(0,1)}$ &  \\
			$(-\frac{1}{2})^-$ &  $ z_{ (\<generic>,0)}^2z_{ (\<generic>,1)}z_{( \<derivative2>,0)}$, $\quad z_{ (\<generic>,0)}^3 z_{( \<derivative2>,0)}^2$,     $\quad z_{ (\<generic>,0)}^3z_{( \<derivative2>,1)}$& 7
	 \\ &&\\
			\hline
			 &&\\
			& $z_{ (\<generic>,0)}z_{ (\<generic>,1)}^3$, $\quad z_{ (\<generic>,0)}z_{ (\<generic>,1)}z_{( \<derivative>,0)}$, $\quad  z_{ (\<generic>,0)}^3z_{ (\<generic>,3)} $& \\
			& $\quad z_{ (\<generic>,0)}^2z_{ (\<generic>,1)}z_{ (\<generic>,2)}$, $\quad z_{ (\<generic>,0)}^2z_{ (\<generic>,1)}^2z_{(\<derivative2>,0)}$,  $\quad z_{ (\<generic>,0)}^2z_{( \<derivative>,0)}z_{( \<derivative2>,0)} $, & \\
			$ (0)^- $ & $z_{ (\<generic>,0)}^3z_{ (\<generic>,1)}z_{(\<derivative2>,0)}^2 $, $\quad z_{ (\<generic>,0)}^3z_{ (\<generic>,1)}z_{( \<derivative2>,0)}$, $\quad z_{ (\<generic>,0)}^3 z_{ (\<generic>,2)}z_{( \<derivative2>,0)} $ & 17\\
			&$ z_{ (\<generic>,0)}^2z_{( \<derivative>,1)}$, $ \quad z_{ (\<generic>,0)}^4 z_{( \<derivative2>,0)}^3$, $\quad z_{ (\<generic>,0)}^4 z_{(\<derivative2>,2)} $,& \\ & $ z_{ (\<generic>,0)}^4z_{( \<derivative2>,0)}z_{(\<derivative2>,1)} $,    $\quad z_{ (\<generic>,0)} z_{ (\<generic>,2)}z_{(0,1)}, $ 
			 $\quad z_{ (\<generic>,1)}^2z_{(0,1)}$ & \\ & $ z_{ (\<generic>,0)}^2z_{( \<derivative2>,1)}z_{(0,1)}$, $\quad z_{ (\<generic>,0)}z_{ (\<generic>,1)}z_{( \<derivative2>,0)}z_{(0,1)} $ & \\ &&\\
			\hline 
		\end{tabular}
		\caption{All the elements of $T_-$ when  $\delta=(1/2)^-$}
		\label{tab:kpz}
		\end{table}

The elements of $T$ can be easily linked with a  class of trees introduced in \cite{BGHZ} to consider a more general version of \eqref{eq:gen_KPZ}. We will refer to \cite[Sec. 4]{BHZ}, \cite[Sec. 4]{BCCH} for usual references to decorated trees. 

Starting from the same label set $\mathcal{L}$, we consider the set of decorated trees whose  vertices are decorated by the set $\mathcal{L}\sqcup  \{ X^{\mathbf{n}}\}_{\mathbf{n}\in \mathbb{N}^2}$ (the symbol $X^{\mathbf{n}}$ stands for the monomial  $X^{\mathbf{n}}=X^{n_1}_1X^{n_2}_2$), and whose edges are decorated by  the two possible decorations $\{(0,0),(0,1)\}$. We  denote by $\mathfrak{T}$ the set of trees of this form and by $\langle\mathfrak{T}\rangle$ the vector space generated by $\mathfrak{T}$. Using the standard operation of tree  product $\tau \cdot \sigma$ among the trees $\tau, \sigma $ and for any $\mathbf{m}\in \mathbb{N}^2$ the symbol $\mathcal{I}_{\mathbf{m}}(\tau)$  for the operation of creating a new tree with an extra edge  decorated with $\mathbf{m}$ pointing to the root which is decorated with $\{0\}$, we can recursively construct  $\mathfrak{T}$ starting from $ \mathcal{L}\sqcup \{ X^{\mathbf{n}}\}_{\mathbf{n}\in \mathbb{N}^2}$ and setting 
\begin{equation}\label{eq:recursive_tree}
\tau=\sigma\cdot X^{\mathbf{n}}\cdot \mathcal{I}_{\mathbf{m}_1}(\tau_1)\cdots  \mathcal{I}_{\mathbf{m}_N}(\tau_N)
\end{equation}
with $\sigma\in \mathcal{L}$, $N\geq 0$, $\tau_i\in \mathfrak{T}$ and  $\mathbf{m}_i\in  \{(0,0),(0,1)\}$.

Following \cite[Sec. 3]{BK23}, the explicit connection between multi-indices and trees is done by an explicit linear map $\Psi\colon \mathcal{B} \to \mathcal{M}$. We define it as the unique linear map satisfying 
  $\Psi(X^{\mathbf{n}})=z_{\mathbf{n}}$, $\Psi(\Xi)= z_{ (\<generic>,0)}$, $\Psi(0)= z_{(0,\mathbf{0})}$, and setting recursively on any $\tau$ of the form \eqref{eq:recursive_tree}
\begin{equation}\label{def_Psi}
 \Psi(\tau)=z_{(\sigma, k)} X^{\mathbf{n}}\prod_{i=1}^N\Psi(\tau_i)\, ,  
 \end{equation}
with $k= \sum_{i=1}^N \mathbf{m}_i+ |\mathbf{n}|e_{\mathbf{0}}$ and $z_{(\sigma, k)}$ is the associated monomial in the  variables related to $\mathcal{N}$. For instance, using the shorthand notation $\mathcal{I}_1=\mathcal{I}_{(0,1)}$ we have the identities
\[\Psi( \mathcal{I}_1 (\mathcal{I}_1 (\Xi))^2 \cdot \mathcal{I}_1 (\mathcal{I}_1 (\Xi))^2)=\Psi( \mathcal{I}_1 (\Xi) \cdot \mathcal{I}_1 (\mathcal{I}_1 (\Xi) \cdot \mathcal{I}_1 (\mathcal{I}_1 (\Xi)^2)=z_{ (\<generic>,0)}^4 z_{( \<derivative2>,0)}^3\,.\]
From this we immediately see that $\Psi$ is not injective.

 Concerning the surjectivity of $\Psi$,  the condition $[\beta]=1$ in $T$ ensures the existence of a tree with $\ell(\beta)$ nodes and $ \sum_{k \in \mathcal{N}}\ell(k)\beta(k)$ edges where the nodes and edges are decorated according to the variables in $\beta$. In particular, in absence of space time variables (i.e. $\beta(\mathbf{n})=0$ for all $\mathbf{n}\in \mathbb{N}^2$) the integer values in the variables $z_{ (\sigma,m)}$, with $\sigma\in\{ \<generic>\,, \<derivative> , \<derivative2> \}$ denotes the number of incoming edges of the type $\mathcal{I}$ (i.e. the fertility) to nodes of type $\Xi$ and edges of form $\mathcal{I}_1$ and $\mathcal{I}_1\cdot \mathcal{I}_1$ respectively, otherwise one needs to add to the nodes polynomials $X^{\mathbf{n}}$ to match the  fertility parameters.  For example, using the symbols $\{ \<generic>\,, \<derivative> , \<derivative2> \}$ for the node  $\Xi$ and edges of form $\mathcal{I}_1$ and $\mathcal{I}_1\cdot \mathcal{I}_1$ the previous example becomes
\[\Psi(\<Xitwo>  )=\Psi(  \<2I1Xi4>)=z_{ (\<generic>,0)}^4 z_{( \<derivative2>,0)}^3\,.\]
This explicit map will play a fundamental role in adapting the techniques on trees to multi-indices in Section~\ref{Sec::3}.

\subsection{BPHZ model over general noise and reduced multi-indices}
\label{section_2.2}

Once a description of the combinatorial objects associated with the generalized KPZ equation is given, we explain how to construct the associated BPHZ model over multi-indices, in the same way as in \cite{LOTT}, using the formalism introduced in \cite{BL23}. This construction provides a consistent way to attach at each $\beta\in T$ and $x\in \mathbb{R}^2$ a proper distributional random field  $x\mapsto\Pi_{x\beta}\in \mathcal{D}'(\mathbb{R}^2)$ describing some non linear functionals of the underlying noise $\xi$, which will encode the solution. This choice will also drastically reduce the underlying vector space of symmetries where we can focus our analysis but still it does not trivialise it. Since the association $\beta \mapsto\Pi_{x\beta}$ will be linear in $\beta$ we work naturally with the dual spaces $\mathcal{T}^*$ and $\mathcal{T}_-^*$, denoting the variables of the dual space with monomials of the form $z^{\beta}_*$ and the canonical pairing  by $ \langle z^{\beta_1}_* , z^{\beta_2} \rangle=\delta_{\beta_1\beta_2}$ . We note that by construction  $\mathcal{T}_-^*$ is finite dimensional but $\mathcal{T}^*$ is naturally identified to a space of formal series.  

To obtain a model similar to the one introduced in \cite{BGHZ}, we fix a decomposition of the heat kernel $P$ over the real line as $P = K + R$, where $K$ is compactly supported around the origin, integrates to zero against all polynomials up to parabolic degree $1$, is symmetric in the space variable (i.e., $K(t, x) = K(t, -x)$), and $R$ is globally smooth (see \cite[Lem. 5.5]{reg}).  
Moreover, we consider a function $\varrho : \mathbb{R}^2 \to \mathbb{R}$ belonging to the class $\mathcal{M}$, which consists of all smooth functions integrating to $1$, symmetric in the space variable (i.e., $\rho(t, x) = \rho(t, -x)$), and such that $\rho = 0$ for $t \leq 0$. Given these technical ingredients, we fix $\xi$, a generic random distribution over $\mathbb{R} \times \mathbb{T}$, and use it to introduce the mollified noise $\xi_{\varepsilon} = \varrho_{\varepsilon} * \xi$, where $*$ denotes convolution over $\mathbb{R}^2$, $\varrho_{\varepsilon} = \varepsilon^{-3} \varrho(\varepsilon^{-2} t, \varepsilon^{-1} x)$ is the parabolic rescaling of $\varrho$ preserving its $L^1$ norm, and $\xi$ is periodically extended over $\mathbb{R}^2$ (denoted by the same symbol).

The trajectories of $\xi_{\varepsilon}$ are almost surely smooth for any $\varepsilon > 0$. This property allows us to define smooth models $(\Pi, \Gamma)$ for $\mathcal{T}$, consisting of a map $x \mapsto \Pi^{\varepsilon}_x : \mathbb{R}^2 \to (\mathcal{T} \to \mathcal{C}^{\infty}(\mathbb{R}^2))$, and a map $\Gamma^{\varepsilon}_{xy} : \mathbb{R}^2 \times \mathbb{R}^2 \to (\mathcal{T} \to \mathcal{T})$, satisfying specific algebraic and analytic bounds between their components (see \cite[Def. 2.17]{reg}). For our purposes, it suffices to define a specific map $\Pi^{\varepsilon} : \mathcal{T} \to \mathcal{C}^{\infty}(\mathbb{R}^2)$, where the dependence on $x$ is suppressed. Once $\Pi^{\varepsilon}$ is fixed, it is a standard (though non-trivial) consequence of the theory of regularity structures and its algebraic framework to construct a model from $\Pi^{\varepsilon}$ (see, e.g., \cite[Remark 8.2]{reg}, \cite[Definition 6.8]{BHZ}).

To define $\Pi^{\varepsilon}$ we start by setting $\Pi_{z_{\mathbf{n}}}^{\varepsilon}(y)= y_1^{n_1}y_2^{n_2}$. Moreover, for any $\beta \in \mathcal{N}$, $\Pi^\varepsilon$ is then obtained through the coupled object
\[
(\Pi^\varepsilon, \Pi^{-,\varepsilon})\colon  \mathcal{T}^*\times \mathcal{T}_-^*\to \mathcal{C}^{\infty}(\mathbb{R}^2)\times \mathcal{C}^{\infty}(\mathbb{R}^2),
\]
defined recursively by
\begin{equation}\label{canonical_recursion_equation}
\begin{cases}
\Pi_{\beta}^\varepsilon &= K*\Pi_{\beta}^{-, \varepsilon} \,,\\[2mm]
\Pi_{\beta}^{-, \varepsilon} &= \displaystyle\sum_{(l,k)\in \mathcal{N}}\;
\sum_{\substack{z^{\beta}= z_{(l,k)} \prod_{\mathbf{n}}\prod_{i=1}^{k(\mathbf{n})}z^{\beta_i}}} 
\prod_{\mathbf{n}}\prod_{i=1}^{k(\mathbf{n})}\Bigl(\frac{1}{\mathbf{n}!}\partial^\mathbf{n}\Pi_{\beta_i}^{\varepsilon}\Bigr)\xi_l\,,
\end{cases}
\end{equation}
with $\xi_{\Xi}= \xi_{\varepsilon}$ and $\xi_{0}=1$. The condition $z^{\beta}= z_{(l,k)} \prod_{\mathbf{n}}\prod_{i=1}^{k(\mathbf{n})}z^{\beta_i}$ in \eqref{canonical_recursion_equation} may seem nontrivial, but the explicit structure of $\mathcal{T}$ allows a systematic construction of $\Pi_{\beta}$, which in turn uniquely determines $\Pi_{\beta}$.

This model and its recursion \eqref{canonical_recursion_equation} arise heuristically by considering the trivial regularisation of equation \eqref{eq:gen_KPZ} without renormalisation, that is, the random PDE
\[\d_t u_{\eps}  = \d_x^2 u_{\eps} + \Gamma(u_{\eps})\,(\d_x u_{\eps})^2
		+ g(u_{\eps})\,\d_x u_{\eps}
		+h(u_{\eps}) + \sigma(u_{\eps})\, \xi_{\eps}
		\]
whose limit is in general ill-defined. In order to take into account the actual renormalised equation \eqref{eq:renorm nonlocal intro2}, we need to modify it according to the specific choice of constants and non-linearities in the original equation.

The way we use to implement the constants is directly imported from \cite{BL23}. We fix  $c\in \mathcal{T}_-^*$  of the form
\[c= \sum_{\beta'\in \mathcal{T}_-} c_{\beta'}z^{\beta'}_*\] 
and we perturb $\Pi^{\varepsilon}_x$ to a functional $\Pi_{x}^{\varepsilon,c}$ in a consistent way to encode the new constants  at the level of the equation. This is in few words the main result in \cite[Thm. 4.1]{BL23}. Even in this case, it is possible to obtain $\Pi_{\beta}^{\varepsilon,c}$  via a more sophisticated recursive construction of the form 
\begin{equation}\label{renormalised_recursion_equation}
\begin{cases}
		\Pi_{\beta}^{\varepsilon,c} &= K*\Pi_{x\beta}^{-, \varepsilon, c}\, \\
		\Pi_{\beta}^{-, \varepsilon, c} &=\Pi_{\beta}^{-, \varepsilon}+  (\Gamma^{*, \varepsilon}_{\PPi}c)(\beta)
	\end{cases}
\end{equation}
where  $\Gamma^{*, \varepsilon}_{\PPi}\colon \mathcal{T}_-^*\to (\mathcal{T})^*$ is given by the sum
\begin{equation}\label{eq:defn_gamma}
(\Gamma^*_{\PPi}c)(\beta)=\sum_{\langle z^{\beta_1}_*\ldots z^{\beta_l}_*D^{(\mathbf{n}_1)}\ldots D^{(\mathbf{n}_l)}z^{\beta'}_* ,z^{\beta}\rangle\neq 0} \left(\frac{1}{l!}\prod_{i=0}^l\frac{1}{\mathbf{n_i}!}\partial^\mathbf{n_i}\Pi_{\beta_i}^{ \varepsilon, c}\right)c_{\beta'}
\end{equation}
with  the derivation maps $(D^{(\mathbf{n})})_{\mathbf{n}\in \mathbb{N}^2}$ explicitly given  by 
\[D^{(\mathbf{n})}= \sum_{(l,k)\in \mathcal{N}}(k(\mathbf{n})+1)(z_{(l,k+e_{\mathbf{n}})})_*\partial_{(z_{(l,k)})_*}+ \partial_{z_{^\mathbf{n}}^*}\,.\]
See \cite[Lem.~3.15, Ex.~3.29]{BL23} for further explanations of the identity \eqref{eq:defn_gamma}.

To fix the components of a generic element $c \in \mathcal{T}-^*$ so as to obtain convergence of $\Pi_{\beta}^{\varepsilon,c}$, we impose a recursive condition involving the stochastic properties of the underlying noise $\xi$. For this reason, we finally introduce a general class of admissible noises for which our results are valid. We recall that we write $\mathbb{E}[X_1 \cdots X_n]$ to denote the expectation of the product of $n$ random variables $X_1, \ldots, X_n$, and we use the notation $\mathbb{K}[X_1, \ldots, X_n]$ for the joint cumulants of the random variables $X_1, \ldots, X_n$.

\begin{definition}\label{defn_admiss_noise}
A  random distribution $\xi$ over $\mathbb{R}^2$  is said an admissible noise if it satisfies the following properties.
\begin{itemize}
\item For all test functions $\varphi\in \mathcal{D}(\mathbb{R}^2)$, $\xi(f)$ has moments of all orders and it is centered.
\item The law $\xi$ is stationary, i.e. invariant under the action of $\mathbb{R}^2$ on $\mathcal{D}'(\mathbb{R}^2)$
by translation.
\item For every $n\geq 2$ there exists a distribution  $\Gamma^n$ such that for every collection of test functions $ \varphi_1\,, \ldots \,, \varphi_n\in \mathcal{D}'(\mathbb{R}^2)$ one has 
\[\mathbb{K}[\xi(\varphi_1),\ldots,\xi(\varphi_n)]= \Gamma^n(\prod_{i=1}^n\varphi_i)\,.\]
\item  $\Gamma^n$ is locally integrable when $n\geq 3$ and  there exists a distribution $\mathcal{C}$ with singular support contained in $0$ such that
\begin{equation}\label{eq:covariance}
\mathbb{E}[\xi(f)\xi(g)]=\Gamma^2(fg)=\mathcal{C}\left(\int_{\mathbb{R}^2}f(z-\cdot)g(z)dz\right)
\end{equation}
and outside $0$ the distribution $\mathcal{C}$ is a smooth function that satisfies for some $0< \kappa<1$  and any $k\in \mathbb{N}^2$ 
\begin{equation}\label{eq:defn_analytic}
\sup_{z=(t,x)\in \mathbb{R}^2\setminus\{ 0\}} \partial^k_{t,x}\mathcal{C}(z) (|t|^{2+ k_2 +k_1/2 - \kappa}+ |x|^{4+  2k_2 +k_1- \kappa} )<\infty\,. 
\end{equation}

\end{itemize}
\end{definition}
\begin{remark}
Definition~\ref{defn_admiss_noise} is directly recalled from \cite[Def.~2.17]{CH}, so that all convergence results of that paper apply in the present setting. Moreover, condition~\eqref{eq:defn_analytic} links the almost sure analytic regularity of~$\xi$ with the algebraic index $\alpha_{\Xi} = -2 + \delta$.

The periodic extension of a space--time white noise $\zeta$ over $\mathbb{T}\times\mathbb{R}$ is trivially an admissible noise.  More generally, if  $\xi$ is a centered translation-invariant Gaussian noise over $\mathbb{T}\times\mathbb{R}$ with covariance distribution $\mathcal{C}$ satisfying~\eqref{eq:defn_analytic} then it is also an admissible noise.
\end{remark}

Once given this general class of noises, the usual way to define a converging choice is the BPHZ renormalisation choice, see   \cite[Prop.5.1]{LOTT}, \cite[Prop. 2.1]{BGHZ}, \cite[Prop. 2.1]{BHZ}. In our context this is simply  obtained by considering the random field $\Pi_{\beta}^{-,\varepsilon, c}$ and imposing recursively its expectation to be zero.
\begin{theorem} \label{BPHZ_theorem}
For any $\varepsilon> 0$ and any admissible noise $\xi	$ there  exists a unique  element $c^{\BPHZ}_{\varepsilon}\in  \mathcal{T}_-^*$ such that  for any $\beta\in T_-$ and $y\in \mathbb{R}^2$ one has
\begin{equation}\label{BPHZ_condition}
\mathbb{E}[\Pi_{\beta}^{-,\varepsilon, c^{\BPHZ}_{\varepsilon}}(y)]=0\,.
\end{equation}
We call $c^{\BPHZ}_{\varepsilon}$ the BPHZ renormalisation constants.
\end{theorem}

\begin{proof}
The proof follows a classical standard argument on trees, see \cite[Thm. 6.18]{BHZ}, which we can adapt in this simpler context. Following \cite[Cor. 4.6]{BL23} there exists an order $|\beta|_{\prec} $ over the elements of $T_-$  such that for any $y\in \mathbb{R}^2$ one has
\[(\Gamma^{*, \varepsilon}_{\PPi}c)(\beta)= c_{\beta}+ \sum_{|\beta'|_{\prec}<|\beta|_{\prec}}g_{\beta'}(y) c_{\beta'}\,,\]
for some random integrable coefficients $g_{\beta'}(y)$. Using the integrability of the admissible noise $\xi$ and this triangular structure we can indeed solve the recursive system 
\begin{equation}\label{eq:linear_system}
\mathbb{E}[\Pi_{\beta}^{-, \varepsilon}(\mathbf{0})]+  c_{\beta}+ \sum_{|\beta'|_{\prec}<|\beta|_{\prec}}\mathbb{E}[g_{\beta'}(\mathbf{0})] c_{\beta'}=0
\end{equation}
which is a system of linear equations admitting a unique solution which we denote by $c^{\BPHZ}_{\varepsilon}$. By construction of $c^{\BPHZ}_{\varepsilon}$, we have then
$\mathbb{E}[\Pi_{\beta}^{-,\varepsilon, c^{\BPHZ}_{\varepsilon}}(\mathbf{0})]=0$. 
In order to prove the equality \eqref{BPHZ_condition}, it is sufficient to remark that the hypothesis of translation invariance implies the existence for any $h\in \mathbb{R}^2$  of a deterministic  linear map $T_{h}\colon \mathcal{T}_-\to\mathcal{T}_-$ with $T_{h}\circ T_{h'}=T_{h+h'} $ and $T_{\mathbf{0}}=\text{id}$ such that for any $y\in \mathbb{R}^2$
\[\mathbb{E}[\Pi_{\beta}^{-, \varepsilon}(\mathbf{0})]= \mathbb{E}[\Pi^{-, \varepsilon}_{T_{-y}\beta}(y)]\]
since the convolution operation sends polynomials to zero and it is deterministic one has also recursively $\mathbb{E}[g_{\beta'}(\mathbf{0})]= \mathbb{E}[g_{T_{-y}\beta'}(y)]$ and we obtain again that by construction  
\[\mathbb{E}[\Pi_{T_{-y}\beta}^{-,\varepsilon, c^{\BPHZ}_{\varepsilon}}(y)]=0\,.\]
Writing now
\[ \mathbb{E}[\Pi_{\beta}^{-,\varepsilon, c^{\BPHZ}_{\varepsilon}}(y)]=\mathbb{E}[\Pi_{T_{y}T_{-y}\beta}^{-,\varepsilon, c^{\BPHZ}_{\varepsilon}}(y)]\]
by expanding the operator $T_{y}$ the term on the right-hand side is a linear combination of terms  whose expectation is zero, hence the result. The uniqueness of $c^{\BPHZ}_{\varepsilon}$ follows from the uniqueness of the system \eqref{eq:linear_system}.
\end{proof}
In order to preserve the space symmetry of the underlying Kernel $K$, we also focus on a special class of noises which preserve the space symmetry.
 \begin{definition} \label{symmetric_noise}
 An admissible noise $\xi$ is said to be a spatially symmetric noise if it satisfies the equality in law as stochastic process
  \begin{equation}\label{equality_in_law}
  \Pi_{\beta}^{-,\varepsilon}(t,-y)\overset{(d)}{=} -\Pi_{\beta}^{-,\varepsilon}(t,y)
  \end{equation}
  for any $\beta\in T_- $ containing at least one factor of  the form $z_{( \<derivative>,k)}$ with $k\geq 0$ or a factor of the form $z_{( 0,1)}$.  
  \end{definition}
  
  \begin{remark}\label{rk_algebraic_reduction}
 The main idea behind condition \eqref{equality_in_law} is the following. If a multi-index $\beta \in T_-$ contains a factor of the form $z_{( \<derivative>,k)}$ with $k \ge 0$, or a factor of the form $z_{(0,1)}$, then by the construction of $T_-$ such a symbol can appear only once for homogeneity reasons, and these two types of factors can never appear together. Since both $z_{(0,1)}$ and $z_{( \<derivative>,k)}$ are associated, by construction, with odd functions, it is therefore reasonable to impose the hypothesis of spatial asymmetry.
  \end{remark}
This condition is non-trivial, as there are examples for which it is satisfied.
  \begin{proposition} \label{symmetric_ex}
The space-time white noise and the periodic extension of any Gaussian admissible random distribution over $\mathbb{R}\times \mathbb{T}$ such that the function $\mathcal{C}$ given by \eqref{eq:covariance} satisfies $\mathcal{C}(t,-x)=\mathcal{C}(t,x)$ are spatially symmetric noises.
 \end{proposition}
 \begin{proof}
 In the case of space time-white noise, one has $\delta=1^-$ and one can simply check the condition \eqref{equality_in_law} on few cases. In the general case,  as explained in Remark~\ref{rk_algebraic_reduction}, the multi-index $\beta$ will contain only once the symbol $z_{( 0,1)}$ or $z_{( \<derivative>,k)}$. Since the condition $\mathcal{C}(t,-x)=\mathcal{C}(t,x)$ with the Gaussianity implies that $\xi_{\varepsilon}(t,-y)= \xi_{\varepsilon}(t,y)$ with equality as stochastic process, the process $\Pi_{\beta}^{-,\varepsilon}$ will contain only one odd function which will be multiplied or convoluted with functions that are  even functions or even in law (in the same sense as $\xi_{\varepsilon}$), therefore the result satisfies \eqref{equality_in_law}.
 \end{proof}

 Thanks to this hypothesis, we can indeed reduce the set of possible renormalisation constants to a smaller one.
\begin{definition}\label{reduced_multi}
We define the set of reduced and even-reduced multi-indices as the following subset of $T$
\begin{align*}
\mathfrak{M}_{ \<derivative2>, \<generic>} &= \lbrace z^{\beta} = \prod_{k \geq 0}  \left( z_{( \<derivative2>,k)} \right)^{\beta( \<derivative2>,k )}  \prod_{k \geq 0}\left(  z_{(\<generic>,k)}\right)^{\beta(\<generic>,k)} : 
		[\beta] =1 \,,\quad ((\beta))\geq 2\rbrace\\
		\mathfrak{M}_{ \<derivative2>, \<generic>}^{\text{G}}&= \lbrace z^{\beta}\in \mathfrak{M}_{ \<derivative2>, \<generic>} \colon  \quad ((\beta)) \quad  \text{is even} \rbrace\,.
		\end{align*}
		The linear spaces generated by these sets will be denoted by $\mathcal{M}_{ \<derivative2>, \<generic>}$ and $\mathcal{M}_{ \<derivative2>, \<generic>}^{\text{G}} $ respectively.
\end{definition}
More generally we will also use the subsets 
$\mathfrak{M}_{ \<derivative2>, \<generic>}^k = \lbrace z^{\beta}  \in \mathfrak{M}_{ \<derivative2>, \<generic>} : ((\beta)) =  k 
		\rbrace $  with corresponding linear space $\mathcal{M}^k_{ \<derivative2>, \<generic>} $ for any $k\geq 2$. The reduction induced by $\mathfrak{M}_{ \<derivative2>, \<generic>} $  can be trivially described via the BPHZ renormalisation constants.

\begin{proposition}\label{prop:reduction}
For any spatially symmetric noise one has   $c^{\BPHZ}_{\varepsilon}( \beta)=0$ for any $\beta\in T_-\setminus \mathfrak{M}_{ \<derivative2>, \<generic>}^- $ with $\mathfrak{M}_{ \<derivative2>, \<generic>}^-=\{\gamma \in \mathfrak{M}_{ \<derivative2>, \<generic>}\colon | \gamma|<0\} $. Moreover, if the noise is also Gaussian one has in addition $c^{\BPHZ}_{\varepsilon}( \beta)=0$ for any $\beta\in T_-\setminus (\mathfrak{M}_{ \<derivative2>, \<generic>}^{\text{G}}) ^-$ with  $(\mathfrak{M}_{ \<derivative2>, \<generic>}^{G}) ^-=\{\gamma \in \mathfrak{M}_{ \<derivative2>, \<generic>}^{\text{G}}\colon | \gamma|<0\}$.
\end{proposition}

\begin{proof}
The result follows trivially from our hypothesis on the noise. Indeed, the only elements $\beta\in T_-$ with $((\beta))=1 $ are 
\[z_{ (\<generic>,0)}\,,\quad  z_{ (\<generic>,1)}z_{(0,1)}\,, \quad  z_{ (\<generic>,0)}z_{( \<derivative>,0)}. \]
Since for these symbols the renormalisation is only additive one has
\[c^{\BPHZ}_{\varepsilon}( \beta)= - \mathbb{E} [\Pi_{\beta}^{-, \varepsilon}(\mathbf{0})] \]
and these quantities are zero because the noise is centered. Concerning all other elements $\beta\in T_-\setminus \mathfrak{M}_{ \<derivative2>, \<generic>}^- $ they exactly coincide with the elements containing at least one factor of  the form $z_{( \<derivative>,k)}$ with $k\geq 0$ or a factor of the form $z_{( 0,1)}$. For all these elements we can then use the spatial symmetry hypothesis from which we deduce $\mathbb{E} [\Pi_{\beta}^{-, \varepsilon}(\mathbf{0})]=0 $. Plugging this condition in the linear system \eqref{eq:linear_system}, we obtain
\[ c_{\varepsilon}^*(\beta)=- \sum_{|\beta'|_{\prec}<|\beta|_{\prec}}\mathbb{E}[g_{\beta'}(\mathbf{0})] c_{\varepsilon}^*(\beta')\,.\]
By simply working recursively on $|\beta'|_{\prec}$ the non zero terms $c_{\varepsilon}^*(\beta')$ will satisfy $\beta'\in\mathfrak{M}_{ \<derivative2>, \<generic>}^- $, therefore the term $g_{\beta'}(\mathbf{0})$ will always contain  terms of the form $K*\Pi_{\gamma}^{-, \varepsilon}(\mathbf{0})$ for some $\gamma \in T_-\setminus \mathfrak{M}_{ \<derivative2>, \<generic>}^- $, which are zero in expectation because the convolution with $K$ does not destroy the spatial symmetry. In the special case when the underlying noise is also Gaussian we  can easily prove that one has also $\mathbb{E} [\Pi_{\beta}^{-, \varepsilon}(\mathbf{0})]=0 $ for any multi-index with $((\beta))$ odd, because it necessarily contain  odd monomials in the noise $\xi$.  Using the same recursive reasoning as above  we can also conclude that  $c^{\BPHZ}_{\varepsilon}( \beta)=0$ for any $\beta\in T_-\setminus (\mathfrak{M}_{ \<derivative2>, \<generic>}^{\text{G}}) ^-$.
\end{proof}

\begin{table}[ht]
\begin{tabular}{c | c | c| c}
$|\beta|$ & $\beta$ &  Trees &$\#\{\beta\}$ \\
 \hline
			&&&\\
			$(-1)^-$ & $ z_{ (\<generic>,0)}z_{ (\<generic>,1)}$, $\quad z_{ (\<generic>,0)}^2 z_{(\<derivative2>,0)}$  &\<Xi2> \,,\;\<I1Xitwo> & 2 \\ & &  &\\
			\hline 
			 & & &\\
			& $z_{ (\<generic>,0)}^2z_{ (\<generic>,2)}$, $ \quad z_{ (\<generic>,0)}z_{ (\<generic>,1)}^2$,  &  \<Xi3a> \,,
\<Xi3> \,,&\\ $(-\frac{1}{2})^-$& $ z_{ (\<generic>,0)}^2z_{ (\<generic>,1)}z_{( \<derivative2>,0)}$, &\{ \<I1IXi3b> \,, \<I1IXi3>\,\}    & 5 \\  & $ z_{ (\<generic>,0)}^3 z_{( \<derivative2>,0)}^2$,     $\quad z_{ (\<generic>,0)}^3z_{( \<derivative2>,1)}$&  \<I1IXi3c> \,, \<I1Xi3c>\, & \\  & & &\\ \hline
			 &&&\\ & $z_{ (\<generic>,0)}z_{ (\<generic>,1)}^3$, $\quad  z_{ (\<generic>,0)}^3z_{ (\<generic>,3)} $,  & \<Xi4> \,, \;  \<Xi4b>\,, \;  & \\ & $z_{ (\<generic>,0)}^2z_{ (\<generic>,1)}z_{ (\<generic>,2)}$,& \{ \<Xi4c> \,, \<Xi4e> \},& \\  & $ z_{ (\<generic>,0)}^2z_{ (\<generic>,1)}^2z_{(\<derivative2>,0)}$,  & \{\<I1Xi4a> \,,  \<I1Xi4b> \,, \<I1Xi4c> \,, \<2I1Xi4b> \}\,, & \\ & $ z_{ (\<generic>,0)}^3z_{ (\<generic>,1)}z_{( \<derivative2>,0)}$,   & \{\<Xi4eb> \,,  \<Xi4cb> \},   \;  & 10\\ $ (0)^- $ &$ z_{ (\<generic>,0)}^3z_{ (\<generic>,1)}z_{(\<derivative2>,0)}^2 $,& \{ \<I1Xi4ab> \,, \<I1Xi4bc> \,, \<I1Xi4ac> \,, \<2I1Xi4c> \},&\\&  $z_{ (\<generic>,0)}^4 z_{( \<derivative2>,0)}^3$, & \{\<Xitwo> \,, \<2I1Xi4> \},& \\ &$ z_{ (\<generic>,0)}^4z_{( \<derivative2>,0)}z_{(\<derivative2>,1)},$& \{  \<Xi4eab> \,, \<Xi4eabbis> \,, \<Xi4cab>  \},\;& \\ &  $ z_{ (\<generic>,0)}^4 z_{(\<derivative2>,2)} $, $\quad  z_{ (\<generic>,0)}^3 z_{ (\<generic>,1)}z_{( \<derivative2>,1)} $& \; \<Xi4ba>, \{\<Xi4ca> \,,  \<Xi4ea> \,, \<Xi4eabis>\} &\\ &&& \\
\hline 
\end{tabular}
		\caption{All reduced multi-indices when  $\delta=(1/2)^-$  with their tree representation.}
		\label{tab:kpz_relevant}
\end{table}

Similarly to what we have done with $T_-$, we also list in Table~\ref{tab:kpz_relevant} the set $\mathfrak{M}_{ \<derivative2>, \<generic>}^- $ when $\delta=1^-$. Recalling the map $\Psi$ from \eqref{def_Psi}, it is trivial to remark that  the inverse image $\Psi^{-1}(\mathfrak{M}_{ \<derivative2>, \<generic>})$ coincides with $\mathfrak{T}_{ \<derivative2>, \<generic>}$, the set of decorated trees with at least two noises among their nodes, no polynomial decorations on the nodes and 
at each edge $\mathcal{I}$ we attach at least one  noise or two incoming edges $\mathcal{I}_1\mathcal{I}_1$. These are called saturated trees, see \cite[Sec. 2]{BGHZ} and the map $\Psi$ can be easily expressed via the graphical notation

\begin{equs}[def_graph_psi]
	\begin{aligned}
	\Psi \left( \begin{tikzpicture}[scale=0.4,baseline=-2]
		\coordinate (root) at (0,-0.4);
		\coordinate (t2) at (-0.8,0.5);
		\coordinate (t3) at (0.8,0.5);
		\draw[symbols] (t2) -- (root);
		\draw[symbols] (t3) -- (root);
	\draw (0,-0.4) node[] {\<generic>};
		\draw (-0.9,0.7) node[] {\tiny$\sigma_1$};
		\draw (0.9,0.7) node[] {\tiny$\sigma_n$};
		\draw (0,0.7) node[] {\tiny$\cdots$};
	\end{tikzpicture} \right) &= z_{(\<generic>,m)}  \prod_{i=1}^m \Psi(\sigma_i), \\	\Psi \left( \begin{tikzpicture}[scale=0.4,baseline=-2]
		\coordinate (root) at (0,-0.4);
		\coordinate (t1) at (-0.3,0.5);
		\coordinate (t2) at (-1.1,0.5);
		\coordinate (t3) at (1.1,0.5);
		\coordinate (t4) at (1.1,0.5);
		\coordinate (tau1) at (1,0.6);
		\draw[kernels2] (t1) -- (root);
		\draw[kernels2] (t2) -- (root);
		\draw[symbols] (t3) -- (root);
		\node[not] (rootnode) at (root) {};
		\draw (-1.2,0.7) node[] {\tiny$\tau_1$};
		\node[not] (rootnode) at (root) {};
		\draw (1.3,0.7) node[] {\tiny$\tau_n$};
		\draw (-0.4,0.7) node[] {\tiny$\tau_{\tiny{2}}$};
			\draw (0.5,0.7) node[] {\tiny$\cdots$};
	\end{tikzpicture} \right) & = z_{(\begin{tikzpicture}[scale=0.1,baseline=-2]
			\coordinate (root) at (0,-0.4);
			\coordinate (t1) at (-.8,1.3);
			\coordinate (t2) at (.8,1.3);
			\draw[kernels2] (t1) -- (root);
			\draw[kernels2] (t2) -- (root);
			\node[not] (rootnode) at (root) {};
		\end{tikzpicture},n-2)} \prod_{i=1}^n \Psi(\tau_i). 
	\end{aligned}
\end{equs}
where with $\cdots$ means that the other trees $ \sigma_i, \tau_i $ are connected to the root via a blue edge. These trees are also present in  Table~\ref{tab:kpz_relevant}.

Once an ambient space for the renormalisation constants has been fixed, these quantities will be paired, in Section~\ref{sec:4}, with a fundamental family of non-linear maps built from $\mathcal{M}_{ \<derivative2>, \<generic>}$ (or $\mathcal{M}_{ \<derivative2>, \<generic>}^G$ when the underlying noise is Gaussian). This family allows us to describe the formal Taylor expansions of equation \eqref{eq:gen_KPZ}.

\begin{definition} \label{def_Uspilon}
We define the evaluation map $ \Upsilon_{\Gamma,\sigma} \colon  	\mathcal{M}_{ \<derivative2>, \<generic>} \rightarrow \mathcal{C}^{\infty}(\mathbb{R}, \mathbb{R})$ as 
	\begin{equs}\label{def:ev_map}
	\Upsilon_{\Gamma,\sigma}(z^{\beta})(u) =  \prod_{k \geq 0}\left( 2  \Gamma^{(k)}(u)  \right)^{\beta( \<derivative2>,k )}  
	\prod_{k \geq 0}\left( \sigma^{(k)}(u)  \right)^{\beta(\<generic>,k)}.
	\end{equs}
for any monomial $z^{\beta}\in \mathfrak{M}_{ \<derivative2>, \<generic>}$.
\end{definition}

\begin{remark}
This map already appears in \cite[Example 2.6]{BL23}. However, in the present setting we restrict its definition to a reduced family of multi-indices, in accordance with Proposition~\ref{prop:reduction}, since only these contribute to the renormalisation procedure. Moreover, the factorial terms from that definition will be added later in Section~\ref{sec:4}.
\end{remark}

\section{Symmetries of the evaluation map}
\label{Sec::3}

We now study the symmetry properties of the map $\Upsilon_{\Gamma,\sigma}$ from Definition~\ref{def_Uspilon}. A first-key aspect is its behavior under changes of variable. For any diffeomorphism $\phi\colon \mathbb{R}\to \mathbb{R}$, we define an action of $\phi$ on the nonlinearities $\sigma, \Gamma\colon \mathbb{R}\to \mathbb{R}$. If $u$ formally satisfies \eqref{eq:gen_KPZ}, then the transformed coefficients $\phi \cdot \sigma$ and $\phi \cdot \Gamma$ are defined so that $v = \phi(u)$ formally solves the transformed equation. 

Indeed, setting $u = \phi^{-1}(v)$ and applying the chain rule, we find that $v$ satisfies
\[
\partial_t v =
\partial_x^{2} v +
(\phi\cdot \Gamma)(v)\,(\partial_x v)^{2} +
(\phi\cdot g)(v)\,\partial_x v+ (\phi\cdot h)(v)
+ (\phi\cdot \sigma)(v)\,\xi ,
\]
where the transformed coefficients are determined implicitly by the following identities:
\begin{align} 
\label{eq:phi_action_sigma}
&(\phi \cdot \sigma)(\phi(u)) = \phi'(u)\, \sigma(u),\\ 
\label{eq:phi_action_Gamma}
&(\phi \cdot \Gamma)(\phi(u))\,(\phi'(u))^2 = \phi'(u)\, \Gamma(u) - \phi''(u),\\ 
&(\phi \cdot g)(\phi(u)) = g(u),\\ 
&(\phi \cdot h)(\phi(u)) = \phi'(u)\, h(u).
\end{align}

Since only $\Gamma$ and $\sigma$ appear in the definition of $\Upsilon_{\Gamma,\sigma}$ (see \eqref{def:ev_map}), we focus exclusively on the actions \eqref{eq:phi_action_sigma} and \eqref{eq:phi_action_Gamma}. Note that the action \eqref{eq:phi_action_sigma} extends naturally to any smooth function $f\colon \mathbb{R}\to \mathbb{R}$. By applying this transformation rule to elements in the image of $\Upsilon_{\Gamma,\sigma}$, we obtain a rigorous symmetry transformation of our system. Since it describes invariance by diffeomorphisms we call it geometric invariance.

\begin{definition} \label{def_symmetries}
We define the subspace $V_{\geo}\subset \mathcal{M}_{ \<derivative2>, \<generic>}$ as the subspace generated by those elements $v\in \mathcal{M}_{ \<derivative2>, \<generic>}$ such that for all choices of $\Gamma$, $\sigma$ and all diffeomorphisms $\phi \colon \mathbb{R} \to \mathbb{R}$ homotopic to the identity, one has 
	\begin{equs}
\varphi\cdot \Upsilon_{\Gamma  ,\sigma}(v) = \Upsilon_{\varphi \cdot\Gamma \, \varphi \cdot \sigma}(v)\,.
\end{equs}
\end{definition}

Given $N \geq 2$ integer, we set $V_{\geo}^N$ to be the elements of $V_{\geo}$ such that $((\beta))=N$. Thanks again to \cite[Lem. 2.23]{BL23}, for any  admissible  and spatially symmetric noise $\xi$, there exists $N_{\xi} \in \mathbb{N}^* $ such that  $\cup_{i=2}^{N_{\xi}} \mathfrak{M}^i_{ \<derivative2>, \<generic>} \subset \mathfrak{M}_{ \<derivative2>, \<generic>}^-\,$. We will use this index $N_{\xi}$ in the sequel  to denote
 \begin{equation}\label{geo_space}
V_{\geo}^{\xi} = \bigoplus_{i=2}^{N_{\xi}} V_{\geo}^{i}\,,\quad  V_{\geo}^{\xi, \text{G}} = \bigoplus_{i=2}^{\lfloor N_{\xi}/2 \rfloor} V_{\geo}^{2i}\,,
\end{equation}	
where the second one is considered only in the case of an underlying stationary Gaussian noise $\xi$.

In addition to the change of variable property, we also want to consider another symmetry coming directly from a stationary Gaussian noise $\xi$ whose periodic extension is admissible. Indeed it follows directly from the hypothesis that one has the equality in law \[\xi \stackrel{\mathrm{law}}{=} T\zeta\,,\]where $\zeta$ denotes space--time white noise on $\mathbb{R}\times\mathbb{T}$ and $T$ is a deterministic pseudo-differential operator of finite order. Therefore, modulo a deterministic term $T$ the usual It\^o applies in this context and we can define the second symmetry on $\Upsilon_{\Gamma, \sigma}$.

\begin{definition} \label{def_symmetries_Ito}
We define the space $V_{\Ito} \subset \mathcal{M}_{\<derivative2>, \<generic>}^G$
as the subspace generated by those elements $v\in \mathcal{M}_{ \<derivative2>, \<generic>}^G$  such that, for all $\Gamma$, $\sigma$, and $\bar{\sigma}$, the equality
$\sigma^2 = \bar{\sigma}^2$ implies the identity \[\Upsilon_{\Gamma, \sigma}(v) = \Upsilon_{\Gamma, \bar{\sigma}}(v)\,.\]
\end{definition}

Finally, we also impose a last structural property, motivated by the fact that in equation~\eqref{eq:gen_KPZ}, if one has $\Gamma = 0$ and $\sigma$ constant, then \eqref{eq:gen_KPZ} reduces to the stochastic heat equation with additive noise, which requires no renormalisation.
Therefore, if the nonlinearities vanish at a point $u_0 \in \mathbb{R}$, we would like any reasonable solution to~\eqref{eq:gen_KPZ} that remains sufficiently close to $u_0$ not to require renormalisation counterterms. We call such counterterms ``nice''. This leads to the following definition.

\begin{definition} \label{def_symmetries_nice}
We define the space $V_{\nice} \subset \mathcal{M}_{\<derivative2>, \<generic>}$
as the subspace generated by those elements $v\in \mathcal{M}_{ \<derivative2>, \<generic>}$  such that, for all $\Gamma$, $\sigma$, if $\Gamma(u_0)=0 $ and $\sigma'(u_0)=0$ for some $u_0\in \mathbb{R}$ then
\[\Upsilon_{\Gamma, \sigma}(v)(u_0) = 0\,.\]
\end{definition}
In this section we will carefully focus on describing all of these symmetries.

\subsection{Geometric invariance}

We first focus on the geometric invariance. To understand it, we consider how an infinitesimal change of coordinates $h \in \mathcal{C}^{\infty}(\mathbb{R}, \mathbb{R})$ is acting on the non-linearities $\Gamma$ and $\sigma$. To encode this perturbation at the level of multi-indices, we add a new element to $\mathcal{L}$. This symbol is denoted by $
\begin{tikzpicture}[scale=0.2,baseline=-2]
  \coordinate (root) at (0,0);
  \node[diff] (rootnode) at (root) {};
\end{tikzpicture}$ 
and its non-linearities are given by
\[ \mathcal{N}_{\begin{tikzpicture}[scale=0.2,baseline=-2]
		\coordinate (root) at (0,0);
		\node[diff] (rootnode) at (root) {};
	\end{tikzpicture}}=\{ ke_{(0,0)} \colon  k \in \mathbb{N}\}\,.
\]
As before, we repeat the procedure of defining multi-indices with these new definitions by setting
\[
\mathcal N^*
  = \mathcal N_{0}
  \sqcup \mathcal N_{\Xi}
  \sqcup \mathcal N_{\begin{tikzpicture}[scale=0.2,baseline=-2]
    \coordinate (root) at (0,0);
    \node[diff] (rootnode) at (root) {};
  \end{tikzpicture}},
\qquad
\mathcal R^*
  = \mathcal N^* \sqcup \mathbb N^2,
\]
and define $\mathfrak M^* = M(\mathcal R^*)$. Using again the classical identification between multi-indices and polynomials, we also use the shorthand notation
$z_{(\begin{tikzpicture}[scale=0.2,baseline=-2]
	\coordinate (root) at (0,0);
	\node[diff] (rootnode) at (root) {};
\end{tikzpicture},k)} 
$ to denote $z_{(\begin{tikzpicture}[scale=0.2,baseline=-2]
	\coordinate (root) at (0,0);
	\node[diff] (rootnode) at (root) {};
\end{tikzpicture},k e_{0,0})}$. Both definitions of $((\beta))$ and $[\beta]$ on $\mathfrak M$ extend naturally to $\mathfrak M^*$ by setting
\[
((\beta))
  = \sum_{k \in \mathcal N_{\Xi}} \beta(k),
\qquad
[\beta]
  = \sum_{k \in \mathcal N^*} \bigl(1 - \ell^*(k)\bigr)\beta(k)
    + \sum_{\mbn \in \mathbb N^2} \beta(\mbn),
\]
where $\ell^*(k) = \sum_{u \in \mathcal R^*} k(u)$. (we keep the same notation as before for the sake of simplicity).
	
	 Below, we combine the notion of  reduced multi-indices with the notion this new label to obtain various sets and vector spaces of populated multi-indices with the extra variables $ z_{(\begin{tikzpicture}[scale=0.2,baseline=-2]
		\coordinate (root) at (0,0);
		\node[diff] (rootnode) at (root) {};
	\end{tikzpicture},k)}  $,  $k\geq 0$.
	\begin{align*}
	\mathfrak{M}_{\begin{tikzpicture}[scale=0.2,baseline=-2]
			\coordinate (root) at (0,0);
			\node[diff] (rootnode) at (root) {};
		\end{tikzpicture}, \<derivative2>, \<generic>} & = \lbrace \prod_{k \geq 0}\left(  z_{(\begin{tikzpicture}[scale=0.2,baseline=-2]
	\coordinate (root) at (0,0);
	\node[diff] (rootnode) at (root) {};
\end{tikzpicture},k)}  \right)^{\beta(\begin{tikzpicture}[scale=0.2,baseline=-2]
	\coordinate (root) at (0,0);
	\node[diff] (rootnode) at (root) {};
\end{tikzpicture},k)}  \prod_{k \geq 0}\left( z_{( \<derivative2>,k)} \right)^{\beta( \<derivative2>,k )} \prod_{k \geq 0}\left(  z_{(\<generic>,k)}\right)^{\beta(\<generic>,k)}\,\colon
[\beta] =1, \,((\beta))>2\rbrace\,,\\ \mathfrak{M}_{\begin{tikzpicture}[scale=0.2,baseline=-2]
		\coordinate (root) at (0,0);
		\node[diff] (rootnode) at (root) {};
	\end{tikzpicture}, \<derivative2>, \<generic>}^N & = \lbrace z^{\beta}  \in \mathfrak{M}_{\begin{tikzpicture}[scale=0.2,baseline=-2]
	\coordinate (root) at (0,0);
	\node[diff] (rootnode) at (root) {};
\end{tikzpicture}, \<derivative2>, \<generic>} : ((\beta)) =  N 
\rbrace, \\	\mathcal{M}_{\begin{tikzpicture}[scale=0.2,baseline=-2]
		\coordinate (root) at (0,0);
		\node[diff] (rootnode) at (root) {};
	\end{tikzpicture}, \<derivative2>, \<generic>}  &=  \langle	\mathfrak{M}_{ \begin{tikzpicture}[scale=0.2,baseline=-2]
	\coordinate (root) at (0,0);
	\node[diff] (rootnode) at (root) {};
\end{tikzpicture},\<derivative2>, \<generic>} \rangle, \quad \mathcal{M}^N_{\begin{tikzpicture}[scale=0.2,baseline=-2]
	\coordinate (root) at (0,0);
	\node[diff] (rootnode) at (root) {};
\end{tikzpicture}, \<derivative2>, \<generic>}  =  \langle	\mathfrak{M}^N_{\begin{tikzpicture}[scale=0.2,baseline=-2]
\coordinate (root) at (0,0);
\node[diff] (rootnode) at (root) {};
\end{tikzpicture}, \<derivative2>, \<generic>} \rangle\,,.
	\end{align*}

All of the spaces introduced above are equipped with a natural scalar product
$\langle \,\cdot , \cdot \rangle$, defined by the condition
$\langle z^{\beta}, z^{\beta'} \rangle = \delta_{\beta,\beta'}$. In particular,
for any $v$ in one of these vector spaces and any $z^{\beta}$ in the corresponding
generating set, the quantity $\langle z^{\beta}, v \rangle$ is the coefficient of
$z^{\beta}$ in $v$.

To encode an infinitesimal change of coordinates along a direction
$h \in \mathcal{C}^{\infty}(\mathbb{R}, \mathbb{R})$ at the level of
$\Upsilon_{\Gamma, \sigma}$, we introduce the perturbed evaluation map
\[
\Upsilon^h_{\Gamma, \sigma} \colon
\mathcal{M}_{\begin{tikzpicture}[scale=0.2,baseline=-2]
  \coordinate (root) at (0,0);
  \node[diff] (rootnode) at (root) {};
\end{tikzpicture}, \<derivative2>, \<generic>}
\longrightarrow \mathcal{C}^{\infty}(\mathbb{R}, \mathbb{R}),
\]
defined by
\begin{equs}
\Upsilon^h_{\Gamma, \sigma}(z^{\beta})(u)
=
\prod_{k \geq 0}
\left( h^{(k)}(u) \right)^{\beta(\begin{tikzpicture}[scale=0.2,baseline=-2]
  \coordinate (root) at (0,0);
  \node[diff] (rootnode) at (root) {};
\end{tikzpicture},k)}
\prod_{k \geq 0}
\left( 2\,\Gamma^{(k)}(u) \right)^{\beta(\<derivative2>,k)}
\prod_{k \geq 0}
\left( \sigma^{(k)}(u) \right)^{\beta(\<generic>,k)} .
\end{equs}
for any
$z^{\beta} \in
\mathfrak{M}_{\begin{tikzpicture}[scale=0.2,baseline=-2]
  \coordinate (root) at (0,0);
  \node[diff] (rootnode) at (root) {};
\end{tikzpicture}, \<derivative2>, \<generic>}$.
This map preserves injectivity with respect to the extended multi-indices.


\begin{theorem} \label{injectivity_Upsilon}
	The map $  \Upsilon_{\Gamma,\sigma}^{h} $ is injective in the sense that if for any $\Gamma,\sigma,h \in \mathcal{C}^{\infty}(\mathbb{R}, \mathbb{R})$ one has $\Upsilon_{\Gamma,\sigma}^h(v) = 0 $ for some $v\in \mathcal{M}_{\begin{tikzpicture}[scale=0.2,baseline=-2]
  \coordinate (root) at (0,0);
  \node[diff] (rootnode) at (root) {};
\end{tikzpicture}, \<derivative2>, \<generic>}$ then $ v=0$.
\end{theorem}
\begin{proof}
 Writing the hypothesis in coordinates, this means that for any $\Gamma,\sigma,h \in \mathcal{C}^{\infty}(\mathbb{R}, \mathbb{R})$ one has
	\begin{equs}
			 \sum_{\beta} \lambda_{\beta} \Upsilon_{\Gamma, \sigma}^h(z^\beta) = 0 \,,
	\end{equs}
where the sum runs over all extended  multi-indices $ \beta \in \mathfrak{M}_{\begin{tikzpicture}[scale=0.2,baseline=-2]
 		\coordinate (root) at (0,0);
 		\node[diff] (rootnode) at (root) {};
 	\end{tikzpicture}, \<derivative2>, \<generic>}$  and the coefficients $ \lambda_{\beta} $  are all zero with the exception of a finite number. Then our thesis is to prove $\lambda_{\beta}=0$ for all $\beta$. We can therefore choose specific $\Gamma,\sigma,h$ for discriminating the various $ z^{\beta} $. Let $ m $ be the highest integer such that  there exists $ z^{\beta} $ with $  \lambda_{\beta} \neq 0 $ and $ (\beta(\begin{tikzpicture}[scale=0.2,baseline=-2]
			\coordinate (root) at (0,0);
			\node[diff] (rootnode) at (root) {};
		\end{tikzpicture},m),
	\beta( \<derivative2>,m ), \beta(\<generic>,m)) \neq (0,0,0)$. We define
	\begin{equs}
		h(x) & =	h_{r_0,...,r_m}(x) = \sum_{k=0}^m r_k \frac{x^k}{k!}, \quad \Gamma(x) =	\Gamma_{s_0,...,s_m}(x) = \sum_{k=0}^m s_k \frac{x^k}{k!},
		\\  \sigma(x) & =	\sigma_{t_0,...,t_m}(x) = \sum_{k=0}^m t_k \frac{x^k}{k!}.
	\end{equs} 
	By choosing  functions $\Gamma,\sigma,h$ parametrised by $r_0,s_0,t_0,...,r_m,s_m,t_m$  one has
	\begin{equs}
		\Upsilon^h_{\Gamma,\sigma}(z_\beta)(0) =  \prod_{k=0}^m \left( 
		 \left( 	r_k \right)^{\beta(\begin{tikzpicture}[scale=0.2,baseline=-2]
				\coordinate (root) at (0,0);
				\node[diff] (rootnode) at (root) {};
			\end{tikzpicture},k)}
	\left( 2  s_k  \right)^{\beta( \<derivative2>,k )} \left( 	t_k  \right)^{\beta(\<generic>,k)} \right).
	\end{equs}
	One can observe that we obtain a monomial in $r_0,s_0,t_0,...,r_m,s_m,t_m$ which is uniquely associated to a multi-index $\beta$. Since this family of monomials is clearly free and describes smooth functions up to Taylor polynomials, we can conclude that $v=0$.
	\end{proof}
At the level of multi-indices we define $ [\cdot, \cdot] \colon \mathcal{M}_{\begin{tikzpicture}[scale=0.2,baseline=-2]
 		\coordinate (root) at (0,0);
 		\node[diff] (rootnode) at (root) {};
 	\end{tikzpicture}, \<derivative2>, \<generic>} \times \mathcal{M}_{\begin{tikzpicture}[scale=0.2,baseline=-2]
 		\coordinate (root) at (0,0);
 		\node[diff] (rootnode) at (root) {};
 	\end{tikzpicture}, \<derivative2>, \<generic>} \to \mathcal{M}_{\begin{tikzpicture}[scale=0.2,baseline=-2]
 		\coordinate (root) at (0,0);
 		\node[diff] (rootnode) at (root) {};
 	\end{tikzpicture}, \<derivative2>, \<generic>}  $ as
\begin{equs}
\,	 [v_1, v_2] = v_1 D v_2 - v_2 D v_1
 \end{equs}
where $D\colon \mathcal{M}_{\begin{tikzpicture}[scale=0.2,baseline=-2]
 		\coordinate (root) at (0,0);
 		\node[diff] (rootnode) at (root) {};
 	\end{tikzpicture}, \<derivative2>, \<generic>} \to \mathcal{M}_{\begin{tikzpicture}[scale=0.2,baseline=-2]
 		\coordinate (root) at (0,0);
 		\node[diff] (rootnode) at (root) {};
 	\end{tikzpicture}, \<derivative2>, \<generic>} $ is the derivation on multi-indices defined by
\begin{equs}
	D   	z_{ (\begin{tikzpicture}[scale=0.2,baseline=-2]
			\coordinate (root) at (0,0);
			\node[diff] (rootnode) at (root) {};
		\end{tikzpicture},k)}   =  	z_{ (\begin{tikzpicture}[scale=0.2,baseline=-2]
		\coordinate (root) at (0,0);
		\node[diff] (rootnode) at (root) {};
	\end{tikzpicture},k+1)}, \quad D  z_{( \<derivative2>,k)}    =  z_{( \<derivative2>,k+1)}, \quad D	z_{ (\<generic>,k)}  = 	z_{ (\<generic>,k+1)}
\end{equs}
and then extended to $\mathcal{M}_{\begin{tikzpicture}[scale=0.2,baseline=-2]
 		\coordinate (root) at (0,0);
 		\node[diff] (rootnode) at (root) {};
 	\end{tikzpicture}, \<derivative2>, \<generic>} $ via the Leibniz rule.

\begin{definition} \label{def:injectivity_Upsilon}
We first define the geometric map $\phi_\geo : \mathcal{M}_{ \<derivative2>, \<generic>} \rightarrow \mathcal{M}_{\begin{tikzpicture}[scale=0.2,baseline=-2]
		\coordinate (root) at (0,0);
		\node[diff] (rootnode) at (root) {};
	\end{tikzpicture}, \<derivative2>, \<generic>}$  by setting 
\begin{equs}
		\phi_\geo(	z_{ (\<generic>,k)})  &= D^{k}  [	z_{ (\<generic>,0 )}, z_{ (\begin{tikzpicture}[scale=0.2,baseline=-2]
			\coordinate (root) at (0,0);
			\node[diff] (rootnode) at (root) {};
		\end{tikzpicture},0)}] \\
	\phi_\geo( z_{( \<derivative2>,k)} ) & = - D^{k+1} \left( 	z_{( \<derivative2>,0)} z_{ (\begin{tikzpicture}[scale=0.2,baseline=-2]
		\coordinate (root) at (0,0);
		\node[diff] (rootnode) at (root) {};
	\end{tikzpicture},0)} 
\right) - 2  z_{ (\begin{tikzpicture}[scale=0.2,baseline=-2]
		\coordinate (root) at (0,0);
		\node[diff] (rootnode) at (root) {};
	\end{tikzpicture},k+2) }
\end{equs}
and then it is extended via the Leibniz rule.
We then define the compensated geometric map $ \hat{\phi}_\geo \colon  \mathcal{M}_{ \<derivative2>, \<generic>} \rightarrow \mathcal{M}_{\begin{tikzpicture}[scale=0.2,baseline=-2]
		\coordinate (root) at (0,0);
		\node[diff] (rootnode) at (root) {};
	\end{tikzpicture}, \<derivative2>, \<generic>} $ by simply putting 
\begin{equs}
	\hat{\phi}_\geo(z^{\beta}) = \phi_\geo(z^{\beta}) - [ z^{\beta}, z_{ (\begin{tikzpicture}[scale=0.2,baseline=-2]
			\coordinate (root) at (0,0);
			\node[diff] (rootnode) at (root) {};
		\end{tikzpicture},0)} ]\,.
\end{equs}
\end{definition}
We observe that 
\begin{equs}[leibniz_calculus]
	&D^{k}  [	z_{ (\<generic>,0 )}, z_{ (\begin{tikzpicture}[scale=0.2,baseline=-2]
			\coordinate (root) at (0,0);
			\node[diff] (rootnode) at (root) {};
		\end{tikzpicture},0)}] = D^{k} \left( 	z_{ (\<generic>,0)}  z_{ (\begin{tikzpicture}[scale=0.2,baseline=-2]
		\coordinate (root) at (0,0);
		\node[diff] (rootnode) at (root) {};
	\end{tikzpicture},1)} \right) - D^{k} \left(	z_{ (\<generic>,1)}  z_{ (\begin{tikzpicture}[scale=0.2,baseline=-2]
	\coordinate (root) at (0,0);
	\node[diff] (rootnode) at (root) {};
\end{tikzpicture},0)}  \right)\\&= z_{ (\<generic>,0)}z_{ (\begin{tikzpicture}[scale=0.2,baseline=-2]
			\coordinate (root) at (0,0);
			\node[diff] (rootnode) at (root) {};
		\end{tikzpicture},k+1)} - z_{ (\<generic>,k+1)}z_{ (\begin{tikzpicture}[scale=0.2,baseline=-2]
			\coordinate (root) at (0,0);
			\node[diff] (rootnode) at (root) {};
		\end{tikzpicture},0)} + \sum_{l=1}^{k}\frac{k!(k-2l+1)}{l!(k-l+1)!} z_{ (\<generic>,l)}z_{ (\begin{tikzpicture}[scale=0.2,baseline=-2]
			\coordinate (root) at (0,0);
			\node[diff] (rootnode) at (root) {};
		\end{tikzpicture},k-l+1)}. 
\end{equs}
This identity is useful for the computation of the examples later on. Moreover, it is straightforward to check that the maps $ 	\phi_\geo  $ and $ 	\hat{\phi}_\geo  $ are well-defined  preserving the populated condition defining  $\mathcal{M}_{\begin{tikzpicture}[scale=0.2,baseline=-2]
		\coordinate (root) at (0,0);
		\node[diff] (rootnode) at (root) {};
	\end{tikzpicture}, \<derivative2>, \<generic>}$ and the vector spaces $\phi_\geo(\mathcal{M}_{ \<derivative2>, \<generic>} ) $,  $\hat{\phi}_\geo(\mathcal{M}_{ \<derivative2>, \<generic>} ) $ lie in the subspace of $ \mathfrak{M}_{ \begin{tikzpicture}[scale=0.2,baseline=-2]
		\coordinate (root) at (0,0);
		\node[diff] (rootnode) at (root) {};
	\end{tikzpicture},\<derivative2>, \<generic>} $ generated by the elements that contain only one variable of type $ z_{(\begin{tikzpicture}[scale=0.2,baseline=-2]
	\coordinate (root) at (0,0);
	\node[diff] (rootnode) at (root) {};
\end{tikzpicture},k)} $ (i.e. $\beta(\begin{tikzpicture}[scale=0.2,baseline=-2]
	\coordinate (root) at (0,0);
	\node[diff] (rootnode) at (root) {};
\end{tikzpicture},k)=1$). Using the map $\hat{\varphi}_{\geo}$, we can characterize the geometric terms of $\Upsilon_{\Gamma, \sigma}$ by examining the kernel of $\hat{\varphi}_{\geo}$.

\begin{theorem} \label{geo_chain_rule} One has $v \in	V_{\geo} $ if and only if $ \hat{\varphi}_{\geo}(v) = 0$.
\end{theorem}
\begin{proof}
	We only prove the if part as the proof for the only if side is verbatim the same as the proof of \cite[Prop. 6.2]{BGHZ}. For $v \in 	V_{\geo}$ we apply the geometric identity to a family of maps $ (\psi_{t})_{t \geq 0}  $ with $ \psi_0 = \id, \partial_t \psi |_{t=0} = h $. By hypothesis on $v$ one has
	\begin{equs}
			\psi_t \cdot \Upsilon_{ \Gamma, \sigma}(v) = \Upsilon_{\psi_t \cdot \Gamma, \psi_t \cdot \sigma}(v).
	\end{equs}
Then, the main idea is to take the derivative at time $t=0$ on both sides of the previous equality:
\begin{equs}
	\partial_t	\left( \psi_t \cdot \Upsilon_{ \Gamma, \sigma}(v) \right)|_{t=0} =	\partial_t	\left( \Upsilon_{\psi_t \cdot \Gamma, \psi_t \cdot \sigma}(v)\right)|_{t=0}.
\end{equs}
In the sequel, we compute the two parts of the equality separately.
One first has
	\begin{equs}
	\partial_t	\left(	\psi_t \cdot \Upsilon^h_{ \Gamma, \sigma}(v) \right) |_{t=0} & = \partial_t	\left(	(\psi_t' \, \Upsilon^h_{ \Gamma, \sigma}(v)) \circ \psi_t^{-1} \right)|_{t=0} \\
	& = 	h'  \Upsilon_{ \Gamma, \sigma}(v) - \partial_u \Upsilon_{ \sigma, \Gamma}(v) h 
	\\ & =\Upsilon^h_{\Gamma,\sigma}\left( [v, z_{ (\begin{tikzpicture}[scale=0.2,baseline=-2]
			\coordinate (root) at (0,0);
			\node[diff] (rootnode) at (root) {};
		\end{tikzpicture},0)} ] \right).
	\end{equs}
Then, on the other hand, one has
\begin{equs}
	\partial_t \Upsilon_{\psi_t \cdot \Gamma, \psi_t \cdot \sigma}(  z_{ (\<generic>,k)} )|_{t=0} &= 	\partial_t \left(  (\psi_t' \, \sigma) \; \circ \psi_t^{(-1)} \right)^{(k)}  |_{t=0} 
	\\ &= \left(\partial_t \left(  (\psi_t' \, \sigma) \; \circ \psi_t^{(-1)} \right) |_{t=0} \right)^{(k)}
	\\ & =  \left( h' \sigma - h \sigma'  \right)^{(k)}  
	\\ &= \Upsilon^{h}_{\Gamma,\sigma} \left(  [	z_{ (\<generic>,0 )}, z_{ (\begin{tikzpicture}[scale=0.2,baseline=-2]
			\coordinate (root) at (0,0);
			\node[diff] (rootnode) at (root) {};
		\end{tikzpicture},0)}] \right)^{(k)}
	\\ & = \Upsilon^{h}_{\Gamma,\sigma} \left(D^{k}  [	z_{ (\<generic>,0 )}, z_{ (\begin{tikzpicture}[scale=0.2,baseline=-2]
			\coordinate (root) at (0,0);
			\node[diff] (rootnode) at (root) {};
		\end{tikzpicture},0)}] \right)
	\\ & = \Upsilon^{h}_{\Gamma,\sigma} [	\phi_\geo(	z_{ (\<generic>,k)}) ].
\end{equs}
One has also
\begin{equs}
		\partial_t \Upsilon_{\psi_t \cdot \Gamma, \psi_t \cdot \sigma}(   z_{(\<derivative2>,k)} )|_{t=0} & =  2 \	\partial_t \left( 
	\frac{	\psi_t' \, \Gamma - \psi_t''}{(\psi_t')^2} \circ \psi_t^{-1} \right)^{(k)}|_{t=0}
	\\ & = 2 \left( \partial_t \left( 
	\frac{	\psi_t' \, \Gamma - \psi_t''}{(\psi_t')^2} \circ \psi_t^{-1} \right)|_{t=0} \right)^{(k)}
	\\ &= 2 \left(  h' \Gamma - h \Gamma' -  h'' - 2 h'\Gamma \right)^{(k)}
	\\ & =  \Upsilon^{h}_{\Gamma,\sigma}  \left(  - 	z_{( \<derivative2>,0)} z_{ (\begin{tikzpicture}[scale=0.2,baseline=-2]
			\coordinate (root) at (0,0);
			\node[diff] (rootnode) at (root) {};
		\end{tikzpicture},1)} - z_{( \<derivative2>,1)}
	z_{ (\begin{tikzpicture}[scale=0.2,baseline=-2]
			\coordinate (root) at (0,0);
			\node[diff] (rootnode) at (root) {};
		\end{tikzpicture},0)}  - 2  z_{ (\begin{tikzpicture}[scale=0.2,baseline=-2]
		\coordinate (root) at (0,0);
		\node[diff] (rootnode) at (root) {};
	\end{tikzpicture},2)}
	\right)^{(k)}
	\\ & =   \Upsilon^{h}_{\Gamma,\sigma}  \left(  - D^{k+1}\left(	z_{( \<derivative2>,0)} z_{ (\begin{tikzpicture}[scale=0.2,baseline=-2]
			\coordinate (root) at (0,0);
			\node[diff] (rootnode) at (root) {};
		\end{tikzpicture},0)}  \right)  - 2  z_{ (\begin{tikzpicture}[scale=0.2,baseline=-2]
			\coordinate (root) at (0,0);
			\node[diff] (rootnode) at (root) {};
		\end{tikzpicture},k+2)}
	\right)
	\\ & = \Upsilon^{h}_{\Gamma,\sigma}  \left(  \phi_\geo( z_{( \<derivative2>,k)} ) \right).
\end{equs}
In the end, gathering the various terms, one gets for all  $\Gamma,\sigma,h \in \mathcal{C}^{\infty}(\mathbb{R}, \mathbb{R})$
\begin{equs}
	 \Upsilon^{h}_{\Gamma,\sigma} \left( \phi_\geo(v) \right)
	 =  \Upsilon^{h}_{\Gamma,\sigma} \left([ v, z_{ (\begin{tikzpicture}[scale=0.2,baseline=-2]
	 		\coordinate (root) at (0,0);
	 		\node[diff] (rootnode) at (root) {};
	 	\end{tikzpicture},0)} ] \right).
\end{equs}
Hence we conclude from the injectivity of the map $ \Upsilon_{\Gamma,\sigma}^{h} $ in Theorem~\ref{injectivity_Upsilon}.
\end{proof}

\subsection{Counting geometric terms}
From Theorem~\ref{geo_chain_rule} we can rewrite $V_{\geo}$ as the kernel of the linear map  $\hat{\phi}_\geo$. This identification comes also with an explicit computation of the dimension of this vector space. To achieve it, we introduce a last specific operation on multi-indices: For any couple $v_1, v_2\in\mathcal {M}$	we define the covariant derivative $\Nabla_{v_1} v_2\in \mathcal {M}$  as 
\begin{equs} \label{covariant_derivative}
	\Nabla_{v_1} v_2 = v_1 D v_2 + \frac{1}{2}  z_{( \<derivative2>,0)} v_1 v_2.
\end{equs}
From a simple computation on the fertility index we can indeed check that $\nabla$ is a well defined map $\nabla\colon \mathcal {M}_{ \<derivative2>, \<generic>} \times \mathcal {M}_{ \<derivative2>, \<generic>} \to \mathcal {M}_{ \<derivative2>, \<generic>} $. In particular, it preserves also the structure of $\hat{\phi}_\geo$.
\begin{proposition} \label{geometric_Nabla} For every 	$v_1, v_2 \in	V_{\geo}$, one has
	$\Nabla_{v_1} v_2  \in V_{\geo}$.
	\end{proposition}
\begin{proof}
One has to check that $ 	\Nabla_{v_1} v_2 $ belongs to the kernel of $ \hat{\varphi}_{\geo} $. One first has
\begin{equs}
	\phi_\geo \left( \Nabla_{v_1} v_2  \right) & = \Nabla_{ \phi_\geo \left(v_1\right)} v_2  +  \Nabla_{v_1} \phi_\geo \left(v_2 \right) + \frac{1}{2} \phi_\geo  \left(  z_{( \<derivative2>,0)} \right) v_1 v_2
	\\ & = \Nabla_{ [  v_1, z_{ (\begin{tikzpicture}[scale=0.2,baseline=-2]
				\coordinate (root) at (0,0);
				\node[diff] (rootnode) at (root) {};
			\end{tikzpicture},0)} ]} v_2  +  \Nabla_{v_1} [  v_2, z_{ (\begin{tikzpicture}[scale=0.2,baseline=-2]
		\coordinate (root) at (0,0);
	\node[diff] (rootnode) at (root) {};
\end{tikzpicture},0)} ]  + \frac{1}{2} \phi_\geo  \left(  z_{( \<derivative2>,0)} \right) v_1 v_2
		\end{equs}
	where we have used the fact that  for $i \in \lbrace 1,2 \rbrace$ $\phi_\geo \left(v_i\right) =  [  v_i, z_{ (\begin{tikzpicture}[scale=0.2,baseline=-2]
				\coordinate (root) at (0,0);
				\node[diff] (rootnode) at (root) {};
			\end{tikzpicture},0)} ]$.
Then, continuing the computation we obtain
\begin{equs}
	\Nabla_{ [  v_1, z_{ (\begin{tikzpicture}[scale=0.2,baseline=-2]
				\coordinate (root) at (0,0);
				\node[diff] (rootnode) at (root) {};
			\end{tikzpicture},0)} ]} v_2 & =  v_1 D z_{ (\begin{tikzpicture}[scale=0.2,baseline=-2]
		\coordinate (root) at (0,0);
	\node[diff] (rootnode) at (root) {};
\end{tikzpicture},0)} D v_2 -  D v_1  z_{ (\begin{tikzpicture}[scale=0.2,baseline=-2]
\coordinate (root) at (0,0);
\node[diff] (rootnode) at (root) {};
\end{tikzpicture},0)} D v_2 + \frac{1}{2}  z_{( \<derivative2>,0)} v_1 D z_{ (\begin{tikzpicture}[scale=0.2,baseline=-2]
\coordinate (root) at (0,0);
\node[diff] (rootnode) at (root) {};
\end{tikzpicture},0)}  v_2
\\ &  - \frac{1}{2}  z_{( \<derivative2>,0)} D v_1  z_{ (\begin{tikzpicture}[scale=0.2,baseline=-2]
		\coordinate (root) at (0,0);
		\node[diff] (rootnode) at (root) {};
	\end{tikzpicture},0)}  v_2
	 \\   \Nabla_{v_1} [  v_2, z_{ (\begin{tikzpicture}[scale=0.2,baseline=-2]
			\coordinate (root) at (0,0);
			\node[diff] (rootnode) at (root) {};
		\end{tikzpicture},0)} ] & = z_{ (\begin{tikzpicture}[scale=0.2,baseline=-2]
		\coordinate (root) at (0,0);
		\node[diff] (rootnode) at (root) {};
	\end{tikzpicture},2)}  v_1 v_2 + z_{ (\begin{tikzpicture}[scale=0.2,baseline=-2]
	\coordinate (root) at (0,0);
	\node[diff] (rootnode) at (root) {};
\end{tikzpicture},1)}  v_1 D v_2 - z_{ (\begin{tikzpicture}[scale=0.2,baseline=-2]
\coordinate (root) at (0,0);
\node[diff] (rootnode) at (root) {};
\end{tikzpicture},1)}  v_1 D v_2 - z_{ (\begin{tikzpicture}[scale=0.2,baseline=-2]
\coordinate (root) at (0,0);
\node[diff] (rootnode) at (root) {};
\end{tikzpicture},0)}  Dv_1 Dv_2
\\ & + \frac{1}{2}  z_{( \<derivative2>,0)} v_1 D z_{ (\begin{tikzpicture}[scale=0.2,baseline=-2]
		\coordinate (root) at (0,0);
		\node[diff] (rootnode) at (root) {};
	\end{tikzpicture},0)}  v_2 - \frac{1}{2}  z_{( \<derivative2>,0)} v_1  z_{ (\begin{tikzpicture}[scale=0.2,baseline=-2]
\coordinate (root) at (0,0);
\node[diff] (rootnode) at (root) {};
\end{tikzpicture},0)}  D v_2
 \\  \frac{1}{2} \phi_\geo  \left(  z_{( \<derivative2>,0)} \right) v_1 v_2 & =   - \frac{1}{2} D \left( z_{( \<derivative2>,0)}  z_{ (\begin{tikzpicture}[scale=0.2,baseline=-2]
		\coordinate (root) at (0,0);
		\node[diff] (rootnode) at (root) {};
	\end{tikzpicture},0)} \right) v_1 v_2
 -    z_{ (\begin{tikzpicture}[scale=0.2,baseline=-2]
		\coordinate (root) at (0,0);
		\node[diff] (rootnode) at (root) {};
	\end{tikzpicture},2)}  v_1 v_2.
\end{equs}
	On the other hand, one has
	\begin{equs}
	\	[  \Nabla_{v_1} v_2, z_{ (\begin{tikzpicture}[scale=0.2,baseline=-2]
				\coordinate (root) at (0,0);
				\node[diff] (rootnode) at (root) {};
			\end{tikzpicture},0)} ]
		 & = v_1 D v_2 D z_{ (\begin{tikzpicture}[scale=0.2,baseline=-2]
		 		\coordinate (root) at (0,0);
		 		\node[diff] (rootnode) at (root) {};
		 	\end{tikzpicture},0)} + \frac{1}{2}  z_{( \<derivative2>,0)} v_1 v_2 D z_{ (\begin{tikzpicture}[scale=0.2,baseline=-2]
		 	\coordinate (root) at (0,0);
		 	\node[diff] (rootnode) at (root) {};
	 	\end{tikzpicture},0)}
 	\\ & - D(v_1 D v_2)  z_{ (\begin{tikzpicture}[scale=0.2,baseline=-2]
 			\coordinate (root) at (0,0);
 			\node[diff] (rootnode) at (root) {};
 		\end{tikzpicture},0)} + \frac{1}{2} D \left(  z_{( \<derivative2>,0)} v_1 v_2 \right)  z_{ (\begin{tikzpicture}[scale=0.2,baseline=-2]
 			\coordinate (root) at (0,0);
 			\node[diff] (rootnode) at (root) {};
 		\end{tikzpicture},0)}.
	\end{equs}
Combining the two sides we conclude  that
$\phi_\geo \left( \Nabla_{v_1} v_2  \right) = 	[  \Nabla_{v_1} v_2, z_{ (\begin{tikzpicture}[scale=0.2,baseline=-2]
			\coordinate (root) at (0,0);
			\node[diff] (rootnode) at (root) {};
		\end{tikzpicture},0)} ]$.
\end{proof}
Since one has trivially $\hat{\varphi}_{\geo}(z_{(\<generic>,0)})=0 $ (even if $z_{(\<generic>,0)}$ does not technically belong to $\mathcal {M}_{ \<derivative2>, \<generic>}$) the main idea to describe $V_{\geo}$ is to iterate $\nabla$ in order to fill out the whole space for $V_{\geo}$.

\begin{definition}\label{def_B_xi}
For any integer $N \ge 1$, we define $\mathfrak{B}_N$ to be the set obtained by iterating the covariant derivative $N-1$ times on $z_{(\<generic>,0)}$,  that is we set recursively $\mathfrak{B}_1=\{ z_{(\<generic>,0)}\}$ and for any $N\geq 1$ we define 

\begin{equation}\label{eq:def_Bn}
\mathfrak{B}_{N+1}= \{\Nabla_{v} w\colon  v\in \mathfrak{B}_{k}, w\in \mathfrak{B}_{N+1-k}, \, 1\leq k\leq N\}\,.
\end{equation}
For any admissible and spatially symmetric noise $\xi$, we set
\[
\mathfrak{B}_\xi = \bigcup_{i=2}^{N_\xi} \mathfrak{B}_i\,, \quad \mathfrak{B}_\xi^{\text{G}} = \bigcup_{i=1}^{\lfloor N_\xi/2 \rfloor} \mathfrak{B}_{2i}\,.
\]
where the second one is introduced if the underlying noise is also Gaussian.
\end{definition}
 For instance, one has
\begin{equs}
\mathfrak{B}_{2} &= \left \lbrace	 		\Nabla_{\<generic>}  \<generic>\right \rbrace\,,\quad \mathfrak{B}_{3} = \left \lbrace	 	\Nabla_{\<generic>} 	\Nabla_{\<generic>}  \<generic>, \, 	\Nabla_{\Nabla_{\<generic>}  \<generic> }   \<generic> \right \rbrace\,,\\
	\mathfrak{B}_{4} &= \left \lbrace	\Nabla_{\<generic>} 	\Nabla_{\<generic>} 	\Nabla_{\<generic>}  \<generic>, \, 	\Nabla_{\<generic>} 	\Nabla_{\Nabla_{\<generic>}  \<generic> }   \<generic>, \,  	\Nabla_{\Nabla_{\<generic>}  \<generic> } 	\Nabla_{\<generic>} 	  \<generic> , \,		\Nabla_{\Nabla_{\<generic>} \Nabla_{\<generic>}  \<generic> }   \<generic>, \,	\Nabla_{\Nabla_{\Nabla_{\<generic>} \<generic>}   \<generic> }   \<generic> \right \rbrace\,,
\end{equs}
where we have made the following abuse of notations replacing $ 	z_{ (\<generic>,0)} $ by $ \<generic> $. By construction of $\mathfrak{B}_N$, all of its elements belong to $V_{\geo}^N$. Moreover it follows almost immediately from \eqref{eq:def_Bn} that the cardinality of $\mathfrak{B}_{N}$ is $C_{N-1}$ the $N-1$th Catalan number.

However, as one can immediately see, the elements of $\mathfrak{B}_{N}$ do not constitute a linearly independent family in general, as can be checked from the following identities
\begin{equs}\label{eq:extra_relations}
	\Nabla_{\Nabla_{\<generic>}  \<generic> } 	\Nabla_{\<generic>} 	  \<generic> =	\Nabla_{\Nabla_{\<generic>} \Nabla_{\<generic>}  \<generic> }   \<generic>\,, \, \quad
	2 \,	\Nabla_{\Nabla_{\<generic>} \Nabla_{\<generic>}  \<generic> }   \<generic> - 	\Nabla_{\<generic>} 	\Nabla_{\Nabla_{\<generic>}  \<generic> }   \<generic>   = 	\Nabla_{\Nabla_{\Nabla_{\<generic>} \<generic>}   \<generic> }    \<generic>\,. 
\end{equs}
However, we are able to relate the dimension of the generating vector space to an elementary set of multi-indices. In what follows, we consider the following sets 
\begin{equs}[free_nov_1]
\mathfrak{M}_{\<generic>} & = \lbrace z^{\beta} = \prod_{k \geq 0}\left( z_{(\<generic>,k)}\right)^{\beta(\<generic>,k)} : 
		[\beta] =1 ,\; ((\beta))\geq 2\rbrace, \quad \mathcal{M}_{\<generic>} = 	\langle \mathfrak{M}_{\<generic> }\rangle,  \\ \mathfrak{M}_{\<generic>}^N & = \lbrace z^{\beta}  \in \mathfrak{M}_{\<generic>}  : ((\beta)) =  N 
		\rbrace, \quad \mathcal{M}_{\<generic>}^N = 	\langle \mathfrak{M}_{\<generic> }^N\rangle.
\end{equs}
Using again the map $\Psi$ from \eqref{def_Psi}, one can show that the inverse image $\Psi^{-1}(\mathfrak{M}_{\<generic>})$ coincides with $\mathfrak{T}_{\<generic>}$, the set of trees with cardinality greater than $2$ whose nodes are all decorated by $\Xi$. The linear vector space generated by this set is denoted by $\langle\mathfrak{T}_{\<generic>}\rangle$.

\begin{proposition} \label{lower_bound_dimension}
For any $N\geq 2$ the linear vector space generated by $\mathfrak{B}_N$ is a subspace of 	$V_{\geo}^N$ with dimension $\geq \mathrm{Card}(\mathfrak{M}_{\<generic>}^N)$.
	\end{proposition}
\begin{proof}
By construction of  $\mathfrak{B}_N$, the linear vector space  generated by $\mathfrak{B}_N$ is a subspace of $V_{\geo}^N$. We denote it by $W^N$ and we will use the additional notation $W=\oplus_{N=2}^{\infty} W^N$ . To compute  its dimension we will simply show the existence of an injective linear map $\Lambda\colon \mathcal{M}_{\<generic>}\to W$ which satisfies $\Lambda\colon \mathcal{M}_{\<generic>}^N\to W^N$ and when we compose it with the canonical projection 
 $ \pi \colon \mathcal {M}_{ \<derivative2>, \<generic>}\to \mathcal{M}_{\<generic>} $ one has $\pi \circ \Lambda=\text{id}$. For any given $ z^{\beta} \in  \mathfrak{M}_{\<generic>}$ we then define
\begin{equs} \label{def_psi}
	\Lambda( z^{\beta} ) = \frac{1}{C_{\beta}} \sum_{z^{\beta} = \Psi(\tau)}  \tilde{\psi}(\tau)\,, \quad \mathrm{Card}(\lbrace   \tau : z^{\beta} = \Psi(\tau) \rbrace)\,,
	\end{equs}
where the map $ \tilde{\psi}\colon \mathcal{B}_{ \<generic>} \oplus \langle \<generic>\rangle\to W\oplus \langle z_{(\<generic>,0)}\rangle $ is defined recursively on decorated trees by setting
\begin{equs} \label{def_tilde_psi}
	\tilde{\psi}(\<generic>) = z_{(\<generic>,0)} \,,\quad \tilde{\psi}( \tau \curvearrowright \sigma) = \Nabla_{\tilde{\psi}( \tau)} \tilde{\psi}( \sigma)\,,
\end{equs}
where the  grafting product $ \curvearrowright $ between two trees  $ \sigma $ and $ \tau $ is defined as
\begin{equs}
	\sigma \curvearrowright \tau = \sum_{u \in N_{\tau}} \sigma \curvearrowright_u \tau
\end{equs} 
with$ N_{\tau} $ the nodes of $\tau$ and $ \curvearrowright_u $ connects the root of $ \sigma $ via a new blue edge to the node $u$. As an example, one has
\begin{equs}
	\<Xi2> \curvearrowright	\<Xi2> = \<Xi4> + \<Xi4e>.
\end{equs}
The map $\tilde{\psi}$ is well defined by \eqref{def_tilde_psi} because the operation $ \curvearrowright $ uniquely generates the vector space of trees $\mathcal{B}_{ \<generic>} \oplus \langle \<generic>\rangle$, see \cite{ChaLiv}. Moreover, since the map $\Psi$ contains a sum of trees with     $((\beta))-1$ grafting operations
the map $\Lambda$ trivially factorises into a map $\Lambda\colon \mathcal{M}_{\<generic>}^N\to W^N$ for any $N\geq 2$.  Applying the projection $\pi$ we obtain for any $z^{\beta}\in  \mathfrak{M}_{\<generic>}$
\begin{equs}
		\pi(\Lambda( z^{\beta} )) = \frac{1}{C_{\beta}} \sum_{z^{\beta} = \Psi(\tau)} \pi( \tilde{\psi}(\tau))\,.
\end{equs}
Since for any  $ v_1, v_2\in  \mathcal {M}_{ \<generic>}$ one has 
$\pi \left(  \Nabla_{v_1} v_2  \right) = v_1 D v_2$, we can easily check the identities
\[
	(\pi\circ\tilde{\psi}(\<generic>) )=\Psi(\<generic>) \,,\quad (\pi\circ \tilde{\psi})( \tau \curvearrowright \sigma)= \tilde{\psi}( \tau) D\tilde{\psi}( \sigma)=\Psi(\tau \curvearrowright \sigma)\,.
\]
from which we conclude that $\pi\circ\tilde{\psi}= \Psi$ and we obtain

\begin{equs}
		\pi(\Lambda( z^{\beta} ))  = \frac{1}{C_{\beta}} \sum_{z^{\beta} = \Psi(\tau)}\Psi(  \tau) = z^{\beta}.
\end{equs}
\end{proof}

\begin{remark}
	The proof is an adaptation of the proof of \cite[Prop. 6.13]{BGHZ}. The main difference is the introduction of the use of the map $ \Psi $ for transferring the argument to the context of multi-indices.
	\end{remark}

\begin{example} \label{example_1}
Thanks to the previous theorem we can find a lower bound   for the dimension of the vector space generated by $\mathfrak{B}_{4} $ by looking at Table~\ref{tab:kpz_relevant}. We want to check that  the family  $ \tilde{\mathfrak{B}}_4 =   \left \lbrace	\Nabla_{\<generic>} 	\Nabla_{\<generic>} 	\Nabla_{\<generic>}  \<generic>, \, 	\Nabla_{\<generic>} 	\Nabla_{\Nabla_{\<generic>}  \<generic> }   \<generic>, \,
\Nabla_{\Nabla_{\Nabla_{\<generic>} \<generic>}   \<generic> }   \<generic> \right \rbrace $ is linearly independent. By projecting the terms with $\pi$ we obtain
\begin{equs}
	\pi \Nabla_{\<generic>} 	\Nabla_{\<generic>} 	\Nabla_{\<generic>}  \<generic> & = 
\pi	\Nabla_{\<generic>} 	\Nabla_{\<generic>}  \left( z_{ (\<generic>,0)} z_{ (\<generic>,1)}  
	\right)
	\\ &= \pi	\Nabla_{\<generic>} 	 \left( z_{ (\<generic>,0)}^2 z_{ (\<generic>,2)}  +  z_{ (\<generic>,0)} z_{ (\<generic>,1)}^2 \right) 
	\\ &= z_{ (\<generic>,0)}^3 z_{ (\<generic>,3)} + 4 z_{ (\<generic>,0)}^2 z_{ (\<generic>,1)} z_{ (\<generic>,2)} +  z_{ (\<generic>,0)} z_{ (\<generic>,1)}^3  \,, 
	\\
	\pi\Nabla_{\Nabla_{\Nabla_{\<generic>} \<generic>}   \<generic> }    \<generic> = 	z_{ (\<generic>,1)} \pi \Nabla_{\Nabla_{\<generic>} \<generic>}   \<generic>  & =   z_{ (\<generic>,1)}^2 \pi \Nabla_{\<generic>} \<generic>   = z_{ (\<generic>,0)}  z_{ (\<generic>,1)}^3\,,
\\
\pi	\Nabla_{\<generic>} 	\Nabla_{\Nabla_{\<generic>}  \<generic> }   \<generic> = \pi	\Nabla_{\<generic>} \left( z_{ (\<generic>,0)}  z_{ (\<generic>,1)}^2 \right) & = 2 z_{ (\<generic>,0)}^2 z_{ (\<generic>,1)} z_{ (\<generic>,2)}  + z_{ (\<generic>,0)}  z_{ (\<generic>,1)}^3\,.
\end{equs}
Due to the triangular structure in the terms $z_{ (\<generic>,0)}^3 z_{ (\<generic>,3)}, z_{ (\<generic>,0)}^2 z_{ (\<generic>,1)} z_{ (\<generic>,2)}, z_{ (\<generic>,0)}  z_{ (\<generic>,1)}^3$  we can conclude that $ \tilde{\mathfrak{B}}_4 $ is linearly independent and $ \dim( \langle \mathfrak{B}_4 \rangle ) \geq 3 $.
\end{example}
From these considerations we immediately obtain $\dim(V_{\geo}^N) \geq \mathrm{Card}(\mathfrak{M}_{\<generic>}^N)$. However, combining Example~\ref{example_1}  with the relations \eqref{eq:extra_relations} one obtains indeed $ \dim( \langle \mathfrak{B}_4 \rangle ) =  3= \mathrm{Card}(\mathfrak{M}_{\<generic>}^4)$. One of our main results is that this phenomenon is indeed optimal and we obtain a full description of $V_{\geo}^N$. 
\begin{theorem}  \label{dim_geo}
	For any $N\geq 2$ the linear vector space generated by $\mathfrak{B}_N$ coincides with	$V_{\geo}^N$ and $\dim(V_{\geo}^N) = \mathrm{Card}(\mathfrak{M}_{\<generic>}^N)$. Moreover the sets $\mathfrak{B}_{\xi} $ , $\mathfrak{B}_{\xi}^{\text{G}}$ are generating sets for $V_{\geo}^{\xi}$ and $V_{\geo}^{\xi, \text{G}}$ whose dimensions are given respectively by
	\[ \sum_{i=2}^{N_\xi}\mathrm{Card}(\mathfrak{M}_{\<generic>}^i)\,, \quad \sum_{i=1}^{\lfloor N_\xi/2\rfloor}\mathrm{Card}(\mathfrak{M}_{\<generic>}^{2i})\,.\]
\end{theorem}
The proof of this result follows trivially from the following Proposition  to find an upper bound on the dimension of $V_{\geo}^N$.

\begin{proposition} \label{upper_bound_dimension}
For any $N\geq 2$ one has $\dim(V_{\geo}^N) \leq \mathrm{Card}(\mathfrak{M}_{\<generic>}^N)$.
\end{proposition}
\begin{proof}
	Let $  v \in V_{\geo}^N $. Due to the fact that $ v $ is geometric, one has $	\hat{\phi}_\geo(v) = 0$ and for any $z^{\tilde{\beta}}\in \mathfrak{M}_{\begin{tikzpicture}[scale=0.2,baseline=-2]
		\coordinate (root) at (0,0);
		\node[diff] (rootnode) at (root) {};
	\end{tikzpicture}, \<derivative2>, \<generic>}^N$ we have $ 
	\langle z^{\tilde{\beta}}, \hat{\phi}_\geo(v)\rangle =0$. Writing then $v$ in coordinates we obtain
\begin{equs} \label{main_equation_system}
	\sum_{z^{\beta} \in \, \mathfrak{M}_{ \<derivative2>, \<generic>}^N} \langle z^{\beta} , v  \rangle \langle z^{\tilde{\beta}}, \hat{\phi}_\geo(z^{\beta})  \rangle = 0\,,
\end{equs}
where we keep only the terms such that $\langle z^{\tilde{\beta}}, \hat{\phi}_\geo(z^{\beta})  \rangle \neq 0$. Suppose now that $ z^{\tilde{\beta}} $ is of the form $	z^{\tilde{\beta}} =  z_{ (\begin{tikzpicture}[scale=0.2,baseline=-2]
			\coordinate (root) at (0,0);
			\node[diff] (rootnode) at (root) {};
		\end{tikzpicture},k +2)} z^{\eta}$
with $k\geq 0$ and $z^{\eta}$ not containing any other variable of the form $z_{ (\begin{tikzpicture}[scale=0.2,baseline=-2]
			\coordinate (root) at (0,0);
			\node[diff] (rootnode) at (root) {};
		\end{tikzpicture},l)}$. We will refer to  $z^{\eta}$ as a cofactor. By definition of $\hat{\phi}_\geo$, $ z^{\beta} $ has to be of one of the following ``cofactor decompositions" (we list them using different notations on  $\beta$):
		\begin{itemize}
		\item $z^{\beta_0} = z_{( \<derivative2>,k)} z^{\eta} $;
		\item  $ z^{\hat{\beta}_0} = z_{ (\<generic>,k+1)} z^{\hat{\eta}_0}$ with  $ z_{ (\<generic>,0)} z^{\hat{\eta}_0} = z^{\eta} $;
		\item  for some $ m >0 $ one has
		$
 z^{\beta_m} = z_{( \<derivative2>,k+m)} z^{\eta_m}$ with    
$ z_{( \<derivative2>,m-1)} z^{\eta_m}   =   z^{\eta}$;
	\item  for some $ m >1 $ one has
		$
  z^{\hat{\beta}_m} = z_{( \<generic>,k+m)}  z^{\hat{\eta}_m}$ with    
$ z_{ (\<generic>,m-1)} z^{\hat{\eta}_m}   =   z^{\eta}$,\end{itemize}
where of course we do not exclude that the remainders $z^{\eta}$, $z^{\hat{\eta}_0}$, $z^{\eta_m}$, $z^{\hat{\eta}_m}$
might contain several repetitions of the monomials 
\[z_{( \<derivative2>,k)}\,, \quad z_{ (\<generic>,k+1)}\,,\quad z_{( \<derivative2>,k+m)}\,, \quad z_{( \<generic>,k+m)}\,,\]
whose number is included in the corresponding multi-index. We also remark that the condition  $z^{\beta_0} = z_{( \<derivative2>,k)} z^{\eta}$ might be satisfied by several $z^{\beta_0}\in \mathfrak{M}_{ \<derivative2>, \<generic>}^N$.  Then, one can see that multi-indices containing a factor $z_{( \<derivative2>,k)}$  are connected with multi-indices of the form $z_{( \<derivative2>,k')}$ with a stricly smaller index $k'$. Indeed, from the definition  of $ \hat{\phi}_\geo $  one has 
\begin{equs}
	\hat{\phi}_\geo(z^{\beta_0} ) & = \beta_0(\<derivative2>,k) \ \phi_\geo( z_{( \<derivative2>,k)} )  z^{\eta}   + (\cdots)
	= -2 \beta_0(\<derivative2>,k) \ z_{ (\begin{tikzpicture}[scale=0.2,baseline=-2]
			\coordinate (root) at (0,0);
			\node[diff] (rootnode) at (root) {};
		\end{tikzpicture},k+2)}
	z^{\eta}   + (\cdots)
	\\ \hat{\phi}_\geo(z^{\hat{\beta}_0}) 
	&  = \hat{\beta}_0(\<generic>,k+1) \phi_\geo(z_{(\<generic>,k+1)})  z^{\hat{\eta}_0}  + (\cdots)
	\\ & =  \hat{\beta}_0(\<generic>,k+1)  z_{ (\begin{tikzpicture}[scale=0.2,baseline=-2]
			\coordinate (root) at (0,0);
			\node[diff] (rootnode) at (root) {};
		\end{tikzpicture},k+2)} z_{(\<generic>,0)}
	z^{\hat{\eta}_0}   + (\cdots)
	\\  & =  \hat{\beta}_0(\<generic>,k+1)  z_{ (\begin{tikzpicture}[scale=0.2,baseline=-2]
			\coordinate (root) at (0,0);
			\node[diff] (rootnode) at (root) {};
		\end{tikzpicture},k+2)} 
	z^{\eta}   + (\cdots)
	\end{equs}
In both computations, the symbol $(\dots)$ denotes extra terms that will not be equal to  $ z_{ (\begin{tikzpicture}[scale=0.2,baseline=-2]
		\coordinate (root) at (0,0);
		\node[diff] (rootnode) at (root) {};
	\end{tikzpicture},k+2)} z^{\eta}$ and the coefficients  $\hat{\beta}_0(\<generic>,k+1)$. The computations for $ z_{( \<derivative2>,k+m)} $ and $ z_{( \<generic>,k+m)} $  are a bit more involved and  they include \eqref{leibniz_calculus} obtaining
	\begin{equs}
		\hat{\phi}_\geo(z^{\beta_m}) 
	&  = \beta_m(\<derivative2>,k+m) \phi_\geo(z_{( \<derivative2>,k+m)} ) z^{\eta_m} + (\cdots)
	\\ & =  -\beta_m(\<derivative2>,k+m) \binom{k+m+1}{k+2}  z_{ (\begin{tikzpicture}[scale=0.2,baseline=-2]
			\coordinate (root) at (0,0);
			\node[diff] (rootnode) at (root) {};
		\end{tikzpicture},k+2)} z_{( \<derivative2>,m-1)} z^{\eta_m}
	   + (\cdots)
	   	\\ & =  -\beta_m(\<derivative2>,k+m) \binom{k+m+1}{k+2}  z_{ (\begin{tikzpicture}[scale=0.2,baseline=-2]
	   		\coordinate (root) at (0,0);
	   		\node[diff] (rootnode) at (root) {};
	   	\end{tikzpicture},k+2)}  z^{\eta}
	   + (\cdots)
\\	   \hat{\phi}_\geo(z^{\hat{\beta}_m} ) 
	   &  = \hat{\beta}_m(\<generic>,k+m) \phi_\geo(z_{(\<generic>,k+m)})  z^{\hat{\eta}}  + (\cdots)
	   \\ & =  \hat{\beta}_m(\<generic>,k+m) \frac{(k+m)!(k-m+3)}{(m-1)!(k+2)!}  z_{ (\begin{tikzpicture}[scale=0.2,baseline=-2]
	   		\coordinate (root) at (0,0);
	   		\node[diff] (rootnode) at (root) {};
	   	\end{tikzpicture},k+2)} z_{ (\<generic>,m-1)} z^{\hat{\eta}_m} \\ &   =   \hat{\beta}_m(\<generic>,k+m) \frac{(k+m)!(k-m+3)}{(m-1)!(k+2)!} z_{ (\begin{tikzpicture}[scale=0.2,baseline=-2]
	   		\coordinate (root) at (0,0);
	   		\node[diff] (rootnode) at (root) {};
	   	\end{tikzpicture},k+2)} 
	   z^{\eta}   + (\cdots)
\end{equs}
Combining the previous computations and testing  \eqref{main_equation_system} with $	z^{\tilde{\beta}} =  z_{ (\begin{tikzpicture}[scale=0.2,baseline=-2]
			\coordinate (root) at (0,0);
			\node[diff] (rootnode) at (root) {};
		\end{tikzpicture},k +2)} z^{\eta}$ one gets for any $z^{\eta}$ the identity
\begin{equs} \label{main_system_equ}
	\begin{aligned}
& -2\sum_{z^{\beta_0} = z_{( \<derivative2>,k)} z^{\eta} }	\beta_0(\<derivative2>,k)\,  \langle z^{\beta_0}, v  \rangle \\&- \sum_{m>0}  \beta_m(\<derivative2>,k+m)\binom{k+m+1}{k+2}  \, \langle z^{\beta_m}, v \rangle \\ &  +  \hat{\beta}_0(\<generic>,k+1)\, \langle z^{\hat{\beta}_0}, v \rangle \\& +\sum_{m>1} \hat{\beta}_m(\<generic>,k+m) \frac{(k+m)!(k-m+3)}{(m-1)!(k+2)!}\langle z^{\hat{\beta}_m} , v \rangle  = 0,
\end{aligned}
\end{equs}
where in the first sum we contain  all possible repetitions in the cofactor decompositions inside $\mathfrak{M}_{ \<derivative2>, \<generic>}^N$. The sums over $m$ are  finite because of the constraint on $N$  (i.e. $\beta_m(\<derivative2>, \ell) = 0$ for all $ \ell > N-2 $ and $\hat{\beta}_m(\<generic>, \ell)=0 $ for all $ \ell > N $). 

By writing down the identity \eqref{main_system_equ} starting from  $z^{\tilde{\beta}}=z_{ (\begin{tikzpicture}[scale=0.2,baseline=-2]
			\coordinate (root) at (0,0);
			\node[diff] (rootnode) at (root) {};
		\end{tikzpicture},N)}z^{\eta} $ down to  the term $z_{ (\begin{tikzpicture}[scale=0.2,baseline=-2]
			\coordinate (root) at (0,0);
			\node[diff] (rootnode) at (root) {};
		\end{tikzpicture},2)} z^{\eta}$ with all possible cofactors $z^{\eta}$, one obtains a system of linear equations in the coefficients $\langle z^{\beta}, v\rangle$. The structure of the identity shows that the terms $z^{\beta}$ containing monomials of the form $z_{(\<derivative2>,\ell)}$, $\ell=0,\ldots,N-2$, can be expressed in a triangular way in terms of monomials belonging only to $\mathfrak{M}^N_{\<generic>}$. 
		
		To make this precise, for any multi-index $z^\beta$ define
		\[ [[\beta]]= \max\{ l \colon \beta(\<derivative2>,l)\neq 0\}\]
	with the convention that $[[\beta]]=-\infty$ if $\beta(\<derivative2>,\ell)=0$ for all $\ell$. Thus $[[\beta]]$ records the highest derivative level appearing with non–zero multiplicity. We order the monomials in $\mathfrak{M}^N_{\<derivative2>,\<generic>}$ by first decreasing $[[\beta]]$, then by decreasing total degree when $[[\beta]]$ coincides, and arbitrarily among the remaining ones. With this ordering, we claim that the resulting linear system is triangular. 
	
	Indeed, fix a monomial $z^{\beta}$ with $[[\beta]]=k\ge 0$ and $ \beta(\<derivative2>,k)=m$. By definition $z^{\beta}$ can be uniquely written in the form
	\[z^{\beta}=z_{(\<derivative2>,k)}^mz^{\theta}\,,\]
	where $z^{\theta}$ contains no factor $z_{(\<derivative2>,k)}$. Consider now the equation \eqref{main_system_equ} when $z^{\tilde{\beta}}=z_{ (\begin{tikzpicture}[scale=0.2,baseline=-2]
			\coordinate (root) at (0,0);
			\node[diff] (rootnode) at (root) {};
		\end{tikzpicture},k+2)}z^{\theta} $. Then the coefficient in front of $\langle z^{\beta}, v\rangle$
appears in this identity with a non-zero numerical factor depending on  $m$
Moreover, every other coefficient appearing in the same equation corresponds either to a monomial $z^{\gamma}$ where $[[\gamma]]>k$  or $ [[\gamma]]=k$ and higher degree. Therefore, with respect to the chosen ordering, $\langle z^{\beta}, v\rangle$ is the leading unknown in that equation. Since we write the equations starting from $k = N-2$ down to $k = 0$, and within each level we eliminate monomials in decreasing multiplicity, the system can be solved inductively. Every coefficient corresponding to a monomial containing derivative variables is uniquely determined in terms of coefficients corresponding to monomials with $[[\beta]] = -\infty$, i.e.\ elements of $\mathfrak{M}^{N}_{\<generic>}$.
	
	It remains to check that every monomial of $ z^{\beta} \in \mathfrak{M}_{ \<derivative2>, \<generic>}^N $ appears in at least one equation. First, by construction, the system contains all multi-indices with at least one variable of type $ z_{(\<derivative2>,\ell)} $. For the elements of  $ \mathfrak{M}^{ N}_{\<generic>} $, one can notice that by picking the right cofactor  $ z^{\eta}$ not containing any variables of derivative type,  using the different decompositions of the form $ z^{\hat{\beta}_0} = z_{ (\<generic>,k+1)} z^{\hat{\eta}_0}$ and $
  z^{\hat{\beta}_m} = z_{( \<generic>,k+m)}  z^{\hat{\eta}_m}$, one runs over all elements of $ \mathfrak{M}^{ N}_{\<generic>} $.

In the end, one obtains a triangular system for the elements of $ \mathfrak{M}_{\<derivative2>, \<generic>}^N $ that terminates with the variables given by $\mathfrak{M}^{N}_{\<generic>}$, whose cardinality provides an upper bound on the dimension of $V^N_{\geo}$.
\end{proof}

\begin{remark} The proof of Proposition~\ref{upper_bound_dimension} is inspired by a proof from Franck Gabriel  that was part of a preliminary draft of \cite{BGHZ}. The goal of it  was to give an upper bound on the dimension of geometric elements for decorated trees. It is more involved and not so general as one has to determine by hand different orbits. Indeed, one has 
	\begin{equs}
		\hat{\phi}_\geo \left( \begin{tikzpicture}[scale=0.4,baseline=-2]
			\coordinate (root) at (0,-0.4);
			\coordinate (t1) at (-0.3,0.5);
			\coordinate (t2) at (-1.1,0.5);
			\coordinate (t3) at (1.1,0.5);
			\coordinate (t4) at (0.3,0.5);
			\coordinate (tau1) at (1,0.6);
			\draw[kernels2] (t1) -- (root);
			\draw[kernels2] (t2) -- (root);
			\draw[symbols] (t3) -- (root);
				\draw[symbols] (t4) -- (root);
			\node[not] (rootnode) at (root) {};
			\draw (-1.2,0.7) node[] {\tiny$\tau_1$};
			\node[not] (rootnode) at (root) {};
			\draw (1.3,0.7) node[] {\tiny$\tau_4$};
			\draw (-0.4,0.7) node[] {\tiny$\tau_{\tiny{2}}$};
			\draw (0.4,0.7) node[] {\tiny$\tau_3$};
		\end{tikzpicture}  \right)& = - 2 \, \begin{tikzpicture}[scale=0.4,baseline=-2]
		\coordinate (root) at (0,-0.4);
		\coordinate (t1) at (-0.3,0.5);
		\coordinate (t2) at (-1.1,0.5);
		\coordinate (t3) at (1.1,0.5);
		\coordinate (t4) at (0.3,0.5);
		\coordinate (tau1) at (1,0.6);
		\draw[symbols] (t1) -- (root);
		\draw[symbols] (t2) -- (root);
		\draw[symbols] (t3) -- (root);
		\draw[symbols] (t4) -- (root);
		\node[not] (rootnode) at (root) {};
		\draw (-1.2,0.7) node[] {\tiny$\tau_1$};
		\node[diff] (rootnode) at (root) {};
		\draw (1.3,0.7) node[] {\tiny$\tau_4$};
		\draw (-0.4,0.7) node[] {\tiny$\tau_{\tiny{2}}$};
		\draw (0.4,0.7) node[] {\tiny$\tau_3$};
	\end{tikzpicture} + (\cdots)
\\
\hat{\phi}_\geo \left( \begin{tikzpicture}[scale=0.4,baseline=-2]
	\coordinate (root) at (0,-0.4);
	\coordinate (t1) at (-0.3,0.5);
	\coordinate (t2) at (-1.1,0.5);
	\coordinate (t3) at (1.1,0.5);
	\coordinate (t4) at (0.3,0.5);
	\coordinate (tau1) at (1,0.6);
	\draw[kernels2] (t1) -- (root);
	\draw[kernels2] (t2) -- (root);
	\draw[symbols] (t3) -- (root);
	\draw[symbols] (t4) -- (root);
	\node[not] (rootnode) at (root) {};
	\draw (-1.2,0.7) node[] {\tiny$\tau_1$};
	\node[not] (rootnode) at (root) {};
	\draw (1.3,0.7) node[] {\tiny$\tau_4$};
	\draw (-0.4,0.7) node[] {\tiny$\tau_{\tiny{3}}$};
	\draw (0.4,0.7) node[] {\tiny$\tau_2$};
\end{tikzpicture}  \right) & = - 2 \, \begin{tikzpicture}[scale=0.4,baseline=-2]
	\coordinate (root) at (0,-0.4);
	\coordinate (t1) at (-0.3,0.5);
	\coordinate (t2) at (-1.1,0.5);
	\coordinate (t3) at (1.1,0.5);
	\coordinate (t4) at (0.3,0.5);
	\coordinate (tau1) at (1,0.6);
	\draw[symbols] (t1) -- (root);
	\draw[symbols] (t2) -- (root);
	\draw[symbols] (t3) -- (root);
	\draw[symbols] (t4) -- (root);
	\node[not] (rootnode) at (root) {};
	\draw (-1.2,0.7) node[] {\tiny $\tau_1$};
	\node[diff] (rootnode) at (root) {};
	\draw (1.3,0.7) node[] {\tiny $\tau_4$};
	\draw (-0.4,0.7) node[] {\tiny $\tau_{2}$};
	\draw (0.4,0.7) node[] {\tiny $\tau_3$};
\end{tikzpicture} + (\cdots).
	\end{equs}
	where the two formulae above are different. Indeed, the difference happens in the $(\cdots)$ that do not contain the same terms. 
	To perform this computation, one uses the definiton of $ \hat{\phi}_\geo $ on decorated given in \cite[Sec. 6.1]{BGHZ}.
If the two trees $\tau_2$ and $\tau_3$ are different then these trees produce two variables in the system with the same number of thick edges. This is where one has to determine the orbits. Then, one can understand why powerful algebraic tools such as  operad and homological algebra from \cite{BD24} are needed to compute the dimension of these spaces.
\end{remark}

\begin{example} \label{example_2}
We illustrate the previous lemma by writing up the system in the case of a generic element $v\in V_{\geo}^4$ using Table~\ref{tab:kpz_relevant}. We want to compute the equation \eqref{main_equation_system} when $z^{\tilde{\beta}}$ has the form 
\[
	 z_{ (\begin{tikzpicture}[scale=0.2,baseline=-2]
			\coordinate (root) at (0,0);
			\node[diff] (rootnode) at (root) {};
		\end{tikzpicture},4)} z^{\eta_4}\,,\quad z_{ (\begin{tikzpicture}[scale=0.2,baseline=-2]
		\coordinate (root) at (0,0);
		\node[diff] (rootnode) at (root) {};
	\end{tikzpicture},3)} z^{\eta_3}\,, \quad   z_{ (\begin{tikzpicture}[scale=0.2,baseline=-2]
	\coordinate (root) at (0,0);
	\node[diff] (rootnode) at (root) {};
\end{tikzpicture},2)}z^{\eta_2}\]
where $ z^{\eta_i} $ does not contain any variable of the form  $  z_{ (\begin{tikzpicture}[scale=0.2,baseline=-2]
		\coordinate (root) at (0,0);
		\node[diff] (rootnode) at (root) {};
	\end{tikzpicture},l)} $ for $l\geq 0$.  For $ z_{ (\begin{tikzpicture}[scale=0.2,baseline=-2]
		\coordinate (root) at (0,0);
		\node[diff] (rootnode) at (root) {};
	\end{tikzpicture},4)} z^{\eta_4}$   the associated multi-indices $z^{\beta}\in \mathfrak{M}_{\<generic>, \<derivative2>}^4$ such that  $ \langle z_{ (\begin{tikzpicture}[scale=0.2,baseline=-2]
		\coordinate (root) at (0,0);
		\node[diff] (rootnode) at (root) {};
	\end{tikzpicture},4)} z^{\eta_4}, \hat{\phi}_\geo(z^{\beta})  \rangle \neq 0$ are given by 
\begin{equs}
z^{\beta_0} =   z_{( \<derivative2>,2)}   z_{(\<generic>,0)}^4 =  z_{( \<derivative2>,2)}   z^{\eta}\,, \quad    z^{\hat{\beta}_0} = z_{(\<generic>,0)}^3z_{(\<generic>,3)} = z_{(\<generic>,3)} z^{\hat{\eta}_0},
\end{equs}
which correspond to the unique cofactor term $ z_{(\<generic>,0)}^4 = z^{\eta_4}$. From this one has the first relation
\begin{equs}
 -2	 \,  \langle z_{( \<derivative2>,2)}   z_{(\<generic>,0)}^4, v  \rangle  +  \, \langle z_{(\<generic>,0)}^3z_{(\<generic>,3)}, v \rangle = 0.
\end{equs}
Passing to $ z_{ (\begin{tikzpicture}[scale=0.2,baseline=-2]
		\coordinate (root) at (0,0);
		\node[diff] (rootnode) at (root) {};
	\end{tikzpicture},3)} z^{\eta_3}$, one has two possible choices for the cofactor $z^{\eta_3}=z_{ (\<generic>,0)}^4z_{( \<derivative2>,0)} $,  $z^{\eta_3}= z_{ (\<generic>,0)}^3 z_{ (\<generic>,1)}$ which imply the two conditions 
\begin{equs}
-2 \	\langle z_{ (\<generic>,0)}^4z_{( \<derivative2>,0)}z_{( \<derivative2>,1)} , v \rangle - \langle z_{( \<derivative2>,2)}   z_{(\<generic>,0)}^4, v\rangle + \langle z_{ (\<generic>,0)}^3z_{ (\<generic>,1)}z_{( \<derivative2>,0)}, v\rangle &= 0\,,
\\
-2 \	\langle  z_{ (\<generic>,0)}^3 z_{ (\<generic>,1)}z_{( \<derivative2>,1)}, v \rangle +  2\langle z_{ (\<generic>,0)}^3z_{ (\<generic>,3)} ,v \rangle + \langle z_{ (\<generic>,0)}^2z_{ (\<generic>,1)}z_{ (\<generic>,2)}, v \rangle  &= 0\, .
\end{equs}
In the last case $ z_{ (\begin{tikzpicture}[scale=0.2,baseline=-2]
		\coordinate (root) at (0,0);
		\node[diff] (rootnode) at (root) {};
	\end{tikzpicture},2)} z^{\eta_2} $, one has four possible choices for the cofactor $z^{\eta_2}=  z_{ (\<generic>,0)}^4$,   $z^{\eta_2}=z_{(\<generic>,0)}^2 z_{(\<generic>,1)}^2$,  $z^{\eta_2}= z_{ (\<generic>,0)}^3z_{ (\<generic>,1)}$,  $z^{\eta_2}= z_{ (\<generic>,0)}^4z_{(\<derivative2>,1)}$,  which imply the four  conditions 
\begin{equs}
&	-6 \	\langle z_{( \<derivative2>,0)}^3   z_{(\<generic>,0)}^4, v \rangle + \ \langle z_{(\<generic>,0)}^3 z_{(\<generic>,1)}, v \rangle  = 0\,,
\\&
	-2\	\langle z_{( \<derivative2>,0)}   z_{(\<generic>,0)}^2 z_{(\<generic>,1)}^2, v \rangle + 3\ \langle z_{(\<generic>,0)} z_{(\<generic>,1)}^3, v \rangle +\ \langle z_{(\<generic>,0)} z_{(\<generic>,1)}z_{(\<generic>,2)}, v \rangle = 0\,,\\&
	-2\langle z_{ (\<generic>,0)}^3z_{ (\<generic>,1)}z_{( \<derivative2>,0)}, v\rangle -4 \	\langle z_{( \<derivative2>,0)}^2   z_{(\<generic>,0)}^3 z_{(\<generic>,1)}, v \rangle + 2\ \langle z_{(\<generic>,0)}^2 z_{(\<generic>,1)}^2, v \rangle =0\,,\\&
	- 2\	\langle z_{ (\<generic>,0)}^4z_{(\<derivative2>,1)}z_{(\<derivative2>,0)} v \rangle + \langle z_{(\<generic>,0)}^3 z_{(\<generic>,1)}z_{(\<derivative2>,1)},v \rangle - 3\langle z_{( \<derivative2>,2)}   z_{(\<generic>,0)}^4,v \rangle = 0\,.
\end{equs}
Combining all of these linear relationship, we obtain a linear system reducing the dimension $10$ of $ \mathcal{M}_{ \<derivative2>, \<generic>}^4$ to a $3$-dimensional subspace and one has $3=\mathrm{Card}(\mathfrak{M}_{\<generic>}^4)$.
\end{example}

\subsection{Itô and nice terms}

We conclude the section by studying the two last vector spaces  $V_{\Ito}$ and $V_{\nice}$. In our case, the explicit form of the evaluation map $\Upsilon_{\Gamma, \sigma}$ and its multiplicative property of multi-indices allow to describe these spaces with some elementary arguments. Thanks to the simple condition $\sigma^2 = \bar{\sigma}^2$, we first show that $V_{\Ito}$ has a trivial structure and it coincides with the full space of multi-indices when the noise is Gaussian.

\begin{proposition} \label{thm_Ito}
One has $V_{\Ito}=\mathcal{M}_{ \<derivative2>, \<generic>}^{\text{G}}$.
\end{proposition}
\begin{proof}
It is sufficient to show the inclusion $\mathcal{M}_{ \<derivative2>, \<generic>}^{\text{G}}\subset V_{\Ito}$. Let $\beta\in \mathfrak{M}_{ \<derivative2>, \<generic>} ^{\text{G}}$ and  $\sigma$, $\bar{\sigma}$ two functions satisfying  the equality 
$ \sigma^2=\bar{\sigma}^2$.  The quadratic constraint and the smoothness of  $\sigma$, $\bar{\sigma}$ imply the identity $ \sigma(u)=(\pm 1)\bar{\sigma}(u)$ for any $u\in \mathbb{R}$  and  similarly $\sigma^{(k)}(u)=(\pm 1)\bar{\sigma}^{(k)}(u) $ for any  $k\geq 1$ with the same sign as $\sigma$ and $\bar{\sigma}$. Plugging this identity in the definition \eqref{def:ev_map}, we immediately obtain the equality
\begin{equation}\label{eq:proof_ito}
\Upsilon_{ \Gamma, \sigma}(z^{\beta})(u)= (\pm 1)^{((\beta))}\Upsilon_{ \Gamma, \bar{\sigma}}(z^{\beta})(u)\,,
\end{equation}
from which we obtain the equality since $((\beta))$ is even. 
\end{proof}
We can use similar arguments for $V_{\nice}$.

\begin{proposition} \label{thm_nice}
One has that $V_{\nice}$ is generated by all the multi-indices containing $z_{(\<generic>,1)}$ or $z_{(\<derivative2>,0)}$.
\end{proposition}
\begin{proof}
The proof follows immediately for the Definition in \eqref{def:ev_map}.
\end{proof}

\begin{example} \label{example_3}
Combining the results in Proposition~\ref{thm_Ito}, Proposition~\ref{thm_nice}, and Theorem~\ref{dim_geo}, we obtain a full description of these vector spaces and their intersections in the case when $\xi$ is a space-time white noise. In this case, one has
\[ V_{\geo} = V_{\geo}\cap V_{\Ito}=\langle \Nabla_{\<generic>}  \<generic>\,,\Nabla_{\<generic>} 	\Nabla_{\<generic>} 	\Nabla_{\<generic>}  \<generic>, \, 	\Nabla_{\<generic>} 	\Nabla_{\Nabla_{\<generic>}  \<generic> }   \<generic>, \,
\Nabla_{\Nabla_{\Nabla_{\<generic>} \<generic>}   \<generic> }   \<generic>\rangle\,.\]
Moreover, by simply checking the elements in Table~\ref{tab:kpz_relevant} with $((\beta))=2,4$, one finds that $V_{\nice}$ contains all the multi-indices with $((\beta))=2,4$, with the exceptions of $z_{ (\<generic>,0)}^3z_{ (\<generic>,3)}$,  and $z_{ (\<generic>,0)}^4z_{( \<derivative2>,2)}$, leading to a vector space of dimension $10$. We can also directly compute 
\[V_{\geo} \cap V_{\nice}=\langle \Nabla_{\<generic>}  \<generic>\,, \, 	\Nabla_{\<generic>} 	\Nabla_{\Nabla_{\<generic>}  \<generic> }   \<generic>, \,
\Nabla_{\Nabla_{\Nabla_{\<generic>} \<generic>}   \<generic> }   \<generic>\rangle\,,\]
by simply remarking that each term $\Nabla_{\<generic>}  \<generic>\,, \, 	\Nabla_{\<generic>} 	\Nabla_{\Nabla_{\<generic>}  \<generic> }   \<generic>, \,
\Nabla_{\Nabla_{\Nabla_{\<generic>} \<generic>}   \<generic> }   \<generic>$ are linear combination of elements in  $V_{\nice}$ and that one has 
\[\Nabla_{\<generic>} 	\Nabla_{\<generic>} 	\Nabla_{\<generic>}  \<generic>= z_{ (\<generic>,0)}^3z_{ (\<generic>,3)}+ \frac{1}{2}z_{ (\<generic>,0)}^4z_{( \<derivative2>,2)}+w\]
with $w\in V_{\nice}$. Therefore $\Nabla_{\<generic>} 	\Nabla_{\<generic>} 	\Nabla_{\<generic>}  \<generic>\not \in V_{\nice}$. Both results differ from the analogous computations obtained with trees in \cite[Proposition 3.18, Corollary 3.20]{BGHZ}. In that context, the presence of a non-trivial It\^o vector space tightens the dimension, leading to an intersection of their geometric and It\^o spaces of dimension $2$ and an intersection with $V_{\nice}$ of dimension one. In our case, even if the presence of $V_{\nice}$ also reduces the dimension by one order, we are still far from achieving the same degree of dimensional reduction. We also expect that similar considerations holds when the noise $\xi$ is just subcritical.
\end{example}
\section{The renormalised equation and its convergence}\label{sec:4}
We finally transfer the optimal choice of the constants to the level of the renormalised equation by including the new non-linearities. Indeed, by plugging  a generic constant  $c\in (\mathcal{T}_-)^* $ into the renormalised model, the solution ansatz on multi-indices, given in \cite{BL23,LOT} produces the following renormalised equation
\begin{equs}[eq:renorm_constant]
\d_t u_{\eps}  &= \d_x^2 u_{\eps} + \Gamma(u_{\eps})\,(\d_x u_{\eps})^2
		+ g(u_{\eps})\,\d_x u_{\eps}
		+h(u_{\eps}) + \sigma(u_{\eps})\, \xi_{\eps}\\&+ \sum_{\beta\in T_-} \frac{c_{\beta}}{\sigma(\beta)}\Upsilon_{\Gamma,\sigma}(z^{\beta})(u_{\eps},\d_x u_{\eps})
\end{equs}
with $\sigma(\beta)$ a \textbf{symmetry factor}, depending only on the values of $\beta$ over $\mathcal{N}$
\[\sigma(\beta)= \prod_{k\in \mathcal{N}_0}(k(\mathbf{0})!)^{\beta(0,k)} \prod_{m\in \mathcal{N}_{\Xi}}(k(\mathbf{0})!)^{\beta(\Xi,m)}\,.\]

Using  the constants
\begin{equation}\label{BPHZ_choice}
C_{\eps}(z_\beta)=\frac{c^{\BPHZ}_{\varepsilon}(\beta)}{\sigma(\beta)}
\end{equation}
 we can finally express \eqref{eq:renorm_constant} into its ``BPHZ version"
 
 \begin{equs}[eq:renorm_constant3]
 	\d_t u_{\eps}  &= \d_x^2 u_{\eps} + \Gamma(u_{\eps})\,(\d_x u_{\eps})^2
 	+ g(u_{\eps})\,\d_x u_{\eps}
 	+h(u_{\eps}) + \sigma(u_{\eps})\, \xi_{\eps}\\&+ \sum_{\beta\in \mathfrak{M}_{ \<derivative2>, \<generic>}^-} C_{\eps}(z_\beta)\Upsilon_{\Gamma,\sigma}(z^{\beta})(u_{\eps})
 \end{equs}
 
  Combining several results in the literature with the algebraic results in Section~\ref{Sec::3}, we  finish this section by proving Theorem~\ref{thm:main renormalisation_intro}.

\begin{proof}[Theorem~\ref{thm:main renormalisation_intro}]
	From \cite{BGHZ} and the results contained in \cite{reg,BHZ,CH,BCCH}, one can rewrite the renormalised equation for the set of reduced decorated trees:
		\begin{equs}[eq:renorm nonlocal intro3]
		\d_t u_{\eps} & = \d_x^2 u_{\eps} + \Gamma(u_{\eps})\,\d_x u_{\eps}\d_x u_{\eps}
		+ g(u_{\eps})\,\d_x u_{\eps}
		+h(u_{\eps}) + \sigma(u_{\eps})\, \xi_{\eps}\; \\ & + \sum_{\tau \in \mathfrak{T}_{ \<derivative2>, \<generic>}^-} C_{\eps}(\tau)  \tilde{\Upsilon}_{\Gamma,\sigma}(\tau)(u_{\eps})\,.
	\end{equs}
where $\mathfrak{T}_{ \<derivative2>, \<generic>}^- =\Psi^{-1}(\mathfrak{M}_{ \<derivative2>, \<generic>}^-)$, the map $ \tilde{\Upsilon}_{\Gamma,\sigma} $ is defined on decorated trees with the following recursion
\begin{equs}
	\tilde{\Upsilon}_{\Gamma,\sigma} \left( \begin{tikzpicture}[scale=0.4,baseline=-2]
		\coordinate (root) at (0,-0.4);
		\coordinate (t2) at (-0.8,0.5);
		\coordinate (t3) at (0.8,0.5);
		\draw[symbols] (t2) -- (root);
		\draw[symbols] (t3) -- (root);
		\draw (0,-0.4) node[] {\<generic>};
		\draw (-0.9,0.7) node[] {\tiny$\sigma_1$};
		\draw (0.9,0.7) node[] {\tiny$\sigma_m$};
		\draw (0,0.7) node[] {\tiny$\cdots$};
	\end{tikzpicture} \right)(u) & = \sigma^{(m)}(u) \prod_{i=1}^m 	\tilde{\Upsilon}_{\Gamma,\sigma} (\sigma_i), \\	\tilde{\Upsilon}_{\Gamma,\sigma} \left( \begin{tikzpicture}[scale=0.4,baseline=-2]
		\coordinate (root) at (0,-0.4);
		\coordinate (t1) at (-0.3,0.5);
		\coordinate (t2) at (-1.1,0.5);
		\coordinate (t3) at (1.1,0.5);
		\coordinate (t4) at (1.1,0.5);
		\coordinate (tau1) at (1,0.6);
		\draw[kernels2] (t1) -- (root);
		\draw[kernels2] (t2) -- (root);
		\draw[symbols] (t3) -- (root);
		\node[not] (rootnode) at (root) {};
		\draw (-1.2,0.7) node[] {\tiny$\tau_1$};
		\node[not] (rootnode) at (root) {};
		\draw (1.3,0.7) node[] {\tiny$\tau_n$};
		\draw (-0.4,0.7) node[] {\tiny$\tau_{\tiny{2}}$};
		\draw (0.5,0.7) node[] {\tiny$\cdots$};
	\end{tikzpicture} \right)(u) &= 2\Gamma^{(n-2)}(u) \prod_{i=1}^n 	\tilde{\Upsilon}_{\Gamma,\sigma} (\tau_i),
\end{equs} 
(see e.g. \cite[Sec. 3]{BGHZ}) and constants $C_{\eps}(\tau) $ are chosen 
to be the BPHZ renormalisation constant, see \cite{BHZ,CH}, then the solution map $u_{\eps}(\Gamma, \sigma,h)$ converges in probability locally in time to a proper continuous random field $ u^{\text{\tiny BPHZ}}(\Gamma, \sigma,h) $, whose law $ U^{\text{\tiny BPHZ}}(\Gamma, \sigma,h) $ is injective in $h$ from \cite[Thm. 3.5]{BGHZ}. This fact is crucial for arguing that the solution will be equivariant under the action of diffeomorphisms.
 
 The next step is to rewrite the counter-terms via multi-indices by noticing that
 \begin{equs} \label{rewrite_multi}
\tilde{\Upsilon}_{\Gamma, \sigma}( \tau ) = \Upsilon_{\Gamma, \sigma}( \Psi(\tau) )
\end{equs}
The identity \eqref{rewrite_multi} is proved by induction noticing that
\begin{equs}
	\tilde{\Upsilon}_{\Gamma,\sigma} \left( \begin{tikzpicture}[scale=0.4,baseline=-2]
		\coordinate (root) at (0,-0.4);
		\coordinate (t1) at (-0.3,0.5);
		\coordinate (t2) at (-1.1,0.5);
		\coordinate (t3) at (1.1,0.5);
		\coordinate (t4) at (1.1,0.5);
		\coordinate (tau1) at (1,0.6);
		\draw[kernels2] (t1) -- (root);
		\draw[kernels2] (t2) -- (root);
		\draw[symbols] (t3) -- (root);
		\node[not] (rootnode) at (root) {};
		\draw (-1.2,0.7) node[] {\tiny$\tau_1$};
		\node[not] (rootnode) at (root) {};
		\draw (1.3,0.7) node[] {\tiny$\tau_n$};
		\draw (-0.4,0.7) node[] {\tiny$\tau_{\tiny{2}}$};
		\draw (0.5,0.7) node[] {\tiny$\cdots$};
	\end{tikzpicture} \right) (u)&= 2\Gamma^{(n-2)}(u) \prod_{i=1}^n 	\tilde{\Upsilon}_{\Gamma,\sigma} (\tau_i)
= 2\Gamma^{(n-2)}(u) \prod_{i=1}^n 	\Upsilon_{\Gamma,\sigma} (\Psi(\tau_i))
\\ & = \Upsilon_{\Gamma,\sigma} (z_{( \<derivative2>,n+ 2)}  \prod_{i=1}^n \Psi(\tau_i))(u)
  =  	\Upsilon_{\Gamma,\sigma} \left( \Psi ( \begin{tikzpicture}[scale=0.4,baseline=-2]
	\coordinate (root) at (0,-0.4);
	\coordinate (t1) at (-0.3,0.5);
	\coordinate (t2) at (-1.1,0.5);
	\coordinate (t3) at (1.1,0.5);
	\coordinate (t4) at (1.1,0.5);
	\coordinate (tau1) at (1,0.6);
	\draw[kernels2] (t1) -- (root);
	\draw[kernels2] (t2) -- (root);
	\draw[symbols] (t3) -- (root);
	\node[not] (rootnode) at (root) {};
	\draw (-1.2,0.7) node[] {\tiny$\tau_1$};
	\node[not] (rootnode) at (root) {};
	\draw (1.3,0.7) node[] {\tiny$\tau_n$};
	\draw (-0.4,0.7) node[] {\tiny$\tau_{\tiny{2}}$};
	\draw (0.5,0.7) node[] {\tiny$\cdots$};
\end{tikzpicture} ) \right)(u)\,,
	\end{equs}
where for the first line, we have used the induction hypothesis on the $\tau_i$,  $  	\tilde{\Upsilon}_{\Gamma,\sigma}(\tau_i)
= \Upsilon_{\Gamma,\sigma}(\Psi(\tau_i)) $.
The second line is just the multiplicativity of $ \Upsilon_{\Gamma, \sigma} $ on multi-indices. Then, we use the recursive definition of the map $\Psi$ given in \eqref{def_Psi}. The computation for the decorated trees without brown edges works exactly the same. Then  one rewrites the counter-terms as
\begin{equs}
	\sum_{\tau \in \ \mathfrak{T}_{ \<derivative2>, \<generic>}^-} C_{\eps}(\tau)  \tilde{\Upsilon}_{\Gamma,\sigma}(\tau)(u_{\eps}) & =  \sum_{z^{\beta} \in \ \mathfrak{M}_{ \<derivative2>, \<generic>}^-} \sum_{z^{\beta} = \Psi(\tau)} C_{\eps}(\tau)  \tilde{\Upsilon}_{\Gamma,\sigma}(\tau)(u_{\eps})
	\\ & = \sum_{z^{\beta} \in \ \mathfrak{M}_{ \<derivative2>, \<generic>}^-} \left(  \sum_{z^{\beta} = \Psi(\tau)} C_{\eps}(\tau)  \right) \Upsilon_{\Gamma,\sigma}(z^{\beta})(u_{\eps})
\end{equs}
where one has 
\begin{equs}
	   C_{\eps}(z^{\beta}) = \sum_{z^{\beta} = \Psi(\tau)} C_{\eps}(\tau) 
\end{equs}
as a consequence for the uniqueness of the BPHZ constants in Theorem~\ref{BPHZ_theorem} and \cite[Thm 6.18]{BHZ}. Then, changing one of the constants $C_{\eps}(\tau)$ with $ z^{\beta} = \Psi(\tau) $ corresponds to change $ C_{\eps}(z^{\beta}) $.   From \cite[Prop. 3.9]{BGHZ}, one can get rid of  the counter-terms on the orthogonal complement of $  V_{\geo} $ and get a choice of renormalisation constants $ C_{\eps}(\tau) $ that guarantees the convergence of $  u_{\eps}$ and such that
\begin{equs}
\sum_{z^{\beta} \in \ \mathfrak{M}_{ \<derivative2>, \<generic>}^-} C_{\eps}(z^{\beta})  \Upsilon_{\Gamma,\sigma}(z^{\beta})(u_{\eps}) = \sum_{v \in \mathfrak{B}_{\xi}} C_{\eps}(v)  \Upsilon_{\Gamma,\sigma}(v)(u_{\eps})
\end{equs}
where the constants  $ C_{\eps}(v) $ depend on the $  C_{\eps}(z^{\beta}) $ and here $  \mathfrak{B}_{\xi}$ defined in \eqref{def_B_xi} is a basis of  $ V_{\geo}^{\xi} $ defined in \eqref{geo_space} whose dimension is given in Theorem~\ref{dim_geo}. This guarantees that $U$, the law of the limit solution  is invariant under the action of diffeomorphisms. Passing to the case when $\xi$ is a Gaussian subcritical noise, using Proposition~\ref{prop:reduction}  we can repeat all the previous arguments starting from the set $(\mathfrak{M}_{ \<derivative2>, \<generic>}^{\text{G}}) ^-$ and not $\mathfrak{M}_{ \<derivative2>, \<generic>} ^-$ at the beginning of the proof and obtain the same result with $\mathfrak{B}_{ \xi}^{\text{G}}$ instead of $\mathfrak{B}_{ \xi}$. From Proposition~\ref{thm_Ito}, we deduce that all the counter-terms $v \in \mathfrak{B}_{ \xi}^{\text{G}} $ belong trivially to $V_{\Ito}$ which implies for $ \sigma^2 = \bar{\sigma}^2 $
\begin{equs}
	(\Upsilon_{\Gamma, \bar{\sigma}} - 	\Upsilon_{\Gamma, \sigma}) v = 0.
\end{equs}
 Then, we conclude from \cite[Sec. 3.5]{BGHZ}, where one has for $\sigma^2 = \bar{\sigma}^2$
\begin{equs}
	U^{\geo}(\Gamma, \sigma,h) = 	U^{\geo}(\Gamma, \bar{\sigma},h + \left( \tilde{\Upsilon}_{\Gamma, \bar{\sigma}} - \tilde{\Upsilon}_{\Gamma, \sigma} \right)\tau_0)
\end{equs}
where $\tau_0$ is a linear combination of decorated trees and
 $ 	U^{\geo}(\Gamma, \sigma,h) $ is the law of the geometric approximation of the solution. Then, we use \eqref{rewrite_multi} to get
 \begin{equs}
 	U^{\geo}(\Gamma, \sigma,h) & = 	U^{\geo}(\Gamma, \bar{\sigma},h + \left( \Upsilon_{\Gamma, \bar{\sigma}} - \Upsilon_{\Gamma, \sigma} \right) \Psi(\tau_0))
 = U^{\geo}(\Gamma, \bar{\sigma},h)
 \end{equs}
 where we have used the fact that $  \Psi(\tau_0) \in \mathfrak{B}_{ \xi}^{\text{G}}  $ which implies
 \begin{equs}
 	\left( \Upsilon_{\Gamma, \bar{\sigma}} - \Upsilon_{\Gamma, \sigma} \right) \Psi(\tau_0) = 0.
 \end{equs}
	\end{proof}

\end{document}